\documentclass[final]{siamart171218} 

\usepackage{amsfonts}
\usepackage{graphicx}
\usepackage[export]{adjustbox}
\usepackage{epstopdf}
\usepackage{algorithmic}
\ifpdf
  \DeclareGraphicsExtensions{.eps,.pdf,.png,.jpg}
\else
  \DeclareGraphicsExtensions{.eps}
\fi
\usepackage{amssymb}
\usepackage{array}
\usepackage{comment}
\usepackage{float}
\usepackage{placeins}
\usepackage{multirow} 
\usepackage{cancel}
\usepackage{caption}
\usepackage{booktabs,subcaption,amsfonts,dcolumn}

\newcommand{\centered}[1]{\begin{tabular}{l}#1\end{tabular}}

\newcommand{\beq} {\begin{equation}}
\newcommand{\eeq} {\end{equation}}
\newcommand{\bdm} {\begin{displaymath}}
\newcommand{\edm} {\end{displaymath}}
\newcommand{\bit}{\begin{itemize}}
\newcommand{\eit}{\end{itemize}}
\newcommand{\bde}{\begin{description}}
\newcommand{\ede}{\end{description}}
\newcommand{\bce}{\begin{center}}
\newcommand{\ece}{\end{center}}
\newcommand{\ben} {\begin{enumerate}}
\newcommand{\een} {\end{enumerate}}
\newcommand{\bea} {\begin{eqnarray}}
\newcommand{\eea} {\end{eqnarray}}
\newcommand{\barr} {\begin{array}}
\newcommand{\earr} {\end{array}}
\newcommand{\bean} {\begin{eqnarray*}}
\newcommand{\eean} {\end{eqnarray*}}
\newcommand{\edoc} {\end{document}}

\newcommand{\Grad} {\ensuremath{\nabla}}

\newcommand{\Div} {\ensuremath{\nabla\cdot}}

\newcommand{\Curl} {\ensuremath{\nabla\times}}

\newcommand{\eval}[2][\right]{\relax
  \ifx#1\right\relax \left.\fi#2#1\rvert}

\newcommand{\e}{e}

\newcommand{\norm}[1]{\left\lVert#1\right\rVert}

\newcommand{\half} {\ensuremath{\frac{1}{2}}}

\newcommand{\mc}[1]{\mathcal{#1}}

\newcommand{\mb}[1]{\mathbf{#1}}

\newcommand{\nor}[1]{\left\| #1 \right\|}
\newcommand{\nort}[1]{{\left\vert\kern-0.25ex\left\vert\kern-0.25ex\left\vert #1 
    \right\vert\kern-0.25ex\right\vert\kern-0.25ex\right\vert}}

\newcommand{\LRp}[1]{\left( #1 \right)}
\newcommand{\LRs}[1]{\left[ #1 \right]}
\newcommand{\LRa}[1]{\left< #1 \right>}
\newcommand{\LRc}[1]{\left\{ #1 \right\}}
\newcommand{\pp}[2]{\frac{\partial #1}{\partial #2}}

\newcommand{\bs}[1]{\boldsymbol{#1}}

\newcommand{\jump}[1] {\ensuremath{[\![{#1}]\!]}}

\newcounter{tablecell}

\newtheorem{propo}{Proposition}

\newtheorem{rema}{Remark}

\newcommand{\figlab}[1]{\label{fig:#1}}
\newcommand{\eqnlab}[1]{\label{eq:#1}}
\newcommand{\theolab}[1]{\label{theo:#1}}

\newcommand{\propolab}[1]{\label{propo:#1}}

\newcommand{\remalab}[1]{\label{rema:#1}}
\newcommand{\tablab}[1]{\label{tab:#1}}

\newcommand{\figref}[1]{\ref{fig:#1}}
\newcommand{\theoref}[1]{\ref{theo:#1}}

\newcommand{\remaref}[1]{\ref{rema:#1}}

\newcommand{\proporef}[1]{\ref{propo:#1}}

\newcommand{\eqnref}[1]{\eqref{eq:#1}}

\newcommand{\seclab}[1]{\label{sect:#1}}
\newcommand{\secref}[1]{\ref{sect:#1}}
\newcommand{\tabref}[1]{\ref{tab:#1}}

\renewcommand{\u}{u}
\renewcommand{\v}{v}

\newcommand{\ed}{e}
\newcommand{\w}{w}
\newcommand{\wb}{\bs{\w}}
\newcommand{\J}{J}
\newcommand{\Jb}{\bs{\J}}

\newcommand{\q}{q}
\renewcommand{\r}{r}
\newcommand{\ph}{\hat{p}}
\newcommand{\rh}{\hat{\r}}

\newcommand{\ub}{\bs{\u}}

\newcommand{\vb}{\bs{\v}}

\renewcommand{\d}{{d}}
\newcommand{\db}{\bs{\d}}
\newcommand{\R}{{\mathbb{R}}}

\newcommand{\F}{\bs{F}}
\newcommand{\f}{f}
\newcommand{\g}{g}
\newcommand{\fb}{\bs{\f}}
\newcommand{\gb}{\bs{\g}}

\newcommand{\pOmega}{{\partial \Omega}}
\newcommand{\K}{K}
\newcommand{\Kp}{K^+}
\newcommand{\Km}{K^-}
\newcommand{\pK}{{\partial \K}}

\newcommand{\Gh}{{\mc{E}_h}}
\newcommand{\Gho}{{\mc{E}_h^o}}
\newcommand{\Ghb}{{\mc{E}_h^{\partial}}}
\newcommand{\Poly}{{\mc{P}}}

\newcommand{\p}{{p}}

\newcommand{\mub}{\bs{\mu}}
\newcommand{\lambdab}{\bs{\lambda}}

\renewcommand{\L}{L}
\newcommand{\Lb}{\bs{\L}}

\newcommand{\n}{\bs{n}}
\newcommand{\nm}{\n^-}
\newcommand{\np}{\n^+}

\newcommand{\ubh}{{\hat{\ub}}}

\newcommand{\Fh}{{\hat{\F}}}

\renewcommand{\c}{{c}}
\newcommand{\cb}{\bs{\c}}

\renewcommand{\H}{H}
\newcommand{\Hb}{\bs{\H}}

\newcommand{\hb}{\bs{h}}

\newcommand{\G}{G}
\newcommand{\Gb}{\bs{\G}}

\newcommand{\s}{s}

\newcommand{\Omegah}{{\Omega_h}}

\newcommand{\pd}{\partial}
\newcommand{\curl}{\Curl}

\renewcommand{\div}{\Div}

\newcommand{\bu}{\boldsymbol{u}}
\newcommand{\bb}{\boldsymbol{b}}

\newcommand{\bbh}{\hat{\bb}}

\newcommand{\lap}{\Delta}

\newcommand{\algn}[1]{\begin{align} #1 \end{align}}
\newcommand{\algns}[1]{\begin{align*} #1 \end{align*}}

\newcommand{\Ha}{\ensuremath{\textup{Ha}}}
\newcommand{\Rey}{\ensuremath{\textup{Re}}}
\newcommand{\Rm}{\ensuremath{\textup{Rm}}}

\newcommand{\Mb}{\bs{M}}

\newcommand{\LbH}{\Lb_h}
\newcommand{\ubH}{\ub_h}
\newcommand{\pH}{p_h}
\newcommand{\phH}{\ph_h}

\newcommand{\JbH}{\Jb_h}
\newcommand{\bbH}{\bb_h}
\newcommand{\rH}{r_h}
\newcommand{\ubhH}{\ubh_h}

\newcommand{\bbhH}{\hat{\bb}_h}

\newcommand{\rhH}{\rh_h}
\newcommand{\GbM}{\mb{G}_h}
\newcommand{\VbM}{\mb{V}_h}
\newcommand{\QbM}{\mb{Q}_h}
\newcommand{\HbM}{\mb{H}_h}
\newcommand{\CbM}{\mb{C}_h}
\newcommand{\SbM}{\mb{S}_h}
\newcommand{\MubM}{\mb{M}_h}
\newcommand{\LambM}{\mb{\Lambda}_h}
\newcommand{\RhoM}{\mb{P}_h}
\newcommand{\GambM}{\mb{\Gamma}_h}

\newcommand{\FbhH}{\Fh_h}
\newcommand{\pOmegah}{{\pOmega_h}}
\newcommand{\Linfty}{{L^\infty}}

\newcommand{\opN}{\bs{N}}
\newcommand{\opT}{\bs{T}}

\newcommand{\Csp}{\mc{C}}
\newcommand{\Lsp}{L}
\newcommand{\Hsp}{H}

\renewcommand{\k}{k}
\newcommand{\kbr}{\overline{\k}} 

\renewcommand{\epsilon}{\varepsilon
}

\graphicspath{{./figures/}}

\newsiamremark{remark}{Remark}
\newsiamremark{hypothesis}{Hypothesis}
\crefname{hypothesis}{Hypothesis}{Hypotheses}
\newsiamthm{claim}{Claim}

\newcommand{\TheTitle}{A Divergence-Free and $\Hsp(div)$-Conforming \\ Embedded-Hybridized DG Method for \\ the Incompressible Resistive MHD equations} 
\newcommand{\TheAuthors}{J. Chen, T.L. Horv\'{a}th, and T. Bui-Thanh}

\headers{A Div-Free and $H(div)$-Conforming EHDG Method for MHD}{\TheAuthors}

\title{{\TheTitle}\thanks{} }

\author{
Jau-Uei Chen\footnotemark[2],
Tam\'{a}s L. Horv\'{a}th\footnotemark[3],
\and Tan Bui-Thanh\footnotemark[2]\,\,\footnotemark[4]
}

\usepackage{amsopn}

\begin{document}

\maketitle

\renewcommand{\thefootnote}{\fnsymbol{footnote}}
\footnotetext[2]{Department of Aerospace Engineering and Engineering Mechanics, The University of Texas at Austin, Austin, TX 78712, USA (\email{chenju@itexas.edu}).
}
\footnotetext[3]{Department of Mathematics and Statistics, Oakland University, Rochester, MI 48309, USA (\email{thorvath@oakland.edu}).
}
\footnotetext[4]{Oden Institute for Computational Engineering and Sciences, The University of Texas at Austin, Austin, TX 78712, USA (\email{tanbui@oden.utexas.edu}).
}
\renewcommand{\thefootnote}{\arabic{footnote}}


\begin{abstract}
We present a divergence-free and $\Hsp\LRp{div}$-conforming hybridized discontinuous Galerkin (HDG) method and a computationally efficient variant called embedded-HDG (E-HDG)  for solving stationary incompressible viso-resistive magnetohydrodynamic (MHD) equations. 
The proposed E-HDG approach uses continuous facet unknowns for the vector-valued solutions (velocity and magnetic fields) while it uses discontinuous facet unknowns for the scalar variable (pressure and magnetic pressure). This choice of function spaces makes E-HDG computationally far more advantageous, due to the much smaller number of degrees of freedom, compared to the HDG counterpart. The benefit is even more significant for three-dimensional/high-order/fine mesh scenarios. On simplicial meshes, the proposed methods with a specific choice of approximation spaces are well-posed for linear(ized) MHD equations.  
For nonlinear MHD problems, we present a simple approach exploiting the proposed linear discretizations by using a Picard iteration.
The beauty of this approach is that  the divergence-free and $\Hsp\LRp{div}$-conforming properties of the velocity and magnetic fields are automatically carried over for 
nonlinear MHD equations.
We study the accuracy and convergence of our E-HDG method for both linear and nonlinear MHD cases through various numerical experiments, including two- and three-dimensional problems with smooth and singular solutions. The numerical examples show that the proposed methods are pressure robust, and the divergence of the resulting velocity and magnetic fields is machine zero for both smooth and singular problems.
\end{abstract}



\section{Introduction}\seclab{introduction}  
Magnetohydrodynamics (MHD) is a field within continuum mechanics that investigates the behavior of electrically conducting fluids in the presence of magnetic fields \cite{davidson_introduction_2001}. This coupled phenomenon holds significant importance across various fields including astrophysics \cite{goedbloed_principles_2004,goossens_introduction_2003}, planetary magnetism \cite{busse_magnetohydrodynamics_1978,krause_mean-field_2016}, nuclear engineering \cite{miyamoto_plasma_1980,forsberg_advanced_2005,tabares_present_2015}, and  metallurgical industry \cite{al-habahbeh_review_2016,davidson_magnetohydrodynamics_1999}.
This paper considers the standard form of the stationary incompressible MHD equations 
\cite{armero_long-term_1996,
gerbeau_stabilized_2000,gerbeau_mathematical_2006,gunzburger_existence_1991}. Specifically, ignoring the effects related to high-frequency phenomena and convection current, and focusing on a medium that is non-polarizable, non-magnetizable, and homogeneous, the resulting MHD equations read
\begin{subequations}
  \eqnlab{mhd_nonlin}
  \begin{align}
	\label{eq:mhd1nonlin}
	- \frac{1}{\Rey} \lap \ub + \nabla p + (\ub \cdot \nabla) \bs{u} + \kappa \bb \times (\nabla \times \bb) &= \gb, \\
	\label{eq:mhd3nonlin}
	\nabla \cdot \ub &= 0, \\
	\label{eq:mhd2nonlin}
	\frac{\kappa}{\Rm} \nabla \times (\nabla \times \bb) + \nabla r - \kappa \nabla \times (\ub \times \db) &= \fb, \\
	\label{eq:mhd4nonlin}
	\nabla \cdot \bb &= 0,
  \end{align}
\end{subequations}
where $\ub$ is the velocity of the fluid (plasma or liquid metal), $\bb$ the magnetic field, $p$ the fluid pressure, and $r$ a Lagrange multiplier\footnote{Sometimes, this variable is also referred to as the magnetic pressure.} that is associated with the divergence constraint \eqnref{mhd4nonlin} on $\bb$. The system \eqnref{mhd_nonlin} is characterized by three dimensionless parameters: the fluid Reynolds number $\Rey > 0$, the magnetic Reynolds number $\Rm > 0$, 
and the coupling parameter $\kappa = \Ha^2 / (\Rey\Rm)$, with the Hartmann number $\Ha > 0$.  For a more detailed exploration of these parameters, we refer to \cite{armero_long-term_1996,gerbeau_mathematical_2006, davidson_introduction_2001}.

The major challenges in the discretization of the MHD equations are the following: (i) multi-physics with
disparate temporal (for the time-dependent MHD equations) and spatial scales; (ii) nonlinearity; (iii) Incompressibility. 
The satisfaction of exact mass conservation in \eqnref{mhd3nonlin} is closely tied to the concept of \textit{pressure-robustness}, which is 
the statement about the independence between the magnitude of the pressure error and the \textit{a priori} error estimate for the velocity \cite{linke_role_2014,linke_pressure-robustness_2016,john_divergence_2017}. Without global enforcement of the continuity equation pointwise, large velocity error can be induced by large pressure error. By global enforcement, we mean that the jump of the normal component of velocity has to vanish across the interior boundaries of elements on a given mesh. In other words, the approximation of velocity $\ubH$ is desired to be in the $\Hsp{}(div)$ space in addition to $\Div{\ubH}=0$, where the divergence operator is defined in a weak sense. The definition of the $\Hsp{}(div)$ space and weak derivative will be elaborated in Section \secref{notations};
(iv) The solenoidal constraint for the magnetic field. The violation of this constraint will cause the wrong topologies of magnetic field lines, leading to plasma transport in an incorrect direction. Furthermore, nonphysical forces proportional to the divergence error could be created, potentially inducing instability \cite{brackbill_effect_1980,balsara_staggered_1999,toth_b0_2000}; and  (v) The dual saddle-point structure of the velocity-pressure. The discretized system is subject to having a notorious large conditional number
and is thus difficult to solve.

Many numerical schemes have been proposed to solve linear, nonlinear, time-dependent, and -independent MHD systems. Regarding spatial discretization, hybridized discontinuous Galerkin (HDG) methods have demonstrated remarkable success \cite{lee_analysis_2019,ciuca_implicit_2020,qiu_mixed_2020,la_spina_superconvergent_2022,gleason_divergence-conforming_2022,muralikrishnan_multilevel_2023}. The HDG methods were first introduced under the context of symmetric elliptic problems \cite{cockburn_unified_2009} to overcome the common criticism had by discontinuous Galerkin (DG) methods on the significantly more globally coupled unknowns than continuous Galerkin methods due to the duplication of degrees of freedoms (DOFs) on element boundaries \cite{cockburn_discontinuous_2017}. The HDG methods reduce the computational cost of DG methods by introducing facet variables uniquely defined on the intersections of element boundaries and removing local (element-wise) DOFs through static condensation, which was initially used in mixed finite element methods (i.e.,\cite{boffi_mixed_2013}). Once the facet variables are solved, the element DOFs can be recovered element-by-element in a completely embarrassing parallel fashion.
Consequently, HDG methods are more efficient while retaining the attractive features of DG methods, such as being highly suitable for solving convection-dominated problems in complex geometries, delivering high-order accuracy in approximations, and accommodating $h$/$p$ refinement \cite{hesthaven_nodal_2008}.

The computational cost of HDG methods can be further lowered by using continuous facet variables across the skeleton of the mesh instead of the discontinuous ones used in HDG methods. This approach led to the embedded discontinuous Galerkin (EDG) methods and was first proposed for solving elliptic problems in \cite{guzey_embedded_2007}. Later, an EDG method was developed for incompressible flows in \cite{labeur_galerkin_2007,labeur_energy_2012} where it was shown that the method inherited many of the desirable features of DG methods. At the same time, the required number of DOFs was less than or equal to those of continuous Galerkin methods on a given mesh. Unfortunately, employing the EDG method can compromise the conservative property  \cite{labeur_energy_2012}. In particular, the velocity field cannot be globally divergence-free, and the mass can only be conserved in the local sense. To strike a balance between HDG and EDG methods, an embedded-hybridized discontinuous Galerkin (E-HDG) method was first developed in \cite{rhebergen_embeddedhybridized_2020} for the Stokes equations. The method is proved to be globally divergence-free and  $\Hsp{}(div)$-conforming. The number of globally coupled DOFs can be substantially reduced by using a continuous basis for the facet velocity field while maintaining a discontinuous basis for the facet pressure. The methodology was later adopted to space-time discretization to solve incompressible flows on moving domains \cite{horvath_exactly_2020,horvath_conforming_2022} and is proved to be globally mass conserving, locally momentum conserving, and energy-stable.  

Several approaches have been suggested to address the issue of the divergence-free constraint on the velocity field within the framework of DG, HDG, or E-HDG methods. An approach to overcome the issue is to use $\Hsp{}(div)$-conforming elements in the approximation of velocity, as discussed in \cite{cockburn_note_2007,greif_mixed_2010,fu_explicit_2019} for DG methods. Alternatively, the constraint can be satisfied locally using solenoidal approximation space for DG methods \cite{baker_piecewise_1990,karakashian_nonconforming_1998,li_locally_2005,yakovlev_locally_2013,klingenberg_efficient_2017} and globally for HDG methods \cite{carrero_hybridized_2006}. On the other hand, $\Hsp{}(div)$-conformity can be acquired with the help of facet variables and proper design of numerical flux for HDG \cite{lehrenfeld_hybrid_2010,lehrenfeld_high_2016,rhebergen_analysis_2017,lederer_hybrid_2018,rhebergen_hybridizable_2018,peters_divergence-conforming_2019,gleason_divergence-conforming_2022} and E-HDG \cite{rhebergen_embeddedhybridized_2020,horvath_exactly_2020,horvath_conforming_2022} methods. Another technique to obtain globally divergence-free methods is to perform post-processing using special projection operators \cite{bastian_superconvergence_2003,cockburn_locally_2005,wang_new_2007,cockburn_equal-order_2009,cockburn_divergence-conforming_2014,guzman_hdiv_2017,lee_analysis_2019-1}. One can also apply pressure-correction methods that relies on Helmholtz decomposition to maintain the divergence-free constraint \cite{botti_pressure-correction_2011,klein_simple_2013}. 

We remark that the divergence-free constraint on the magnetic field given in \eqnref{mhd4nonlin} can be implied by the initial condition in the context of time-dependent MHD equations on the continuous level, and it is also known as the solenoidal involution property of the magnetic field. However, temporal and spatial discretization errors can destroy such a property. Numerous methods have been proposed to satisfy the $\Div{\bb}=0$ constraint in MHD calculations, and some of the ideas can be linked to the approaches developed to handle the $\Div{\ub}=0$ constraint in the context of solving incompressible flow problems. These methods include source term methods \cite{powell_solution-adaptive_1999,janhunen_positive_2000},  projection method \cite{brackbill_effect_1980,derigs_novel_2016} (similar to the projection-correction methods \cite{botti_pressure-correction_2011,klein_simple_2013}),  hyperbolic divergence cleaning methods \cite{dedner_hyperbolic_2002,klingenberg_efficient_2017,bohm_entropy_2020,ciuca_implicit_2020} (similar to artificial compressibility methods \cite{bassi_artificial_2006,bassi_implicit_2007}), locally divergence-free methods \cite{li_locally_2005,yakovlev_locally_2013} (use locally solenoidal approximation space and is similar to \cite{baker_piecewise_1990,karakashian_nonconforming_1998,klingenberg_efficient_2017}),  globally divergence-free methods \cite{fu_explicit_2019} (use globally solenoidal approximation space), and constrained transport (CT) methods \cite{evans_simulation_1988,balsara_staggered_1999,londrillo_high-order_2000,toth_b0_2000}. Another approach to obtain a divergence-free and $\Hsp\LRp{div}$-conforming method was developed in \cite{gleason_divergence-conforming_2022}, using an HDG method that hybridizes the facet Lagrange multiplier variable as well.

In this paper, we devise a divergence-free and $\Hsp\LRp{div}$-conforming HDG and E-HDG methods for solving the stationary incompressible viso-resistive MHD equations given in \eqnref{mhd_nonlin}. {\em Though both approaches are constructed in parallel, our exposition  will focus on E-HDG.} 
We obtain $\Hsp\LRp{div}$-conformity by following an idea similar to \cite{rhebergen_embeddedhybridized_2020,horvath_exactly_2020} and \cite{gleason_divergence-conforming_2022} through hybridization via a facet pressure and a facet Lagrange multiplier field using discontinuous facet functions. For the E-HDG variant, we use continuous facet functions for the velocity and the magnetic fields. Moreover, we extended the work in \cite{lee_analysis_2019} and employed an upwind type numerical flux that is based on the first-order form of the linearized MHD system. This is in contrast to the work in  \cite{gleason_divergence-conforming_2022} where the authors hybridized another popular class of DG methods called interior penalty discontinuous Galerkin (IPDG) methods \cite{douglas_interior_1976,baker_finite_1977,wheeler_elliptic_1978,arnold_interior_1982,baker_piecewise_1990} to construct the divergence-free and divergence-conforming HDG method for the time-dependent incompressible viso-resistive MHD equations. To ensure stability, the penalty parameter in IPDG methods, such as the one in typical Nitsche methods, must be sufficiently large. However, no analytically proven bound is available for this penalty parameter. Conversely, our approaches do not suffer from such difficulty, and the criteria of the stabilization parameters are well-defined.  With a few assumptions, our proposed schemes are well-posed. The resulting E-HDG discretization for the linearized MHD model can be incorporated into a Picard iteration to construct a fully nonlinear solver provided it converges. This approach ensures that the divergence-free and $\Hsp\LRp{div}$-conforming properties still hold for the nonlinear case. 
Moreover, all results we discussed in the context of our E-HDG method are still applied to the HDG counterpart, including well-posedness, divergence-free property, and $\Hsp\LRp{div}$-conformity.  

The paper is organized as follows. Section \secref{notations} outlines the notations. Section \secref{EHDG_formulation} proposes both the HDG and E-HDG discrtizatinos for the linearized incompressible viso-resistive MHD equations. In addition, the well-posedness of both methods is proven. Further, we prove the divergence-free property and $\Hsp\LRp{div}$-conformity of both the velocity (i.e., pointwise mass conservation) and the magnetic (i.e., pointwise absence of magnetic monopoles) fields for linear and nonlinear cases. The implementation aspect is discussed in Section \secref{numerical_results}, where we also compare the computational costs required by HDG and E-HDG methods. Several numerical examples for linear and nonlinear incompressible viso-resistive MHD equations are presented to demonstrate the accuracy and convergence of our proposed methods in both two- and three-dimensional settings. Section \secref{conclusion} concludes the paper with future work.
\section{Notations}\seclab{notations}
In this section, we introduce common notations and conventions to be
used in the rest of the paper. Let $\Omega \subset \R^\d$, $d=2,3$,
be a bounded domain such that it is simply connected, and its boundary $\pOmega$ is a Lipschitz manifold with only one component. Suppose that we have a triangulation of $\Omega$ consisting of a finite number of nonoverlapping $d$-dimensional simplices, i.e., triangles for two dimensions and tetrahedra for three dimensions, respectively. We assume that the triangulation is shape-regular, i.e., for all $d$-dimensional simplices in the triangulation, the ratio of the diameter of the simplex and the radius of an inscribed $d$-dimensional ball is uniformly bounded. We will use $\Omegah$ and $\Gh$ to denote the sets of $d$- and ($d-1$)-dimensional simplices of the triangulation
and call $\Gh$ the mesh skeleton of the triangulation. The boundary and interior mesh skeletons are defined by $\Ghb := \{ e \in \Gh \,:\, e \subset \pd \Omega\}$ and $\Gho := \Gh \setminus \Ghb$. We also define $\pOmegah := \LRc{\pK:\K \in \Omegah}$. The mesh size of triangulations is $h := \max_{\K \in \Omegah} \ensuremath{\text{diam}}(\K)$.

We use $\LRp{\cdot,\cdot}_D$ (respectively $\LRa{\cdot,\cdot}_D$)
to denote the $L^2$-inner product on a $d$- (respectively $(d-1)$-)
dimensional domain $D$. The standard notation $W^{s,p}(D)$, $s \ge 0$, $1 \le p \le \infty$,
is used for the Sobolev space on $D$ based on $L^p$-norm with
differentiability $s$ (see, e.g., \cite{evans_partial_2022}) and $\nor{\cdot}_{W^{s,p}(D)}$
denotes the associated norm. In particular, if $p = 2$, we use $H^s(D) := W^{s,2}(D)$
and $\nor{\cdot}_{s,D}$.
$W^{s,p}(\Omegah)$ denotes the space of functions whose restrictions on $\K$
reside in $W^{s,p}(\K)$ for each $\K \in \Omegah$ and its norm is
$\nor{u}_{W^{s,p}(\Omegah)}^p := \sum_{\K \in \Omegah} \nor{u|_{\K}}_{W^{s,p}(\K)}^p$
if $1 \le p < \infty$ and $\nor{u}_{W^{s,\infty}(\Omegah)} := \max_{\K \in \Omegah} \nor{u|_{\K}}_{W^{s,\infty}(\K)}$.
For simplicity, we use $\LRp{\cdot, \cdot}$, $\LRa{\cdot,\cdot}$,
$\nor{\cdot}_s$, $\nor{\cdot}_{\pOmegah}$, and $\nor{\cdot}_{W^{s,\infty}}$ for
$\LRp{\cdot, \cdot}_{\Omega}$, $\LRa{\cdot,\cdot}_{\pOmegah}$,
$\nor{\cdot}_{s, \Omega}$, $\nor{\cdot}_{0, \pOmegah}$, and $\nor{\cdot}_{W^{s,\infty}(\Omegah)}$, respectively.


For vector- or matrix-valued functions these notations are naturally extended with a
component-wise inner product.
We define similar spaces (respectively inner products and norms) on a single element and a single skeleton face/edge
by replacing $\Omegah$ with $\K$ and $\Gh$ with $\e$.
We define the gradient of a vector, the divergence of a matrix, and the outer product symbol $\otimes$ as:
\[
  \LRp{\nabla \ub}_{ij} = \pp{u_i}{x_j}, \quad
  \LRp{\div \Lb}_i = \div \Lb\LRp{i,:} = \sum_{j=1}^d\pp{\bs{L}_{ij}}{x_j}, \quad
  \LRp{\bs{a} \otimes \bs{b}}_{ij} = a_i b_j = \LRp{\bs{a}\bs{b}^T}_{ij}.
\]
The curl of a vector when $d=3$ takes its standard form, $\LRp{\Curl\bb}_i=\sum_{j,k}\epsilon_{ijk}\pp{\bb_k}{x_j}$, where $\epsilon$ is the Levi-Civita symbol. When $d=2$, let us explicitly define the curl of a vector as the scalar quantity $\Curl\bb=\pp{\bb_2}{x_1}-\pp{\bb_1}{x_2}$, and the curl of a scalar as the vector quantity $\Curl a=\LRp{\pp{a}{x_2},-\pp{a}{x_1}}$. In this paper, $\n$ denotes a unit outward normal vector field on faces/edges.
If $\pd \Km \cap \pd \Kp \in \Gh$ for two
distinct simplices $\Km, \Kp$, then $\nm$ and $\np$ denote
the outward unit normal vector fields on $\pd \Km$ and $\pd \Kp$, respectively, and
$\nm = - \np$ on $\pd \Km \cap \pd \Kp$. We simply use $\n$ to denote either $\nm$ or $\np$ in an
expression that is valid for both cases, and this convention is also
used for other quantities (restricted) on a face/edge $\e \in \Gh$. We also define $\bs{N}:=\n\otimes\n$ and $\bs{T}:=\bs{I}-\bs{N}$. For a scalar quantity
$u$ which is double-valued on $e := \pd \Km \cap \pd \Kp$, the jump term on $e$ is defined by
$\jump{u \n}|_e = u^+ \np + u^- \nm$ where $u^+$ and $u^-$ are the traces of $u$ from $\Kp$-
and $\Km$-sides, respectively. For double-valued vector quantity $\ub$ and matrix quantity $\Lb$,
jump terms are $\jump{\ub \cdot \n}|_e = \ub^+ \cdot \np + \ub^- \cdot \nm$, ${\jump{\ub\times\n}}|_{\e}=\ub^+\times\n^++\ub^-\times\n^-$, and
$\jump{\Lb \n}|_e = \Lb^+ \np + \Lb^- \nm$ where $\Lb \n$ denotes the matrix-vector product.

We define $\Poly_k\LRp{\K}$ as
the space of polynomials of degree at most $k$ on $\K$, with $k \geq 0$, and we define
\algns{
  \Poly_k\LRp{\Omegah} := \LRc{ u \in \Lsp^2(\Omega) \;:\; \eval{u}_{\K} \in \Poly_k\LRp{\K} \; \forall \K \in \Omegah } .
}
The space of polynomials on the mesh skeleton $\Poly_k\LRp{\Gh}$ is similarly defined,
and their extensions to vector- or matrix-valued polynomials $\LRs{\Poly_k(\Omegah)}^d$,
$\LRs{\Poly_k(\Omegah)}^{d\times d}$, $\LRs{\Poly_k(\Gh)}^d$, etc, are straightforward.

Finally, we use the usual definition of the $\Hsp{}(div)$- and $\Hsp{}(curl)$-conforming spaces, which are typical for mixed methods, and for methods dealing with electromagnetism, see \cite{ern_finite_2021, boffi_mixed_2013},
\begin{align}
&\Hsp{}\LRp{div,\Omega}:=\LRc{\ub \in \LRs{\Lsp^2(\Omega)}^d \;:\; \div\ub\in{\Lsp^2(\Omega)}},\notag\\
&\Hsp{}\LRp{curl,\Omega}:=\LRc{\ub \in \LRs{\Lsp^2(\Omega)}^d \;:\; \Curl\ub\in\LRs{\Lsp^2(\Omega)}^{\tilde{d}}},\eqnlab{eq:Hcurl}
\end{align}
where $\tilde{d} = 3$ if $d = 3$, $\tilde{d} = 1$ if $d = 2$. In addition, the divergence $\div(\cdot)$ and curl $\curl(\cdot)$ operators should be thought of in the weak sense (an extension of weak derivative defined in Definition 2.3 in \cite{ern_finite_2021}). Note that the jump condition $\eval{\jump{\ub\cdot\n}}{\e}=0$ and $\eval{\jump{\ub\times\n}}{\e}=0$ is necessary for ensuring $\ub\in\Hsp{}\LRp{div,\Omega}$ and $\ub\in\Hsp{}\LRp{curl,\Omega}$, respectively (Theorem 18.10 in \cite{ern_finite_2021}).
\section{An E-HDG Formulation}\seclab{EHDG_formulation}
First, consider the following incompressible viso-resistive MHD system linearized from Eq. \eqnref{mhd_nonlin}
\begin{subequations}
  \eqnlab{mhdlin}
  \begin{align}
	\label{eq:mhd1lin_1}
	- \frac{1}{\Rey} \lap \ub + \nabla p + (\wb \cdot \nabla) \bs{u} + \kappa \db \times (\nabla \times \bb) &= \gb, \\
	\label{eq:mhd3lin_1}
	\nabla \cdot \ub &= 0, \\
	\label{eq:mhd2lin_1}
	\frac{\kappa}{\Rm} \nabla \times (\nabla \times \bb) + \nabla r - \kappa \nabla \times (\ub \times \db) &= \fb, \\
	\label{eq:mhd4lin_1}
	\nabla \cdot \bb &= 0.
  \end{align}
\end{subequations}
Here, $\db$ is a prescribed magnetic field 
and $\wb$ is a prescribed velocity field.
From this point forward, we assume (see, e.g., \cite{cesmelioglu_analysis_2013,houston_mixed_2009} for similar assumptions) $\db \in \LRs{W^{1,\infty}\LRp{\Omega}}^d$,
$\wb \in \LRs{W^{1,\infty}\LRp{\Omegah}}^d \cap H(div,\Omega)$, $\nabla \cdot \wb = 0$ and $\gb,\fb\in\LRs{\Lsp^{2}\LRp{\Omega}}^d$.

To apply the upwind type of numerical flux based on the work \cite{lee_analysis_2019}, we cast \eqnref{mhdlin} into a first-order form by introducing auxiliary variables $\Lb$ and $\Jb$, 
\begin{subequations}
  \eqnlab{mhdlin-first}
  \begin{align}
	\label{eq:mhd1lin1n}
	\Rey \Lb - \nabla \ub &= \bs{0}, \\
	\label{eq:mhd2lin1n}
	- \nabla \cdot \Lb + \nabla p + (\wb \cdot \nabla) \ub + \kappa \db \times (\nabla \times \bb) &= \gb, \\
	\label{eq:mhd3lin1n}
	\nabla \cdot \ub &= 0, \\
	\label{eq:mhd4lin1n}
	\frac{\Rm}{\kappa} \Jb - \nabla \times \bb &= \bs{0}, \\
	\label{eq:mhd5lin1n}
	\nabla \times \Jb + \nabla r  - \kappa \nabla \times (\ub \times \db) &= \fb, \\
	\label{eq:mhd6lin1n}
	\nabla \cdot \bb &= 0,
  \end{align}
\end{subequations}
with Dirichlet boundary conditions 
\begin{equation}
  \eqnlab{MHD_BCs_Dirichlet}
  \begin{aligned}
	\ub = \bs{u}_D , \quad \bb:= \bb_D, \quad r = 0   \qquad  \textrm{ on } \pOmega.
  \end{aligned}
\end{equation}
In addition, we require the compatibility condition for $\ub_D$ and the mean-value zero condition for $p$:
\begin{equation}
  \label{compatibility}
  \LRa{\ub_D \cdot \n, 1}_{\pOmega} = 0 , \qquad 
  (p,1)_\Omega = 0.
\end{equation}

To achieve $\Hsp{}(div)$-conforming property, we introduce constant parameters $\alpha_1,\beta_1,\beta_2\in\R$, and define the numerical flux inspired by the work \cite{lee_analysis_2019} as 

\begin{equation}\small
  \label{eq:EHDGflux}
  \LRs{
	\begin{array}{c}
	  \Fh^1 \cdot \n \\
	  \Fh^2 \cdot \n \\
	  \textcolor{black}{\Fh^3 \cdot \n} \\
	  \Fh^4 \cdot \n \\
	  \Fh^5 \cdot \n \\
	  \textcolor{black}{\Fh^6 \cdot \n}
	\end{array}
  } =
  \LRs{
	\hspace{-0.4em}
	\begin{array}{c}
	  -\ubh \otimes \n \\
	  -\Lb\n + m\ub +\textcolor{black}{\ph}\n + \half \kappa \db \times \LRp{\n \times \LRp{\bb+\bbh}} + \alpha_1 \LRp{\ub - \ubh} \\
	  \textcolor{black}{\ub \cdot \n} \\
	  -\n \times \bbh \\
	  \n \times \Jb  + \textcolor{black}{\rh}\n -\half \kappa \n \times \LRp{\LRp{\ub +\ubh}\times \db} + \LRp{\beta_1\opT+\beta_2\opN}\LRp{\bb-\bbh} \\
	  \textcolor{black}{\bb \cdot \n}
	\end{array}
	\hspace{-0.4em}
  },
\end{equation}
where $m:= \wb \cdot \n$. It should be noted that $\ubh$, $\ph$, $\bbh$, and $\rh$ are the restrictions (or traces) of $\ub$, $\p$, $\bb$, and $\r$ on $\Gh$. These $\ubh$, $\ph$, $\bbh$, and $\rh$ will be regarded as unknowns in discretizations to obtain an E-HDG method. It will be shown that the conditions $\alpha_1
> \half \nor{\wb}_\Linfty$, and $\beta_1\opT+\beta_2\opN>0$\footnote{The sign of ``greater than" here means that the matrix (or the second order tensor) $\beta_1\opT+\beta_2\opN$ is positive definite.} are
sufficient for the well-posedness of our E-HDG formulation.
Note that all 6 components of the E-HDG flux, $\Fh$, for simplicity are denoted
in the same fashion (by a bold italic symbol).
It is, however, clear from \eqnref{mhdlin-first} that  $\Fh^1$ is a third order tensor,
$\Fh^2$ is a second order tensor, $\Fh^3$ is a vector, etc, and that the
normal E-HDG flux components, $\Fh^i \cdot \n$ in \eqref{eq:EHDGflux},
are tensors of one order lower.

For discretization, we introduce the discontinuous piecewise and the continuous polynomial spaces
\begin{gather*}
  { \GbM := \LRs{\Poly_k(\Omegah)}^{d \times d}, \qquad \VbM := \LRs{\Poly_k(\Omegah)}^d, \qquad \QbM := \Poly_{\textcolor{black}{\kbr}}(\Omegah) ,} \\
  { \HbM := \LRs{\Poly_k(\Omegah)}^{\tilde{d}}, \qquad \CbM := \LRs{\Poly_{k}(\Omegah)}^d, \qquad \SbM := \Poly_{\textcolor{black}{\kbr}}(\Omegah), \qquad \MubM := \LRs{\Poly_k(\Gh)\textcolor{black}{\cap\Csp(\Gh)}}^d ,} \\
  {\RhoM:= \LRs{\Poly_k(\Gh)}, \qquad \LambM := \LRs{\Poly_k(\Gh)\textcolor{black}{\cap\Csp(\Gh)}}^d, \qquad \GambM := \LRs{\Poly_k(\Gh)},}
\end{gather*}
where $\kbr:=\k-1$, $\Csp(\Gh)$ is the continuous function space defined on the mesh skeleton, and $\tilde{d}$ is defined in \eqnref{eq:Hcurl}.

\begin{rema}\remalab{HDG}
The functions in $\MubM$ and $\LambM$ are used to approximate the traces of the velocity and the magnetic field, respectively. By a slight modification of these spaces (i.e., $\MubM:=\LRs{\Poly_k(\Gh)}^d$ and $\LambM := \LRs{\Poly_k(\Gh)}^d$), a divergence-free and $\Hsp{}(div)$-conforming HDG method can be obtained. All the results presented in Sections \secref{wellposednessEHDG}, \secref{localwellposednessEHDG} and \secref{conservationEHDG} can be directly applied to the resulting HDG method. In addition, we will numerically compare the computational time needed by HDG and E-HDG methods in Section \secref{numerical_results}. 
\end{rema}

Let us introduce two identities which are useful throughout the paper:
\begin{subequations}
  \eqnlab{eq:int-by-parts}
  \algn{
	\LRp{ \ub , \db \times \LRp{\Curl \bb} }_\K
	&= \LRp{ \bb, \Curl \LRp{\ub \times \db} }_\K + \LRa{\db \times \LRp{\n \times \bb}, \ub}_\pK, \\
	\LRs{ \db \times \LRp{\n \times \bb} } \cdot \ub
	&= - \LRs{\n \times \LRp{\ub \times \db}} \cdot \bb .
  }
\end{subequations}
These identities follow from integration by parts and vector product identities.

Next, we multiply \eqref{eq:mhd1lin1n} through \eqref{eq:mhd6lin1n} by
test functions $(\Gb,\vb,q,\Jb,\cb,s)$, integrate by parts all terms, and introduce the numerical
flux \eqref{eq:EHDGflux} in the boundary terms.  This results in a
local discrete weak formulation: 
\begin{subequations}
  \eqnlab{local}
  \begin{align}
	\label{eq:local_1_1}
	\Rey\LRp{\LbH, \Gb}_\K + \LRp{\ubH, \Div\Gb}_\K
	+ \LRa{\FbhH^1 \cdot \n, \Gb}_\pK &= 0, \\
	\label{eq:local_2_1}
	\LRp{\LbH, \Grad\vb}_\K - \LRp{\pH, \Div \vb}_\K - \LRp{\ubH \otimes \wb , \Grad \vb}_\K \quad & \\
	+ \kappa\LRp{\bbH, \Curl\LRp{\vb \times \db}}_\K
	+ \LRa{\FbhH^2 \cdot \n, \vb}_\pK &= \LRp{\gb, \vb}_\K, \notag \\ 
	\label{eq:local_3_1}
	-\LRp{\ubH, \Grad \q}_\K
	+ \LRa{\textcolor{black}{\FbhH^3 \cdot \n}, \q}_\pK &= 0, \\
	\label{eq:local_4_1}
	\frac{\Rm}{\kappa}\LRp{\JbH, \Hb}_\K - \LRp{\bbH, \Curl \Hb}_\K
	+ \LRa{\FbhH^4 \cdot \n, \Hb}_\pK &= 0, \\
	\label{eq:local_5_1}
	\LRp{\JbH, \Curl \cb}_\K - \LRp{\rH, \Div \cb}_\K - \kappa\LRp{\ubH, \db \times \LRp{\Curl\cb}}_\K \quad & \\
	+ \LRa{\FbhH^5 \cdot \n, \cb}_\pK &= \LRp{\fb,\cb}_\K, \notag \\
	\label{eq:local_6_1}
	-\LRp{\bbH, \Grad \s}_\K
	+ \LRa{\textcolor{black}{\FbhH^6 \cdot \n}, \s}_\pK &= 0,
  \end{align}
\end{subequations}
for all $(\Gb,\vb,q,\Hb,\cb,s) \in \GbM\LRp{K} \times \VbM\LRp{K} \times \QbM\LRp{K}\times \HbM\LRp{K} \times \CbM\LRp{K} \times \SbM\LRp{K}$
and for all $\K \in \Omegah$, where quantities with subscript $h$ are the discrete counterparts of the continuous ones, for example, $\ubH$ and  $\LbH$ are the discrete approximations of $\ub$ and $\Lb$.

Since  $\ubh$, $\ph$, $\bbh$, and $\rh$ are facet unknowns introduced in addition to the original unknowns, we
need to equip extra equations to make the system \eqnref{local} well-posed. To
that end, we observe that an
element $\K$ communicates with its neighbors only through the trace unknowns. 
For the E-HDG method to be conservative, we weakly enforce the
continuity of the numerical flux \eqref{eq:EHDGflux} across each interior edge.
Since $\ubhH$ and $\bbhH$ are single-valued on $\Gh$, we automatically have that
$\jump{\FbhH^1 \cdot \n} = \bs{0}$ and $\jump{\FbhH^4 \cdot \n} = \bs{0}$.
The conservation constraints to be enforced are reduced to
\begin{equation}\label{global}\small
    \begin{aligned}
      &\LRa{\jump{\FbhH^2 \cdot \n},\mub}_\ed = 0, \qquad 
      \LRa{\jump{\FbhH^3 \cdot \n},\rho}_{\ed} = 0, \\
      &\LRa{\jump{\FbhH^5 \cdot \n}, \lambdab}_\ed = 0, \qquad 
      \LRa{\jump{\FbhH^6 \cdot \n}, \gamma}_\ed = 0,
    \end{aligned}
\end{equation}
for all $(\mub,\rho,\lambdab,\gamma) \in \MubM\LRp{e} \times \RhoM\LRp{e} \times \LambM\LRp{e} \times \GambM\LRp{e}$, 
and for all $e$ in $\Gho$. Furthermore, the following constraint on the domain boundary is required  in order to establish the  well-posedness of our HDG formulations: 
\begin{equation}
    \eqnlab{Constraint}
    \LRa{\ubhH\cdot\n,\rho}_e = \LRa{\ubH\cdot\n,\rho}_e,
\end{equation}
for all $\rho\in\RhoM\LRp{e}$ for all $e$ in $\Ghb$. This constraint means that we weakly enforce $\ubhH=\ubH$ on the boundary, and is also used in \cite{labeur_energy_2012,rhebergen_spacetime_2012,rhebergen_hybridizable_2018} where hybridized IPDG methods are developed for solving the incompressible Navier-Stokes equations. A similar constraint is applied to the magnetic field to ensure pointwise satisfaction of no monopole condition,
\begin{equation}
    \eqnlab{Constraint_b}
    \LRa{\bbhH\cdot\n,\gamma}_e = \LRa{\bbH\cdot\n,\gamma}_e,
\end{equation}
for all $\gamma\in\GambM\LRp{e}$ for all $e$ in $\Ghb$.

Finally, we enforce the Dirichlet boundary conditions weakly through the facet unknowns:
\begin{equation}\eqnlab{BCs}
\begin{aligned}
  \LRa{\ubhH,\mub}_e = \LRa{\ub_D,\mub}_e, \qquad 
  \LRa{\bbhH,\lambdab}_e = \LRa{\bb_D,\lambdab}_e, \qquad
  \LRa{\rhH,\gamma}_e = 0,
\end{aligned}
\end{equation}
for all $(\mub,\lambdab,\gamma) \in \Mb_h\LRp{e} \times \bs{\Lambda}_h\LRp{e} \times \Gamma_h\LRp{e}$ 
for all $e$ in $\Ghb$. 
In Eq. \eqnref{local}-\eqnref{Constraint_b} we seek $(\LbH,\ubH,\pH,\JbH,\bbH,\rH) 
\in \GbM \times \VbM \times \QbM \times \HbM \times \CbM \times \SbM$ and
$(\ubhH, \ph,\bbhH,\rhH) \in \MubM \times \RhoM \times \LambM \times \GambM$.
For simplicity, we will not state explicitly that equations hold for all test functions, for all elements, or
for all edges.

We will refer to $\LbH,\ubH,\pH,\JbH,\bbH,$ and $\rH$ as the
\textit{local variables}, and to equation \eqnref{local} on each element as the \textit{local solver}. 
This reflects the fact that we can solve for local variables
element-by-element as functions of $\ubhH, \phH, \bbhH,$ and $\rhH$. On the
other hand, we will refer to $\ubhH, \phH, \bbhH,$ and $\rhH$
as the \textit{global variables}, which are governed by equations
\eqref{global},\eqnref{BCs}, and \eqnref{Constraint} on the mesh skeleton. 
For the uniqueness of the discrete pressure $\pH$, 
we enforce the discrete counterpart of \eqref{compatibility}:
\begin{align}
  \label{eq:global_5_1}
  (\pH,1) &= 0.
\end{align}

\subsection{Well-posedness of the E-HDG formulation}
\seclab{wellposednessEHDG}
In this subsection, we discuss the well-posedness of \eqnref{local}--\eqnref{global_5_1}. We would like to point out that the result presented in this subsection is also valid for the proposed HDG version in \cref{rema:HDG}.
\begin{theorem}\theolab{global}
Let $\Omega$ be simply connected with one component to $\pOmega$.
Let $\alpha_1$, $\beta_1$, $\beta_2$ $\in\R$ such that $\alpha_1 > \half \nor{\wb}_{L^\infty(\Omega)}$ and  $\beta_1\opT+\beta_2\opN>0$.
The system \eqnref{local}--\eqnref{global_5_1} is well-posed, in the sense that
given $\fb$, $\gb$, $\ub_D$, and $\hb_D$, there exists a unique solution
${\LRp{ \LbH, \ubH, \pH, \JbH, \bbH, \rH, \ubhH, \phH, \bbhH, \rhH }}$.
\end{theorem}
\begin{proof}
\eqnref{local}--\eqnref{global_5_1} has the same number of equations and unknowns, so it is enough to show
that $\LRp{\gb,\fb, \ub_D, \bb_D} = \bs{0}$ implies
${(\LbH,\ubH,\pH,\JbH,\bbH, \rH,\ubhH,\phH,\bbhH,\rhH) = \bs{0}}$.
To begin, we take $(\Gb,\vb,q,\Jb,\cb,s) = (\LbH,\ubH,\pH,\JbH,\bbH,\rH)$,
integrate by parts the first four terms of \eqref{eq:local_2_1} and the first term of \eqref{eq:local_5_1},
sum the resulting equations in \eqnref{local}, and sum over all elements to arrive at 
\begin{align}
  \Rey\nor{\LbH}_0^2 + \frac{\Rm}{\kappa}\nor{\JbH}_0^2
  - \LRa{\ubhH \otimes \n, \LbH}
  + \LRa{ \frac{m}{2} \ubH, \ubH } \notag
  + \LRa{ \alpha_1 (\ubH-\ubhH), \ubH } \\
  \label{eq:sum_energy4_1}
  + \LRa{\phH\n, \ubH}
  + \LRa{ \half \kappa \db \times \LRp{\n \times \bbhH}, \ubH }
  - \LRa{ \n \times \bbhH, \JbH }
  + \LRa{ \rhH \n, \bbH } \\
  + \LRa{ \LRp{\beta_1\opT+\beta_2\opN}(\bbH-\bbhH), \bbH } \notag
  - \LRa{ \half\kappa\n \times \LRp{\ubhH \times \db}, \bbH }
  &= 0,
\end{align}
where we have used $\Div \wb = 0$ and the following integration by parts identities:
\algns{
- \LRp{\ubH,\wb \cdot \nabla \ubH}_\K \allowbreak
= - \half\LRp{\wb,\nabla(\ubH\cdot\ubH)}_\K \allowbreak
&= - \LRa{\frac{m}{2}\ubH,\ubH}_\pK .
}
Next, setting $({\mub, \rho, \lambdab, \gamma}) = ({\ubhH,\phH,\bbhH,\rhH})$,
and summing \eqref{global} over all interior edges give
\begin{align}
  \LRa{-\LbH\n + m\ubH +\phH\n + \half\kappa\db \times \LRp{\n \times \bbH}
  + \alpha_1\LRp{\ubH - \ubhH},\ubhH}_{\pOmegah\setminus\pOmega} \notag & \\
  \label{eq:sum_energy5_1}
  + \LRa{\ubH\cdot\n ,\phH}_{\pOmegah\setminus\pOmega}\notag &\\
  + \LRa{\n \times \JbH + \rhH\n -\half\kappa\n \times \LRp{\ubH \times \db}
  + \LRp{\beta_1\opT+\beta_2\opN}\LRp{\bbH -\bbhH}, \bbhH}_{\pOmegah\setminus\pOmega} & \\
  \notag
  +\LRa{\bbH\cdot\n, \rhH}_{\pOmegah\setminus\pOmega}
  &= 0,
\end{align}
where we used the continuity of $\db$ and the single-valued nature across the element boundaries of global variables to eliminate 
$\langle{\db \times ({\n \times \bbhH}),\ubhH}\rangle_{\pOmegah\setminus\pOmega}$ and 
$\langle{\n \times \LRp{\ubhH \times \db},\bbhH}\rangle_{\pOmegah\setminus\pOmega}$.

Since $\ub_D = \bs{0}$ and $\bb_D = \bs{0}$ by assumption, we conclude from the boundary conditions \eqnref{BCs}
that $\ubhH = \bs{0}$, $\bbhH = \bs{0}$, and $\rhH = 0$  on $\pOmega$. In addition, from the constraint \eqnref{Constraint} we also have $\LRa{\ubH\cdot\n,\phH}_e=\LRa{\ubhH\cdot\n,\phH}_e$ on the boundary and hence $\LRa{\ubH\cdot\n,\phH}_{\pOmega}=0$. 
Subtracting \eqref{eq:sum_energy5_1} from \eqref{eq:sum_energy4_1} and using the fact that $\ubhH, \bbhH,$ $\rhH$, and $\LRa{\ubH\cdot\n,\phH}_{\pOmega}$ vanish on the physical boundary $\pOmega$,  we arrive at
\begin{align}
  \label{eq:sum_energy6_2}
  \Rey\nor{\LbH}_0^2 + \frac{\Rm}{\kappa}\nor{\JbH}_0^2
  + \LRa{ \alpha_1 (\ubH-\ubhH), (\ubH-\ubhH) }
  + \LRa{ \frac{m}{2} \ubH, \ubH } \quad & \\
  \quad - \LRa{ m \ubH, \ubhH }  \notag
    + \LRa{\LRp{\beta_1\opT+\beta_2\opN}\LRp{\bbH-\bbhH}, \bbH-\bbhH}
    &= 0.
\end{align}
Finally, using the fact that $\wb \in H(div,\Omega)$ and $\ubhH = \bs{0}$ on $\pOmega$,
we can freely add
$0=\LRa{\frac{m}{2}\ubhH,\ubhH}$ to rewrite \eqref{eq:sum_energy6_2} as
\begin{align}
  \label{eq:sum_energy6_1}
  \Rey\nor{\LbH}_0^2 + \frac{\Rm}{\kappa}\nor{\JbH}_0^2
  + \LRa{ \LRp{\alpha_1+\frac{m}{2}} (\ubH-\ubhH), (\ubH-\ubhH) } \quad & \\
    + \LRa{\LRp{\beta_1\opT+\beta_2\opN}\LRp{\bbH-\bbhH}, \bbH-\bbhH}
    &= 0.\notag  
\end{align}
Recalling $\alpha_1 > \half \nor{\wb}_\Linfty$ and $\beta_1\opT+\beta_2\opN>0$, we can conclude that
$\LbH = \bs{0}$, $\JbH = \bs{0}$, that $\ubH = \ubhH$, and $\bbH = \bbhH$ on $\Gh$. 

Now, we integrate \eqref{eq:local_1_1} by parts to obtain $\Grad\ubH = \bs{0}$ in $\K$,
which implies that $\ubH$ is element-wise constant.
The fact that $\ubH = \ubhH$ on $\Gh$
means $\ubH$ is continuous across $\Gh$. Since $\ubH = \bs{0}$ on $\pOmega$,
we conclude that $\ubH = \bs{0}$ and therefore $\ubhH = \bs{0}$.

Since $\bbH = \bbhH$ on $\Gh$, $\bbH$ is continuous on $\Omega$.
Integrating both \eqref{eq:local_4_1} and
\eqref{eq:local_6_1} by parts, we have $\Curl \bbH = \bs{0}$ and $\Div \bbH = 0$ on $\Omega$.
When $\bbH \in H(div,\Omega) \cap H(curl,\Omega)$ and $\bbH = \bs{0}$ on $\pOmega$,
and recalling that $\Omega$ is simply connected with one component to the boundary,
there is a constant $C>0$ such that $\norm{\bbH}_0
\leq C (\norm{\Div \bbH}_0 + \norm{\Curl \bbH}_0)$ \cite[Lemma~3.4]{girault_finite_1986}. 
This implies that $\bbH = \bs{0}$, and hence $\bbhH = \bs{0}$.

Taking account of the vanishing variables we had discussed, integrating by parts reduces 
\eqref{eq:local_2_1} and \eqref{eq:local_5_1} to:
\begin{equation}
    (\Grad \pH, \vb)_\K - \LRa{\LRp{\pH - \phH}\n,\vb}_{\pK} = 0,
\end{equation}
and
\begin{equation}
    (\Grad \rH, \cb)_\K - \LRa{\LRp{\rH - \rhH}\n,\cb}_{\pK} = 0,  
\end{equation}
respectively. Given that $\pH|_{\K}$, $\rH|_{\K}\in\Poly_{k-1}\LRp{\K}$ and a simplicial mesh is used, we can invoke the argument of Nédélec space to conclude that $\pH = \phH$ and $\rH = \rhH$ on $\pK$ (Proposition 4.6 in \cite{rhebergen_analysis_2017}). This implies that $(\Grad \pH, \vb)_\K=0$ and $(\Grad \rH, \cb)_\K=0$. Thus, $\pH$ and $\rH$ are elementwise constants.
Since $\rH = \rhH$ on $\Gho$, then $\rH$ is continuous on $\Omega$,
and since $\rH = 0$ on $\pOmega$,
we can conclude that $\rH = 0$, and hence $\rhH = 0$.
Finally, we use the result $\pH = \phH$ on $\Gho$ to
conclude that $\pH$ is continuous and a constant on $\Omega$.
Using the zero-average condition \eqref{eq:global_5_1} yields $\pH = 0$ and hence $\phH=0$.
\end{proof}

\subsection{Well-posedness of the local solver}\seclab{localwellposednessEHDG}
A key advantage of HDG or E-HDG methods is the decoupling
computation of the local variables $(\LbH,\ubH,\pH,\JbH,\bbH,\rH)$
and the global variables $({\ubhH, \phH, \bbhH,\rhH})$. In our E-HDG scheme, we first solve \eqnref{local} for local unknowns $(\LbH,\ubH,\pH,\JbH,\bbH,\rH)$ as a function of $({\ubhH, \phH,\bbhH,\rhH})$ (local solver), then these are substituted into
\eqref{global} on the mesh
skeleton to solve for the unknowns $({\ubhH, \phH,\bbhH,\rhH})$ (global solver). Finally, $( \LbH, \ubH, \allowbreak \pH, \allowbreak \JbH, \bbH, \rH )$ are computed with the local solver using $({\ubhH, \phH,\bbhH,\rhH})$, so well-posedness of the local solver is essential. It should be emphasized again that the result presented in this subsection is also valid for the HDG version in \cref{rema:HDG}.

\begin{theorem}\theolab{local}
Let $\alpha_1,\beta_1,\beta_2\in\R$ such that $\alpha_1 > \half \nor{\wb}_{L^\infty(\Omega)}$ and $\beta_1\opT+\beta_2\opN>0$. The local solver given by \eqnref{local} is well-posed. In other words, given $(\ubhH, \phH, \bbhH, \rhH, \gb,$ $ \fb, \rho_h)$, there exists 
a unique solution 
$(\LbH,\ubH,\pH,\JbH,\bbH,\rH)$ of the system.
\end{theorem}
\begin{proof}
We show that
$(\ubhH, \phH, \bbhH, \rhH, \gb, \fb, \rho_h) = \bs{0}$ implies
${(\LbH,\ubH,\pH,\JbH,\bbH,\rH) = \bs{0}}$. To begin,
set $(\ubhH, \phH, \bbhH, \rhH, \gb, \fb, \rho_h) = \bs{0}$. Take $(\Gb,\vb,\q,\Jb,\cb,\s) = (\LbH,\ubH,\pH,\JbH,\bbH,\rH)$,
integrate by parts the first four terms in \eqref{eq:local_2_1} and
the first term in \eqref{eq:local_5_1}, and sum the resulting equations to get 
\begin{align}
  \label{eq:sum_energy2_1}
  &\Rey\nor{\LbH}_{0,\K}^2 
  + \LRa{ \LRp{\alpha_1 + \frac{m}{2}}\ubH, \ubH }_\pK \\
  &\quad + \frac{\Rm}{\kappa}\nor{\JbH}_{0,\K}^2  
    + \LRa{\LRp{\beta_1\opT+\beta_2\opN}\bbH,\bbH}_{\pK}
  = 0. \notag
\end{align}
Recalling $\alpha_1 > \half \nor{\wb}_\Linfty$ and $\beta_1\opT+\beta_2\opN>0$, we can yield
\begin{align*}
  \LbH = \bs{0}, \quad   \JbH = \bs{0}, \quad \text{ in } \K; \qquad 
  \ubH = \bs{0}, \quad   \bbH = \bs{0}, \quad \text{ on } \pK.
\end{align*}
Using an argument similar to that in Section \secref{wellposednessEHDG}
we can conclude $\ubH = \bbH = \bs{0}$ in $\K$. From
\eqref{eq:local_2_1} and \eqref{eq:local_5_1}, we have:
\begin{equation}
    -(\pH, \div\vb)_\K = 0,\quad\forall\vb\in\VbM\LRp{K},
\end{equation}
and
\begin{equation}
    -(\rH, \div\cb)_\K= 0,\quad\forall\cb\in\CbM\LRp{K},
\end{equation}
respectively. Since the space $\LRc{q : q=\div\vb,\,\forall\vb\in\VbM\LRp{K}}\supseteq\QbM\LRp{K}$ and\\
$\LRc{s : s=\div\cb,\,\forall\cb\in\CbM\LRp{K}}\supseteq\SbM\LRp{K}$, we can pick $\div\vb=\pH$ and $\div\cb=\rH$ and conclude that $\pH=\rH=0$ in $\K$.
\end{proof}

\subsection{Conservation properties of the E-HDG method}
\seclab{conservationEHDG}
In this section, we prove that our method is divergence-free and $\Hsp\LRp{div}$-conforming for both velocity (i.e., the exactness of mass conservation) and magnetic 
 (i.e., the absence of magnetic monopoles) fields. Same conclusions can be drawn for the HDG version in \cref{rema:HDG}.

\begin{propo}[divergence-free property and $\Hsp\LRp{div}$-conformity for the velocity field]\propolab{mass_conservation}
Let $\ubH\in\VbM$ and $\ubhH\in\MubM$ be the solution to the proposed E-HDG discretization \eqnref{local}-\eqnref{global_5_1}, then 
\begin{subequations}
    \begin{align}
        \eqnlab{div_u_in_L2}
        &\Div\LRp{\eval{\ubH}_{\K}}=0,
        && \forall\K\in\Omegah;\\
        \eqnlab{jump_u_o}
        & \eval{\jump{\ubH\cdot\n}}_{\e}=0,
        && \forall\e\in\Gho.\\
        \eqnlab{jump_u_b}
        & \ubH\cdot\n=\ubhH\cdot\n,
        && \text{on }\e\text{ and }\forall\e\in\Ghb.
    \end{align}
\end{subequations}
\end{propo}
\begin{proof}
    Apply integration-by-parts to Eq.  \eqnref{local_3_1}:
    \beq\eqnlab{mass_conservation_proof1}
        \LRp{\Div\LRp{\eval{\ubH}_{\K}}, \q}_\K = 0,\quad\forall\q\in\QbM(\K),\,\forall\K\in\Omegah.
    \eeq
    Since $\Div\LRp{\eval{\ubH}_{\K}}\in\QbM(\K)$, we can take $\q=\Div\LRp{\eval{\ubH}_{\K}}$, yielding $\norm{\Div\LRp{\eval{\ubH}_{\K}}}_{0,\K}^2=0$, which implies that $\Div\LRp{\eval{\ubH}_{\K}}=0$ for all $\K\in\Omegah$. 
    It follows from Eq. \eqref{global} that:
    \beq\eqnlab{mass_conservation_proof2}
        \LRa{\jump{\FbhH^3 \cdot \n},\rho}_\ed
        =\LRa{\jump{\ubH \cdot \n},\rho}_\ed
        =0,\quad\forall\rho\in\RhoM(\e),\,\forall\e\in\Gho.
    \eeq
    Since $\eval{\jump{\ubH \cdot \n}}_{\e}\in\RhoM(\e)$, we can take $\rho=\jump{\ubH \cdot \n}$, yielding $\norm{\jump{\ubH \cdot \n}}_{0,\e}^2=0$ for all $\e\in\Gho$. Thus, $\eval{\jump{\ubH\cdot\n}}_{\e}=0$ for all $\e\in\Gho$. The proof of Eq. \eqnref{jump_u_b} follows the same argument with the aid of Eq. \eqnref{Constraint}.  
\end{proof}

\begin{propo}[divergence-free property and $\Hsp\LRp{div}$-conformity for the magnetic field]\propolab{no_monopole}
Let $\bbH\in\CbM$ and $\bbhH\in\LambM$ be the solution to the proposed E-HDG discretization \eqnref{local}-\eqnref{global_5_1}, then 
\begin{subequations}
    \begin{align}
        \eqnlab{div_b_in_L2}
        &\Div\LRp{\eval{\bbH}_{\K}}=0,
        && \forall\K\in\Omegah;\\
        \eqnlab{jump_b_o}
        & \eval{\jump{\bbH\cdot\n}}_{\e}=0,
        && \forall\e\in\Gho.\\
        \eqnlab{jump_b_b}
        & \bbH\cdot\n=\bbhH\cdot\n,
        && \text{on }\e\text{ and }\forall\e\in\Ghb.
    \end{align}
\end{subequations}  
\end{propo}
\begin{proof}
    The result holds by directly following the similar argument as the proof of Proposition \proporef{mass_conservation}.
\end{proof}

\begin{rema}
As can be seen, both Propositions \proporef{mass_conservation} and \proporef{no_monopole} also hold true for the nonlinear case. That is, they are still valid if $\wb$ and $\db$ are replaced by $\ubH$ and $\bbH$ in \eqnref{local}-\eqref{global}.
\end{rema}


\section{Numerical Results}\seclab{numerical_results}
A nonlinear solver can be constructed through the employment of the linear E-HDG (or HDG in \cref{rema:HDG}) scheme given by \eqnref{local}-\eqnref{global_5_1} in a  Picard iteration. If we consider the linearized MHD equations \eqnref{mhdlin} to be a linear map $\LRp{\wb,\db}\mapsto\LRp{\ub,\bb}$, then any fixed point of that mapping is a solution to the nonlinear incompressible viso-resistive MHD equations \eqnref{mhd_nonlin}. With this in mind, we can use the general linearized incompressible MHD E-HDG scheme \eqnref{local}--\eqnref{global_5_1} in an iterative manner to numerically solve the nonlinear incompressible MHD equations. The convergence of such an interaction can be consulted from \cite{muralikrishnan_multilevel_2023}. Let the superscript denote an iteration number, we set the initial guess $\ubH^0=\bs{0}$ and $\bbH^0=\bs{0}$ and the stopping criterion 
\beq
\max\LRc{\frac{\norm{\ubH^i-\ubH^{i-1}}_0}{\norm{\ubH^{i}}_0},\frac{\norm{\bbH^i-\bbH^{i-1}}_0}{\norm{\bbH^{i}}_0}}<\varepsilon,
\eeq
where $\varepsilon$ is a user-defined tolerance. In particular, we take $\varepsilon=O(10^{-10})$ in all numerical experiments for the nonlinear examples.

In this section, a series of numerical experiments is presented to illustrate the capability of the E-HDG method in both linear and nonlinear scenarios. First, a comparison is drawn between the proposed HDG and E-HDG methods regarding the DOFs and the actual computational time (wall-clock time). Then the order of accuracy for the linear scheme is numerically investigated by applying the E-HDG method to two- and three-dimensional problems with smooth solutions. The convergence of a two-dimensional singular problem, defined on a nonconvex domain, is also presented. Moreover, the pressure-robustness of our method is numerically demonstrated by perturbing smooth manufactured solutions. Finally, the order of accuracy for the nonlinear solver, where the linear scheme is integrated into a Picard iteration, is studied through two- and three-dimensional problems featuring smooth solutions, including a stationary liquid duct flow in plasma physics and manufactured solutions. It should be emphasized that the divergence-free property and $\Hsp{}(div)$-conformity still obviously hold for our nonlinear solver and will be validated through numerical demonstrations.

Our methods\textemdash both HDG and E-HDG)\textemdash are implemented based on the Modular Finite Element Method (MFEM) library \cite{mfem}. Furthermore, we use the direct solver of MUMPS \cite{MUMPS2001,MUMPS2006} through PETSc \cite{balay_efficient_1997,balay_petsctao_2023} to solve the systems of linear equations composed by the Schur complement (or static condensation) resulting from the discretization \eqnref{local}--\eqnref{global_5_1}. In addition, we take stabilization parameters $\alpha_1\in\LRc{125,1000}$ and $\beta_1=\beta_2\in\LRc{1,100,1000}$. Although it is proved that the well-posedness of both local and global solvers can be guaranteed by the conditions $\alpha_1>\half \nor{\wb}_\Linfty$, and $\beta_1\opT+\beta_2\opN>0$, we numerical found that small increments in the values of the stabilization parameters can improve the order of accuracy. However, large values of the parameters can cause serious adverse effects in convergence. 

\begin{rema}
    In this work, the auxiliary variables $\LbH$ and $\JbH$ can be locally eliminated through local Eq. \eqnref{local_1_1} and \eqnref{local_4_1}, respectively. Since the numerical flux defined in \eqnref{local_1_1} only associates with a single global variable $\ubhH$, the local variable $\LbH$ can be expressed by $\ubH$ and $\ubhH$,
    thanks to the block diagonal structure endowed by the term  $\Rey\LRp{\LbH, \Gb}_\K$. A similar procedure can also be followed to express $\JbH$ by $\bbH$ and $\bbhH$ with the help of Eq. \eqnref{local_4_1}.
    Through the elimination, the assembly operation (construct the local Schur complement and allocate it to the global matrix) and reconstruction operation (solve for the local variables with the given global variables) can be computationally cheaper.
\end{rema}
\begin{rema}
    Even though the well-posedness of the method is proved in Theorem \theoref{global}, the inclusion of the pressure constraint given in \eqnref{global_5_1} is not straightforward to implement. Note that the discretization is ill-posed without the pressure constraint, and the local variable $\pH$ and global variable $\phH$ can only be determined up to constant. Such a singular system can still be handled by Krylov type of iterative solver without encountering breakdowns \cite{benzi_golub_liesen_2005, Howard2014}. However, in order to use a direct solver, an additional treatment is necessary. In this paper, we restrict one DOF of the global variable $\phH$ to be zero such that both $\pH$ and $\phH$ can be determined. Once the system is solved by the direct solver, we then enforce the pressure constraint \eqnref{global_5_1} by post-processing.            
\end{rema}
\begin{rema}
    All $\Lsp^{\infty}$-norms are computed as the maximum norm of the function values evaluated on all elements using a set of quadrature points with the order of accuracy $2k+3$.
\end{rema}

\subsection{Computational Performance of the proposed HDG and E-HDG methods}\seclab{performance}
In this subsection, we discuss the computational costs of the HDG and the E-HDG methods in which the discretization is based on \eqnref{local}-\eqnref{global_5_1} but with different trace approximation spaces (see Remark \remaref{HDG}). Table \tabref{DOF} summarizes the DOFs needed by the HDG and E-HDG methods, and Table \tabref{time} summarizes the corresponding computational time. 
The values presented in each cell of Table \tabref{time} denote the total wall-clock time spent by the entire process. This includes the three main tasks: the assembly (locally constructing the Schur complement and allocating it to the global matrix), the solution of the system of equations (obtaining the global variables), and the local reconstruction (recovering the local variables from the given global variables through the solution of the local equations \eqnref{local}). The measurements are based on the average of five runs, with each run recording the maximal time among all MPI processes.  

The reduction in DOFs becomes notably more pronounced for three-dimensional cases, particularly on finer meshes. For example, applying the E-HDG method with $k=1$ on a mesh comprising 24576 elements results in a maximum DOF reduction of up to 72.58\%. This reduction is directly reflected in the computational time, see Table \tabref{time}, where a 47.74\% saving in total computational time is achieved. However, on coarser meshes, despite substantial reductions in DOFs, the corresponding savings in computational time are limited (perhaps due to the efficiency of MUMPS \cite{MUMPS2001,MUMPS2006}). This discrepancy can be explained through Table \tabref{solver_time} and Table \tabref{other_time}. The former delineates the wall-clock time spent by the linear solver, while the latter encapsulates the times allocated to the assembly and the local reconstruction tasks. Analysis of Table \tabref{other_time} reveals that the times devoted to assembly and local reconstruction remain nearly constant irrespective of mesh refinement, approximation degree, or dimension. On the other hand, the reduction trend in total computational time presented in Table \tabref{time} aligns closely with the computational time required by the linear solver detailed in Table \tabref{solver_time}. This alignment suggests that the advantage of downsizing DOFs may become more substantial when the linear solver time dominates the overall computational time. In essence, while reducing DOFs may not significantly impact the assembly and reconstruction times for the HDG and E-HDG methods, it notably enhances the efficiency of the linear solver in the E-HDG method for larger problems.

In addition to the reduction on computational time, reducing DOFs also adds advantages in memory management and this can be seen in Table \tabref{time}. 
On the three-dimensional mesh consisting of 24576 elements, the linear solver fails when using the HDG methods along with $k=3$ and $k=4$ due to insufficient memory\footnote{Such breakdown can be avoided by using an iterative solver. However, the design of a preconditioned iterative solver is beyond the scope of this paper, and hence we will pursue this in our future work (see also our previous work in \cite{muralikrishnan_multilevel_2023}).}. In contrast, such challenges can be overcome by using the E-HDG method, where the linear solver remains operational under identical circumstances.  

\begin{table}[!htb]
\centering
\begin{subtable}[t]{0.49\textwidth}
\resizebox{\textwidth}{!}{
\begin{tabular}{|c|cccc|}
\multicolumn{5}{c}{\textbf{Two-dimension}} \\
\hline
\hline
\multicolumn{5}{c}{DOFs used in the HDG method} \\
\hline
elem. \#&
$k=1$ & 
$k=2$ & 
$k=3$ & 
$k=4$ \\
\hline
2 & 60 &	90 &	120 &	150 	\\ 
8 & 192 &	288 &	384 &	480 	\\ 
32 & 672 &	1.01E+03 &	1.34E+03 &	1.68E+03 	\\ 
128 & 2.50E+03 &	3.74E+03 &	4.99E+03 &	6.24E+03 	\\ 
512 & 9.60E+03 &	1.44E+04 &	1.92E+04 &	2.40E+04	\\
\hline
\multicolumn{5}{c}{DOFs used in the E-HDG method} \\
\hline
elem. \#&
$k=1$ & 
$k=2$ & 
$k=3$ & 
$k=4$ \\
\hline
2 & 36 &	66 &	96 &	126 	\\ 
8 & 100 &	196 &	292 &	388	\\ 
32 & 324 &	660 &	996 &	1.33E+03 	\\ 
128 & 1.16E+03 &	2.40E+03 &	3.65E+03 &	4.90E+03 	\\ 
512 & 4.36E+03 &	9.16E+03 &	1.40E+04 &	1.88E+04 	\\ 
\hline
\multicolumn{5}{c}{Percentage of reduction in DOFs (\%)} \\
\hline
elem. \# &
$k=1$ & 
$k=2$ & 
$k=3$ & 
$k=4$ \\
\hline
2 & -40.00 &	-26.67 &	-20.00 &	-16.00 	\\ 
8 & -47.92 &	-31.94 &	-23.96 &	-19.17 	\\ 
32 & -51.79 &	-34.52 &	-25.89 &	-20.71 	\\ 
128 & -53.69 &	-35.79 &	-26.84 &	-21.47 	\\ 
512 &- 54.62 &	-36.42 &	-27.31 &	-21.85 	\\ 
\hline
\end{tabular}
}
\caption*{}
\end{subtable}
\hspace{\fill}
\begin{subtable}[t]{0.49\textwidth}
\resizebox{\textwidth}{!}{
\begin{tabular}{|c|cccc|}
\multicolumn{5}{c}{\textbf{Three-dimension}} \\
\hline
\hline
\multicolumn{5}{c}{DOFs used in the HDG method} \\
\hline
elem. \#&
$k=1$ & 
$k=2$ & 
$k=3$ & 
$k=4$ \\
\hline
6 & 432 &	864 &	1.44E+03 &	2.16E+03 	\\ 
48& 2.88E+03 &	5.76E+03 &	9.60E+03 &	1.44E+04 	\\ 
364& 2.07E+04 &	4.15E+04 &	6.91E+04 &	1.04E+05 	\\ 
3072& 1.57E+05 &	3.13E+05 &	5.22E+05 &	7.83E+05 	\\ 
\textcolor{black}{24576}& \textcolor{black}{\textbf{1.22E+06}} &	\textcolor{black}{2.43E+06} &	\textcolor{black}{4.06E+06} &	\textcolor{black}{6.08E+06} 	\\ 
\hline
\multicolumn{5}{c}{DOFs used in the E-HDG method} \\
\hline
elem. \#&
$k=1$ & 
$k=2$ & 
$k=3$ & 
$k=4$ \\
\hline
6& 156 &	378 &	744 &	1.25E+03 	\\ 
48& 882 &	2.19E+03 &	4.46E+03 &	7.69E+03 	\\ 
364& 5.93E+03 &	1.47E+04 &	3.05E+04 &	5.31E+04 	\\ 
3072& 4.35E+04 &	1.08E+05 &	2.24E+05 &	3.93E+05 	\\ 
\textcolor{black}{24576}& \textcolor{black}{\textbf{3.34E+05}} &	\textcolor{black}{8.24E+05} &	\textcolor{black}{1.72E+06} &	\textcolor{black}{3.02E+06} 	\\ 
\hline
\multicolumn{5}{c}{Percentage of reduction in DOFs (\%)} \\
\hline
elem. \#&
$k=1$ & 
$k=2$ & 
$k=3$ & 
$k=4$ \\
\hline
6&    -63.89 &	-56.25 &	-48.33 &	-41.94 	\\ 
48&   -69.38 &	-61.98 &	-53.56 &	-46.62 	\\ 
364&  -71.38 &	-64.45 &	-55.93 &	-48.79 	\\ 
3072& -72.21 &	-65.59 &	-57.05 &	-49.83 	\\ 
\textcolor{black}{24576}& \textcolor{black}{\textbf{-72.58}} &	\textcolor{black}{-66.14} &	\textcolor{black}{-57.59} &	\textcolor{black}{-50.33} 	\\ 
\hline
\end{tabular}
}
\caption*{}
\end{subtable}
\vspace{-20pt}
\caption{The summary of DOFs used in E-HDG and HDG 
discretizations given by \eqnref{local}-\eqnref{global_5_1}. Note that $k$ denotes the degree of approximation and ``elem. \#" indicates the number of elements used in a given mesh.}\tablab{DOF}
\end{table}

\begin{table}[!htb]
\centering
\begin{subtable}[t]{0.49\textwidth}
\resizebox{\textwidth}{!}{
\begin{tabular}{|p{1.3cm}|p{1.2cm}p{1.2cm}p{1.2cm}p{1.2cm}|}
\multicolumn{5}{c}{\textbf{Two-dimension}} \\
\hline
\hline
\multicolumn{5}{c}{Total number of MPI processes}\\
\hline
elem. \#&
$k=1$ & 
$k=2$ & 
$k=3$ & 
$k=4$ \\
\hline
2   & 1 & 1 & 1 & 1 \\ 
8   & 1 & 1 & 1 & 1 \\ 
32  & 2 & 2 & 2 & 2 \\
128 & 2 & 2 & 2 & 2 \\
512 & 4 & 4 & 4 & 4  \\
\hline
\multicolumn{5}{c}{Total wall-clock time by the HDG method (sec)}\\
\hline
elem. \#&
$k=1$ & 
$k=2$ & 
$k=3$ & 
$k=4$ \\
\hline
2   & 0.02 & 0.03 & 0.07 & 0.18 \\ 
8   & 0.03 & 0.09 & 0.25 & 0.69 \\ 
32  & 0.05 & 0.18 & 0.50 & 1.39 \\
128 & 0.15 & 0.67 & 1.98 & 5.54 \\
512 & 0.31 & 1.39 & 4.01 & 11.29  \\
\hline
\multicolumn{5}{c}{Total wall-clock time by the E-HDG method (sec)} \\
\hline
elem. \#&
$k=1$ & 
$k=2$ & 
$k=3$ & 
$k=4$ \\
\hline
2   & 0.01 & 0.03 & 0.07 & 0.18 \\ 
8   & 0.02 & 0.09 & 0.25 & 0.69 \\ 
32  & 0.05 & 0.18 & 0.50 & 1.39 \\
128 & 0.14 & 0.66 & 1.96 & 5.51 \\
512 & 0.27 & 1.36 & 3.98 & 11.21 \\
\hline
\multicolumn{5}{c}{Reduction in total computational time (\%)} \\
\hline
elem. \# &
$k=1$ & 
$k=2$ & 
$k=3$ & 
$k=4$ \\
\hline
2	&	-50.00	&	0.00	&	0.00	&	0.00\\
8	&	-33.33	&	0.00	&	0.00	&	0.00\\
32	&	0.00	&	0.00	&	0.00	&	0.00\\
128	&	-6.67	&	-1.49	&	-1.01	&	-0.54\\
512	&	-12.90	&	-2.16	&	-0.75	&	-0.71\\
\hline
\end{tabular}
}
\caption*{}
\end{subtable}
\hspace{\fill}
\begin{subtable}[t]{0.49\textwidth}
\resizebox{\textwidth}{!}{
\begin{tabular}{|p{1.3cm}|p{1.2cm}p{1.2cm}p{1.2cm}p{1.2cm}|}
\multicolumn{5}{c}{\textbf{Three-dimension}} \\
\hline
\hline
\multicolumn{5}{c}{Total number of MPI processes}\\
\hline
elem. \#&
$k=1$ & 
$k=2$ & 
$k=3$ & 
$k=4$ \\
\hline
6     & 1 & 1 & 1 & 2 \\ 
48    & 1 & 2 & 2 & 2 \\ 
364   & 2 & 4 & 4 & 8 \\
3072  & 2 & 8 & 8 & 16 \\
24576 & \textbf{4} & 16 & 16 & 32  \\
\hline
\multicolumn{5}{c}{Total wall-clock time by the HDG method (sec)} \\
\hline
elem. \#&
$k=1$ & 
$k=2$ & 
$k=3$ & 
$k=4$ \\
\hline
6     & 0.17   & 1.81   & 11.06  & 25.04\\ 
48    & 1.17   & 7.04   & 44.79  & 197.09\\ 
364   & 4.96   & 29.47  & 182.88 & 412.77\\
3072  & 43.90  & 127.43 & 783.20 & 1726.52\\
\textcolor{black}{24576} & \textcolor{black}{\textbf{303.48}} & \textcolor{black}{879.87} & \textcolor{black}{-}      & \textcolor{black}{-}\\
\hline
\multicolumn{5}{c}{Total wall-clock time by the E-HDG method (sec)} \\
\hline
elem. \#&
$k=1$ & 
$k=2$ & 
$k=3$ & 
$k=4$ \\
\hline
6     & 0.16   & 1.77   & 11.03   & 25.05	 \\ 
48    & 1.14   & 6.93   & 45.03   & 198.85 	 \\ 
364   & 4.65   & 28.49  & 182.44  & 404.38   \\
3072  & 37.91  & 117.19 & 739.76  & 1650.05   \\
\textcolor{black}{24576} & \textcolor{black}{\textbf{158.59}} & \textcolor{black}{522.89} & \textcolor{black}{3341.03} & \textcolor{black}{7473.77}   \\
\hline
\multicolumn{5}{c}{Reduction in total computational time (\%)} \\
\hline
elem. \#&
$k=1$ & 
$k=2$ & 
$k=3$ & 
$k=4$ \\
\hline
6	&	-5.88	&	-2.21	&	-0.27	&	0.04\\
48	&	-2.56	&	-1.56	&	0.54	&	0.89\\
364	&	-6.25	&	-3.33	&	-0.24	&	-2.03\\
3072	&	-13.64	&	-8.04	&	-5.55	&	-4.43\\
\textcolor{black}{24576}	&	\textcolor{black}{\textbf{-47.74}}	&	\textcolor{black}{-40.57}	&	\textcolor{black}{-}	&	\textcolor{black}{-}\\
\hline
\end{tabular}
}
\caption*{}
\end{subtable}
\vspace{-20pt}
\caption{The summary of total computational time (the averaged maximum of wall-clock time over five runs of identical setting, among all MPI processes) taken by E-HDG and HDG methods to solve two- and three-dimensional problems with the discretization given in \eqnref{local}-\eqnref{global_5_1}. The two-dimensional problem is the one presented in Section \secref{lin_smooth_2d} with $\Rey=\Rm=1$ and the three-dimensional problem is the one presented in Section \secref{lin_smooth_3d} with $\Rey=\Rm=1$. Note that $k$ denotes the degree of approximation and ``elem. \#" indicates the number of elements used in a given mesh.}\tablab{time}
\end{table}

\begin{table}[!htb]
\centering
\begin{subtable}[t]{0.49\textwidth}
\resizebox{\textwidth}{!}{
\begin{tabular}{|p{1.5cm}|p{1.5cm}p{1.5cm}p{1.5cm}p{1.5cm}|}
\multicolumn{5}{c}{\textbf{Two-dimension}} \\
\hline
\hline
\multicolumn{5}{c}{Total number of MPI processes}\\
\hline
elem. \#&
$k=1$ & 
$k=2$ & 
$k=3$ & 
$k=4$ \\
\hline
2   & 1 & 1 & 1 & 1 \\ 
8   & 1 & 1 & 1 & 1 \\ 
32  & 2 & 2 & 2 & 2 \\
128 & 2 & 2 & 2 & 2 \\
512 & 4 & 4 & 4 & 4  \\
\hline
\multicolumn{5}{c}{Wall-clock time of linear solver}\\
\multicolumn{5}{c}{in the HDG method (sec)}\\
\hline
elem. \#&
$k=1$ & 
$k=2$ & 
$k=3$ & 
$k=4$ \\
\hline
2	&	0.01	&	0.01	&	0.01	&	0.01\\
8	&	0.01	&	0.01	&	0.01	&	0.01\\
32	&	0.01	&	0.02	&	0.02	&	0.02\\
128	&	0.03	&	0.04	&	0.06	&	0.09\\
512	&	0.06	&	0.11	&	0.17	&	0.25\\
\hline
\multicolumn{5}{c}{Wall-clock time of linear solver}\\
\multicolumn{5}{c}{in the E-HDG method (sec)}\\
\hline
elem. \#&
$k=1$ & 
$k=2$ & 
$k=3$ & 
$k=4$ \\
\hline
2	&	0.01	&	0.01	&	0.01	&	0.01\\
8	&	0.01	&	0.01	&	0.01	&	0.01\\
32	&	0.01	&	0.01	&	0.02	&	0.02\\
128	&	0.02	&	0.03	&	0.05	&	0.07\\
512	&	0.03	&	0.08	&	0.13	&	0.20\\
\hline
\multicolumn{5}{c}{Reduction in computational time of linear solver (\%)} \\
\hline
elem. \# &
$k=1$ & 
$k=2$ & 
$k=3$ & 
$k=4$ \\
\hline
2	&	0.00	&	0.00	&	0.00	&	0.00\\
8	&	0.00	&	0.00	&	0.00	&	0.00\\
32	&	0.00	&	-50.00	&	0.00	&	0.00\\
128	&	-33.33	&	-25.00	&	-16.67	&	-22.22\\
512	&	-50.00	&	-27.27	&	-23.53	&	-20.00\\
\hline
\end{tabular}
}
\caption*{}
\end{subtable}
\hspace{\fill}
\begin{subtable}[t]{0.49\textwidth}
\resizebox{\textwidth}{!}{
\begin{tabular}{|p{1.5cm}|p{1.5cm}p{1.5cm}p{1.5cm}p{1.5cm}|}
\multicolumn{5}{c}{\textbf{Three-dimension}} \\
\hline
\hline
\multicolumn{5}{c}{Total number of MPI processes}\\
\hline
elem. \#&
$k=1$ & 
$k=2$ & 
$k=3$ & 
$k=4$ \\
\hline
6     & 1 & 1 & 1 & 2 \\ 
48    & 1 & 2 & 2 & 2 \\ 
364   & 2 & 4 & 4 & 8 \\
3072  & 2 & 8 & 8 & 16 \\
24576 & \textbf{4} & 16 & 16 & 32  \\
\hline
\multicolumn{5}{c}{Wall-clock time of linear solver}\\
\multicolumn{5}{c}{in the HDG method (sec)}\\
\hline
elem. \#&
$k=1$ & 
$k=2$ & 
$k=3$ & 
$k=4$ \\
\hline
6	&	0.01	&	0.05	&	0.11	&	0.42\\
48	&	0.04	&	0.13	&	0.83	&	0.98\\
364	&	0.41	&	1.26	&	3.51	&	5.93\\
3072	&	7.32	&	14.9	&	61.91	&	106.37\\
24576	&	\textbf{156.51}	&	419.73	&	-	&	-\\
\hline
\multicolumn{5}{c}{Wall-clock time of linear solver}\\
\multicolumn{5}{c}{in the E-HDG method (sec)}\\
\hline
elem. \#&
$k=1$ & 
$k=2$ & 
$k=3$ & 
$k=4$ \\
\hline
6	&	0.01	&	0.03	&	0.09	&	0.17\\
48	&	0.02	&	0.07	&	1.2	&	0.55\\
364	&	0.09	&	0.41	&	1.17	&	2.75\\
3072	&	0.97	&	3.52	&	16.48	&	37.67\\
24576	&	\textbf{13.08}	&	66.88	&	414.78	&	940.11\\
\hline
\multicolumn{5}{c}{Reduction in computational time of linear solver (\%)} \\
\hline
elem. \#&
$k=1$ & 
$k=2$ & 
$k=3$ & 
$k=4$ \\
\hline
6	&	0.00	&	-40.00	&	-18.18	&	-59.52\\
48	&	-50.00	&	-46.15	&	44.58	&	-43.88\\
364	&	-78.05	&	-67.46	&	-66.67	&	-53.63\\
3072	&	-86.75	&	-76.38	&	-73.38	&	-64.59\\
24576	&	\textbf{-91.64}	&	-84.07	&	-	&	-\\
\hline
\end{tabular}
}
\caption*{}
\end{subtable}
\vspace{-20pt}
\caption{The summary of computational time (the averaged maximum of wall-clock time over five runs of identical setting, among all MPI processes) taken by the linear solver for solving the two- and three-dimensional problems using E-HDG and HDG methods with the discretization given in \eqnref{local}-\eqnref{global_5_1}. The two-dimensional problem is the one presented in Section \secref{lin_smooth_2d} with $\Rey=\Rm=1$ and the three-dimensional problem is the one presented in Section \secref{lin_smooth_3d} with $\Rey=\Rm=1$. Note that $k$ denotes the degree of approximation and "elem. \#" indicates the number of elements used in a given mesh.}\tablab{solver_time}
\end{table}

\begin{table}[!htb]
\centering
\begin{subtable}[t]{0.49\textwidth}
\resizebox{\textwidth}{!}{
\begin{tabular}{|p{1.5cm}|p{1.5cm}p{1.5cm}p{1.5cm}p{1.5cm}|}
\multicolumn{5}{c}{\textbf{Two-dimension}} \\
\hline
\hline
\multicolumn{5}{c}{Total number of MPI processes}\\
\hline
elem. \#&
$k=1$ & 
$k=2$ & 
$k=3$ & 
$k=4$ \\
\hline
2   & 1 & 1 & 1 & 1 \\ 
8   & 1 & 1 & 1 & 1 \\ 
32  & 2 & 2 & 2 & 2 \\
128 & 2 & 2 & 2 & 2 \\
512 & 4 & 4 & 4 & 4  \\
\hline
\multicolumn{5}{c}{Wall-clock time of assembly \& reconstruction}\\
\multicolumn{5}{c}{in the HDG method (sec)}\\
\hline
elem. \#&
$k=1$ & 
$k=2$ & 
$k=3$ & 
$k=4$ \\
\hline
2	&	0.01	&	0.02	&	0.06	&	0.17\\
8	&	0.02	&	0.08	&	0.24	&	0.68\\
32	&	0.03	&	0.16	&	0.48	&	1.37\\
128	&	0.12	&	0.63	&	1.92	&	5.45\\
512	&	0.25	&	1.28	&	3.84	&	11.04\\
\hline
\multicolumn{5}{c}{Wall-clock time of assembly \& reconstruction}\\
\multicolumn{5}{c}{in the E-HDG method (sec)}\\
\hline
elem. \#&
$k=1$ & 
$k=2$ & 
$k=3$ & 
$k=4$ \\
\hline
2	&	0.01	&	0.03	&	0.06	&	0.17\\
8	&	0.02	&	0.08	&	0.24	&	0.68\\
32	&	0.03	&	0.16	&	0.48	&	1.37\\
128	&	0.12	&	0.63	&	1.91	&	5.43\\
512	&	0.24	&	1.28	&	3.85	&	11.01\\
\hline
\multicolumn{5}{c}{Reduction in computational time}\\
\multicolumn{5}{c}{of assembly \& reconstruction (\%)} \\
\hline
elem. \# &
$k=1$ & 
$k=2$ & 
$k=3$ & 
$k=4$ \\
\hline
2	&	0.00	&	50.00	&	0.00	&	0.00\\
8	&	0.00	&	0.00	&	0.00	&	0.00\\
32	&	0.00	&	0.00	&	0.00	&	0.00\\
128	&	0.00	&	0.00	&	-0.52	&	-0.37\\
512	&	-4.00	&	0.00	&	0.26	&	-0.27\\
\hline
\end{tabular}
}
\caption*{}
\end{subtable}
\hspace{\fill}
\begin{subtable}[t]{0.49\textwidth}
\resizebox{\textwidth}{!}{
\begin{tabular}{|p{1.5cm}|p{1.5cm}p{1.5cm}p{1.5cm}p{1.5cm}|}
\multicolumn{5}{c}{\textbf{Three-dimension}} \\
\hline
\hline
\multicolumn{5}{c}{Total number of MPI processes}\\
\hline
elem. \#&
$k=1$ & 
$k=2$ & 
$k=3$ & 
$k=4$ \\
\hline
6     & 1 & 1 & 1 & 2 \\ 
48    & 1 & 2 & 2 & 2 \\ 
364   & 2 & 4 & 4 & 8 \\
3072  & 2 & 8 & 8 & 16 \\
24576 & \textbf{4} & 16 & 16 & 32  \\
\hline
\multicolumn{5}{c}{Wall-clock time of assembly \& reconstruction}\\
\multicolumn{5}{c}{in the HDG method (sec)}\\
\hline
elem. \#&
$k=1$ & 
$k=2$ & 
$k=3$ & 
$k=4$ \\
\hline
6	&	0.15	&	1.77	&	10.95	&	24.62\\
48	&	1.13	&	6.91	&	43.96	&	196.12\\
364	&	4.54	&	28.21	&	179.37	&	406.84\\
3072	&	36.58	&	112.53	&	721.29	&	1620.16\\
24576	&	\textbf{146.97}	&	460.14	&	-	&	-\\
\hline
\multicolumn{5}{c}{Wall-clock time of assembly \& reconstruction}\\
\multicolumn{5}{c}{in the E-HDG method (sec)}\\
\hline
elem. \#&
$k=1$ & 
$k=2$ & 
$k=3$ & 
$k=4$ \\
\hline
6	&	0.15	&	1.74	&	10.93	&	24.89\\
48	&	1.12	&	6.86	&	43.82	&	198.3\\
364	&	4.56	&	28.08	&	181.27	&	401.64\\
3072	&	36.94	&	113.67	&	723.29	&	1612.38\\
24576	&	\textbf{145.51}	&	456.01	&	2926.25	&	6533.66\\
\hline
\multicolumn{5}{c}{Reduction in computational time}\\
\multicolumn{5}{c}{of assembly \& reconstruction (\%)} \\
\hline
elem. \#&
$k=1$ & 
$k=2$ & 
$k=3$ & 
$k=4$ \\
\hline
6	&	0.00	&	-1.69	&	-0.18	&	1.10\\
48	&	-0.88	&	-0.72	&	-0.32	&	1.11\\
364	&	0.44	&	-0.46	&	1.06	&	-1.28\\
3072	&	0.98	&	1.01	&	0.28	&	-0.48\\
24576	&	\textbf{-0.99}	&	-0.90	&	-	&	-\\
\hline
\end{tabular}
}
\caption*{}
\end{subtable}
\vspace{-20pt}
\caption{The summary of computational time (the averaged maximum of wall-clock time over five runs of identical setting, among all MPI processes) taken by assembly execution and local reconstruction for solving the two- and three-dimensional problems using E-HDG and HDG methods with the discretization given in \eqnref{local}-\eqnref{global_5_1}. The two-dimensional problem is the one presented in Section \secref{lin_smooth_2d} with $\Rey=\Rm=1$ and the three-dimensional problem is the one presented in Section \secref{lin_smooth_3d} with $\Rey=\Rm=1$. Note that $k$ denotes the degree of approximation and ``elem. \#" indicates the number of elements used in a given mesh.}\tablab{other_time}
\end{table}


\subsection{Linear examples}\seclab{linear}
A series of linear numerical experiments is carried out to verify our method in this subsection. We first analyze the accuracy and the convergence in two dimensions for the case of a smooth manufactured solution. In addition, the pressure robustness of our method is also tested. We then analyze the accuracy and convergence for a singular manufactured solution. Finally, we perform the analysis of the accuracy, convergence, and pressure robustness  for 
a smooth manufactured solution in three dimensions.

\subsubsection{Two-dimensional smooth manufactured solution}\seclab{lin_smooth_2d}
This example illustrates the convergence of the E-HDG scheme applied to a problem posed on the square domain $\Omega=(0,1)\times(0,1)$. In particular, the two-dimensional
manufactured vortex solution considered in \cite{gleason_divergence-conforming_2022} is adopted. We take $\Rey=\Rm\in\LRc{1,1000}$ and $\kappa=1$, and set $\gb$ and $\fb$ such that the manufactured solution for \eqnref{mhdlin-first}-\eqref{compatibility} is
\begin{subequations}\eqnlab{smooth_2d}
\begin{align}
    \begin{split}
        \bu = 
        \begin{pmatrix}
        -2x^2e^x\LRp{-y^2+y}\LRp{2y-1}\LRp{x-1}^2,\\
        -xy^2e^x\LRp{x\LRp{x+3}-2}\LRp{x-1}\LRp{y-1}^2
        \end{pmatrix},
    \end{split}\\
    \begin{split}
        \bb = 
        \begin{pmatrix}
        -2x^2e^x\LRp{-y^2+y}\LRp{2y-1}\LRp{x-1}^2,\\
        -xy^2e^x\LRp{x\LRp{x+3}-2}\LRp{x-1}\LRp{y-1}^2
        \end{pmatrix},        
    \end{split}\\
    \begin{split}
        \p = \p_0\sin{(\pi x)}\sin{(\pi y)},
    \end{split}\\
    \begin{split}
        \r = 0,
    \end{split}
\end{align}
\end{subequations}
with the prescribed fields $\wb = \bu$ and $\db = \bb$, and a constant $\p_0$. Table \tabref{lin_smooth_2d} shows the convergence rates for each local variable and the $\Lsp^{\infty}$-norm of the divergence errors, with the corresponding convergence plots in Figure \figref{lin_smooth_2d}.  Examining  Table \tabref{lin_smooth_2d} suggests that the increment in $\Rey$ and $\Rm$ improves the convergence rates of some local variables in this problem, notably $\LbH$, $\ubH$, and $\bbH$. For a more definitive assessment of convergence rates from the numerical experiment, we focus on the results corresponding to $\Rey=\Rm=1$.
In summary, we observe the super convergence rate of $\kbr+3/2$ for $\rH$, the optimal convergence rates of $\k+1$ for $\ubH,\bbH$, the optimal convergence rate of $\kbr+1$ for $\pH$, and sub-optimal convergence rates of $\k$ for $\LbH,\JbH$.

To numerically assess the pressure robustness of our method, we intentionally perturb the pressure solution. The test is carried out on two different meshes, one with 32 elements and another one with 512 elements, using polynomial degree $\k=2$ for both and a wide range of $\p_0$ values. The results of this study are presented in Table \tabref{pressure_robust_2d}. It is observed from the table that the $\Lsp^2$-errors of all local variables including the velocity and magnetic fields are independent of $\p_0$ regardless of which mesh is used. The observation implies that these errors do not depend on the pressure field and hence our method could be pressure robust. A particularly noteworthy discovery is the independence of the magnetic field error from the pressure field, a phenomenon previously observed in \cite{gleason_divergence-conforming_2022} as well. Plausible reasoning for this observation may stem from the absence of the pressure field in the magnetic induction equation presented in \eqnref{mhd2lin_1}.

\begin{table}[!htb]
\centering
\begin{tabular}{|c|cccccc|cc|}
\multicolumn{9}{c}{$\Rey=\Rm=1,\kappa=1$} \\
\hline
 &
\begin{tabular}{@{}c@{}} $\Rey\LbH$ \end{tabular} & 
\begin{tabular}{@{}c@{}} $\ubH$\end{tabular} & 
\begin{tabular}{@{}c@{}} $\pH$ \end{tabular} & 
\begin{tabular}{@{}c@{}} $\frac{\Rm}{\kappa}\JbH$\end{tabular} & 
\begin{tabular}{@{}c@{}} $\bbH$\end{tabular} & 
\begin{tabular}{@{}c@{}} $\rH$\end{tabular} & 
$\norm{\Div{\bu_h}}_{\infty}$ & $\norm{\Div{\bb_h}}_{\infty}$\\
\hline
$k=1$ & 1.02 &	2.29 &	1.12 &	1.20 &	2.40 &	1.96 &	3.85E-15 & 4.06E-15 \\ 
$k=2$ & 2.02 &	3.08 &	2.21 &	2.28 &	3.06 &	2.73 &	3.44E-14 & 2.71E-14 \\ 
$k=3$ & 3.04 &	4.06 &	3.17 &	3.35 &	4.03 &	3.75 &	7.94E-14 & 8.45E-14 \\ 
$k=4$ & 4.03 &	5.08 &	4.28 &	4.51 &	4.91 &	4.79 &	2.90E-13 & 4.55E-13 \\     
\hline
\multicolumn{9}{c}{$\Rey=\Rm=1000,\kappa=1$} \\
\hline
$k=1$ & 1.29 &	1.34 &	0.99 &	1.36 &	1.46 &	0.65 &	4.30E-15 & 2.54E-15 \\ 
$k=2$ & 2.93 &	4.05 &	2.02 &	3.01 &	4.14 &	2.40 &	2.50E-14 & 2.17E-14 \\ 
$k=3$ & 3.98 &	5.28 &	3.02 &	3.85 &	5.18 &	3.81 &	1.35E-12 & 1.07E-13 \\ 
$k=4$ & 4.16 &	5.21 &	4.00 &	4.17 &	5.20 &	4.39 &	2.67E-12 & 2.22E-12 \\  
\hline
\end{tabular}
\caption{Convergence rates of all local variables and divergence errors of velocity and magnetic fields for the E-HDG method applied to solve the two-dimensional problem with a smooth manufactured solution given in \eqnref{smooth_2d} with $\p_0=1$. The corresponding results are also presented in Figure \figref{lin_smooth_2d}. In this table, the convergence rates are evaluated at the last two data sets and the divergence errors are evaluated at the last data set.}\tablab{lin_smooth_2d}
\end{table}

\begin{table}[!htb]
\centering
\resizebox{\textwidth}{!}{
\begin{tabular}{|c|c|c|c|c|c|c|c|c|}
\multicolumn{9}{c}{32 elements in total, $h\approx1.46E-1$} \\
\hline
$p_0$ &
$\Rey\norm{\Lb-\LbH}_0$ &
$\norm{\ub-\ubH}_0$ &
\textcolor{black}{$\norm{\p-\pH}_0$} &
$\frac{\Rm}{\kappa}\norm{\Jb-\JbH}_0$ &
$\norm{\bb-\bbH}_0$ &
$\norm{\r-\rH}_0$ &
$\norm{\Div{\ub_h}}_{\infty}$ &
$\norm{\Div{\bb_h}}_{\infty}$ \\
\hline
1   & 2.09E-2  & 1.27E-3 & \textcolor{black}{5.57E-2} & 1.67E-2 & 9.66E-4 & 3.97E-2 & 7.15E-16 & 7.49E-16	\\
10  & 2.09E-2  & 1.27E-3 & \textcolor{black}{2.02E-1} & 1.67E-2 & 9.66E-4 & 3.97E-2 & 1.40E-15 & 6.11E-16	\\ 
25  & 2.09E-2  & 1.27E-3 & \textcolor{black}{4.90E-1} & 1.67E-2 & 9.66E-4 & 3.97E-2 & 2.78E-15 & 6.66E-16	\\  
100 & 2.09E-2  & 1.27E-3 & \textcolor{black}{1.95} & 1.67E-2 & 9.66E-4 & 3.97E-2 & 1.03E-14 & 6.38E-16	\\  
\hline
\multicolumn{9}{c}{512 elements in total, $h\approx3.66E-2$} \\
\hline
1   & 1.27E-3  & 1.09E-5 & \textcolor{black}{2.30E-3} & 7.48E-4 & 1.07E-5 & 1.40E-3 & 4.44E-15 & 3.77E-15	\\
10  & 1.27E-3  & 1.09E-5 & \textcolor{black}{1.26E-2} & 7.48E-4 & 1.07E-5 & 1.40E-3 & 7.41E-15 & 3.77E-15	\\ 
25  & 1.27E-3  & 1.09E-5 & \textcolor{black}{3.11E-2} & 7.48E-4 & 1.07E-5 & 1.40E-3 & 1.73E-14 & 4.05E-15	\\  
100 & 1.27E-3  & 1.09E-5 & \textcolor{black}{1.24E-1} & 7.48E-4 & 1.07E-5 & 1.40E-3 & 7.92E-14 & 4.11E-15	\\  
\hline
\end{tabular}
}
\caption{The errors in the local variables for the smooth manufactured solution given in \eqnref{smooth_2d} for meshes of 32 and 512 elements, a polynomial degree of $k=2$, and a range of $p_0$ values. The physical parameters are set to be $\Rey=\Rm=1$ and $\kappa=1$.}
\tablab{pressure_robust_2d}
\end{table}

\begin{figure*}
    \centering
    \resizebox{\textwidth}{!}{
    \begin{tabular}{ccc}
     \centered{\includegraphics{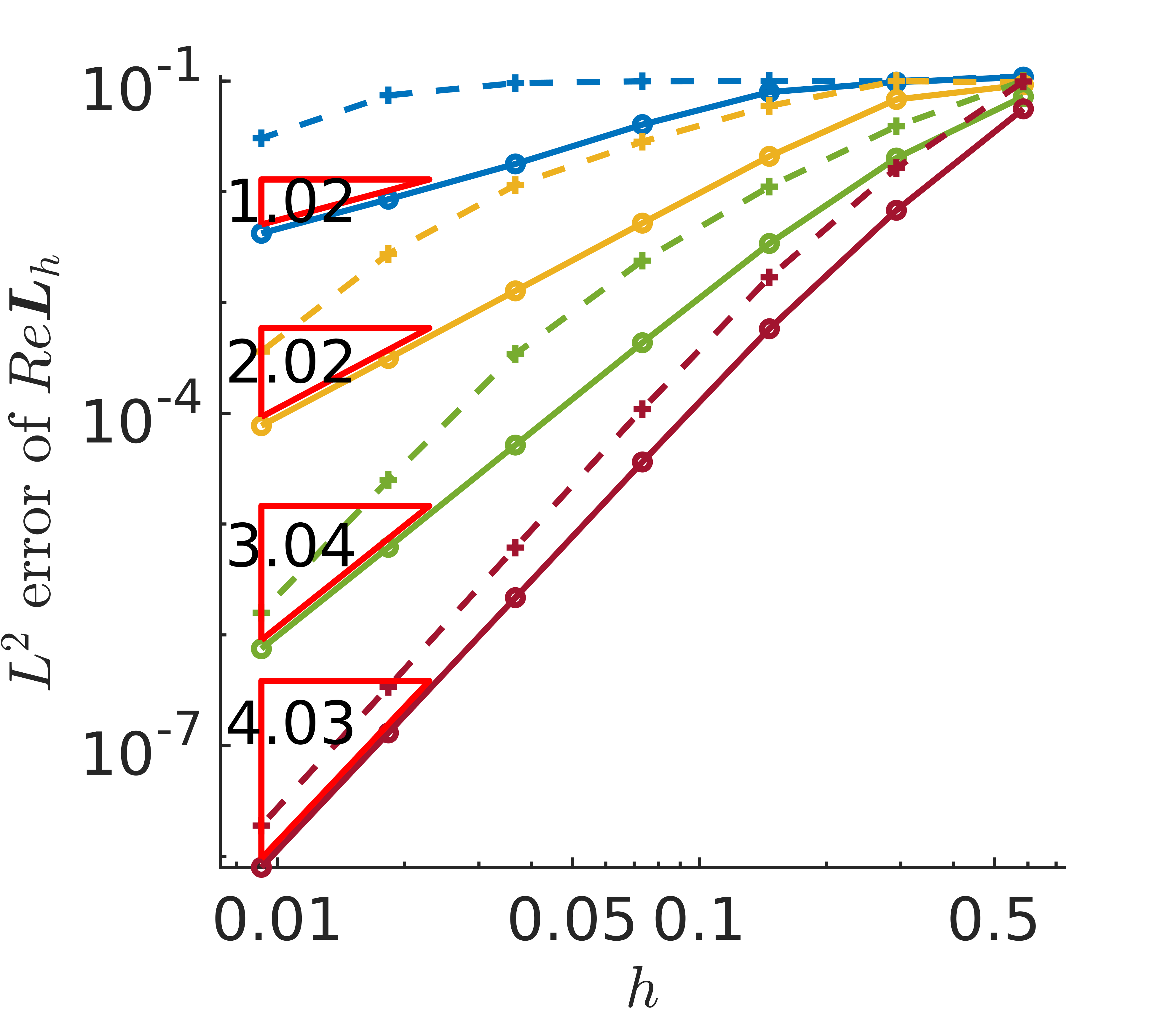} }& 
     \centered{\includegraphics{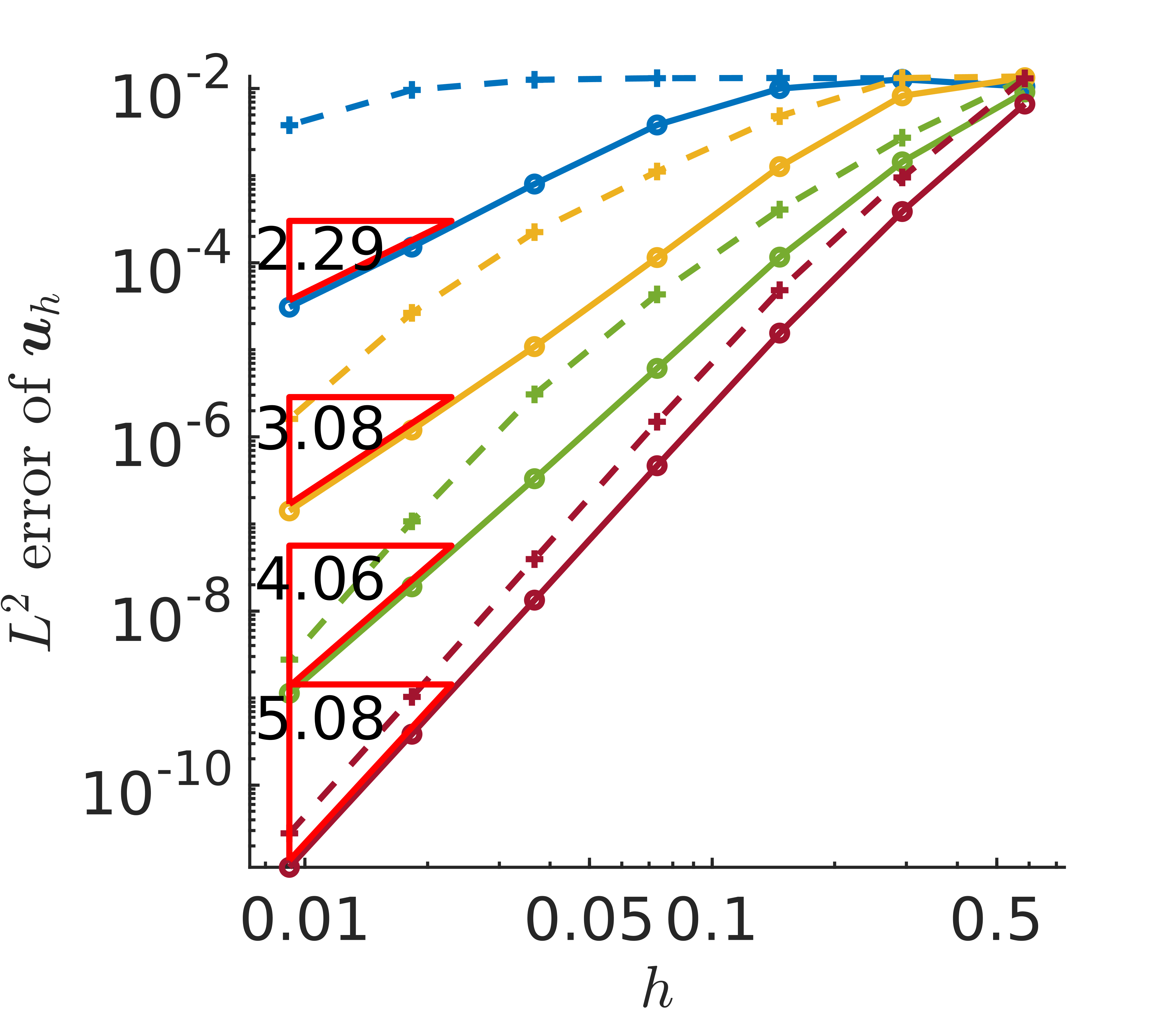}} &
     \centered{\includegraphics{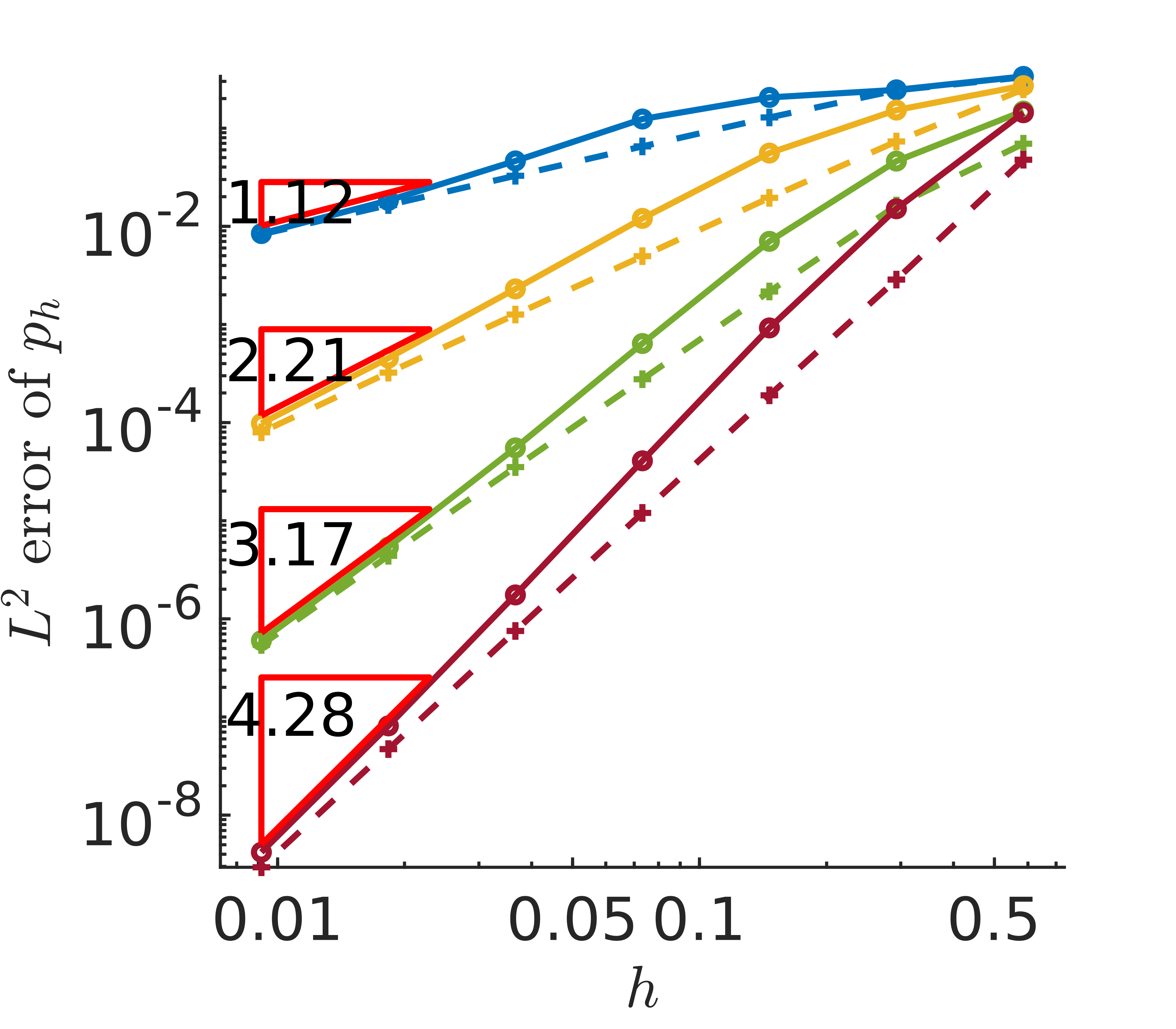}}\\
     \centered{\includegraphics{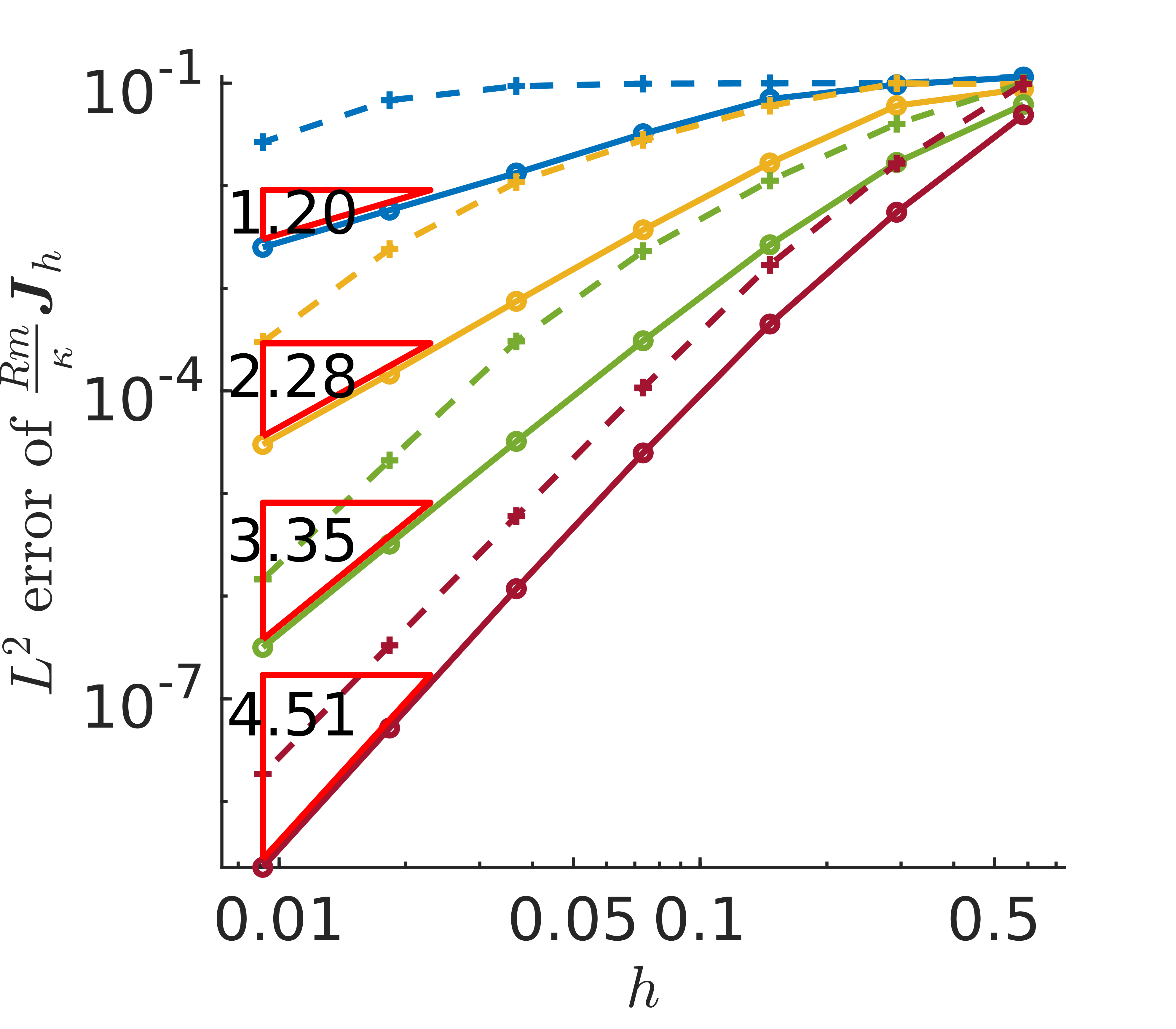}} &
     \centered{\includegraphics{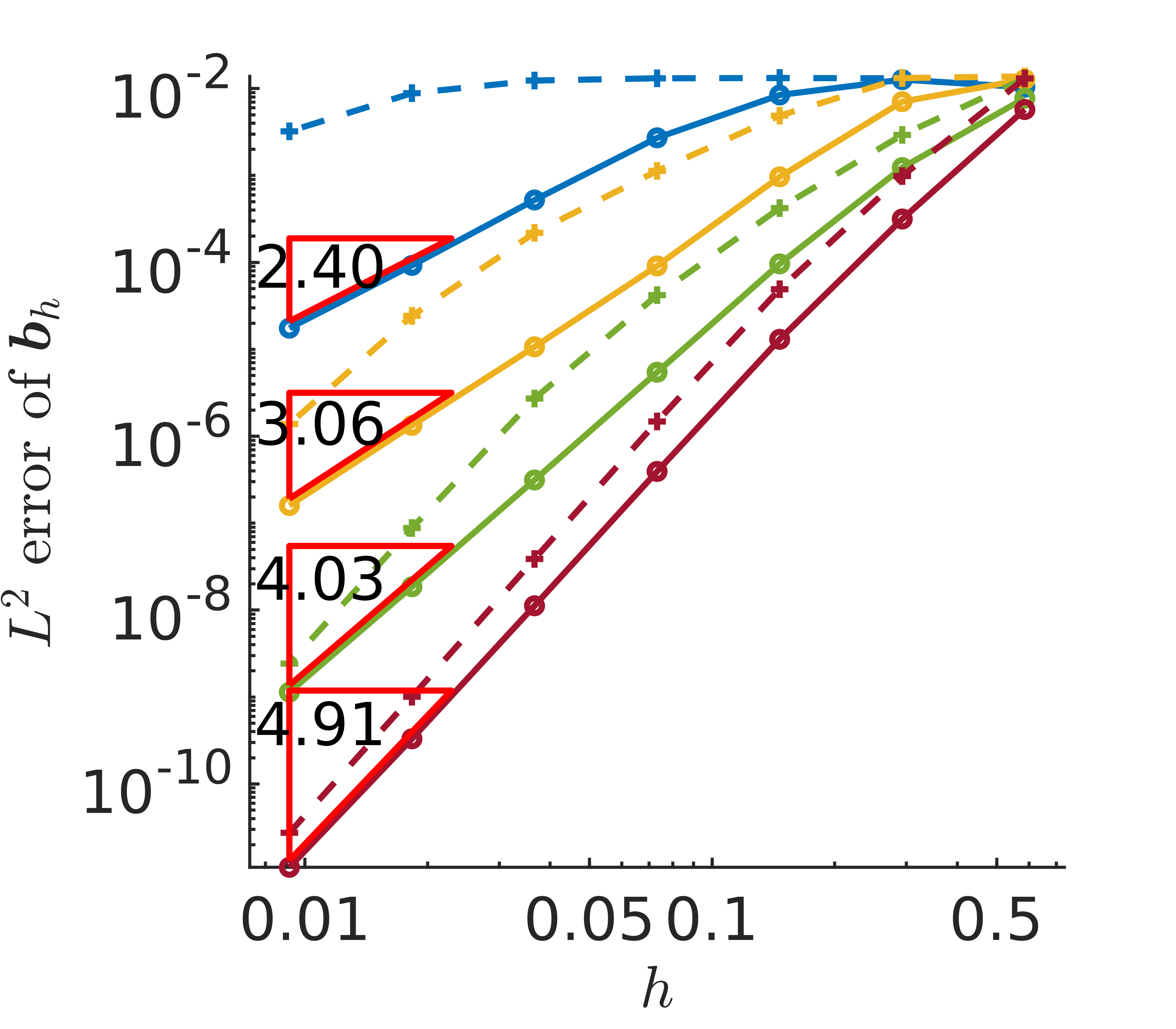}} &
     \centered{\includegraphics{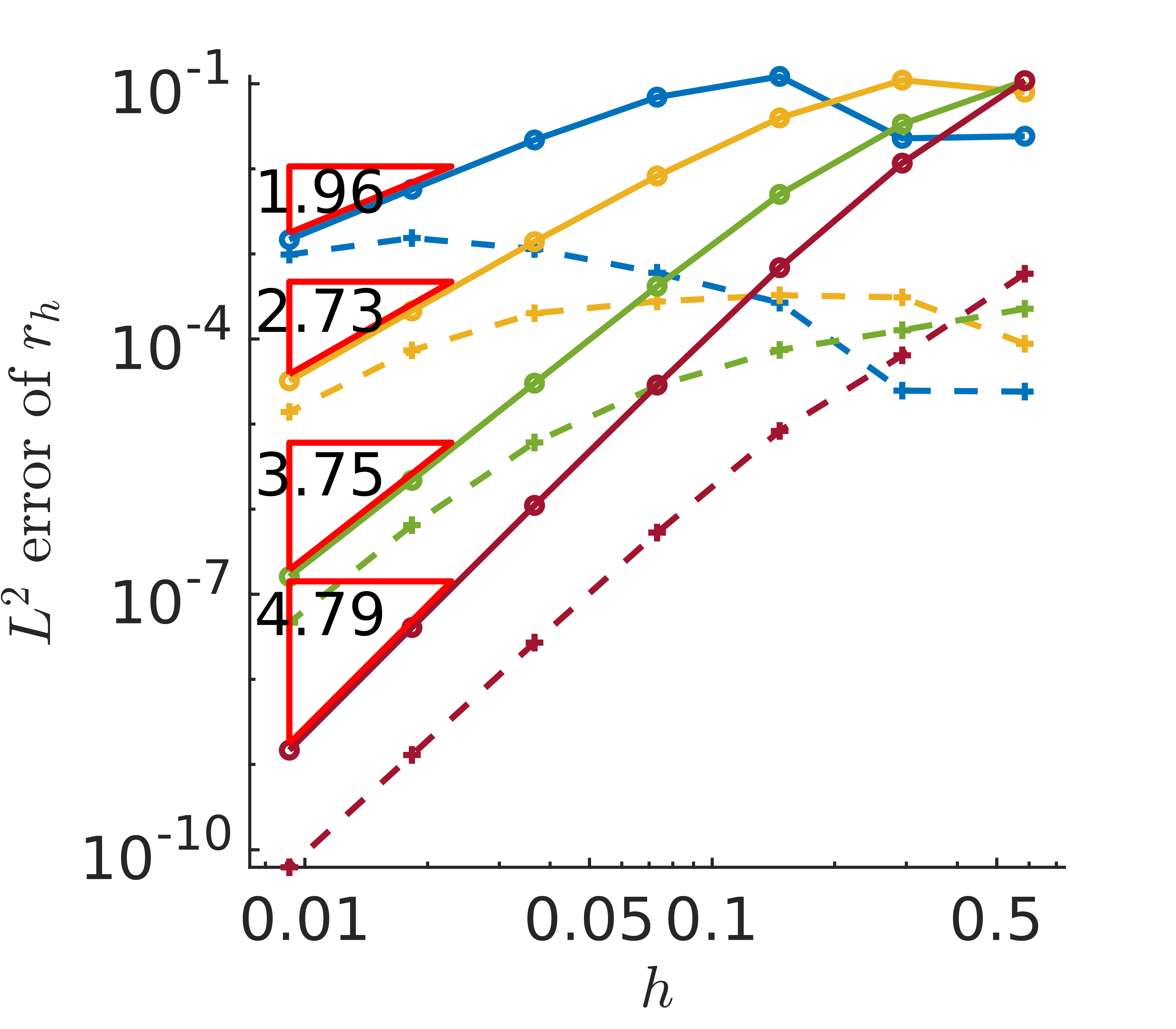}}\\
     \centered{\includegraphics{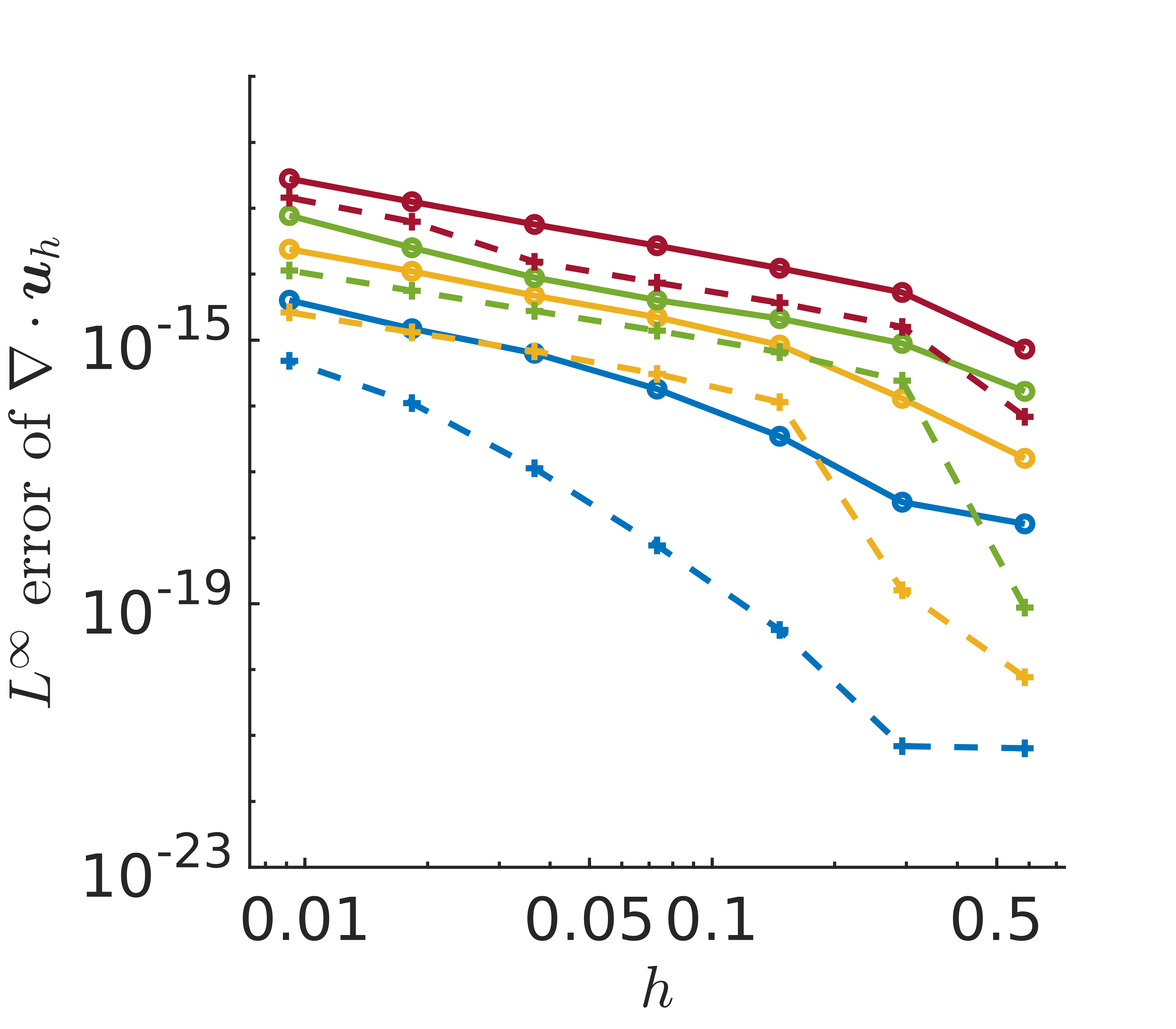}} &
     \centered{\includegraphics{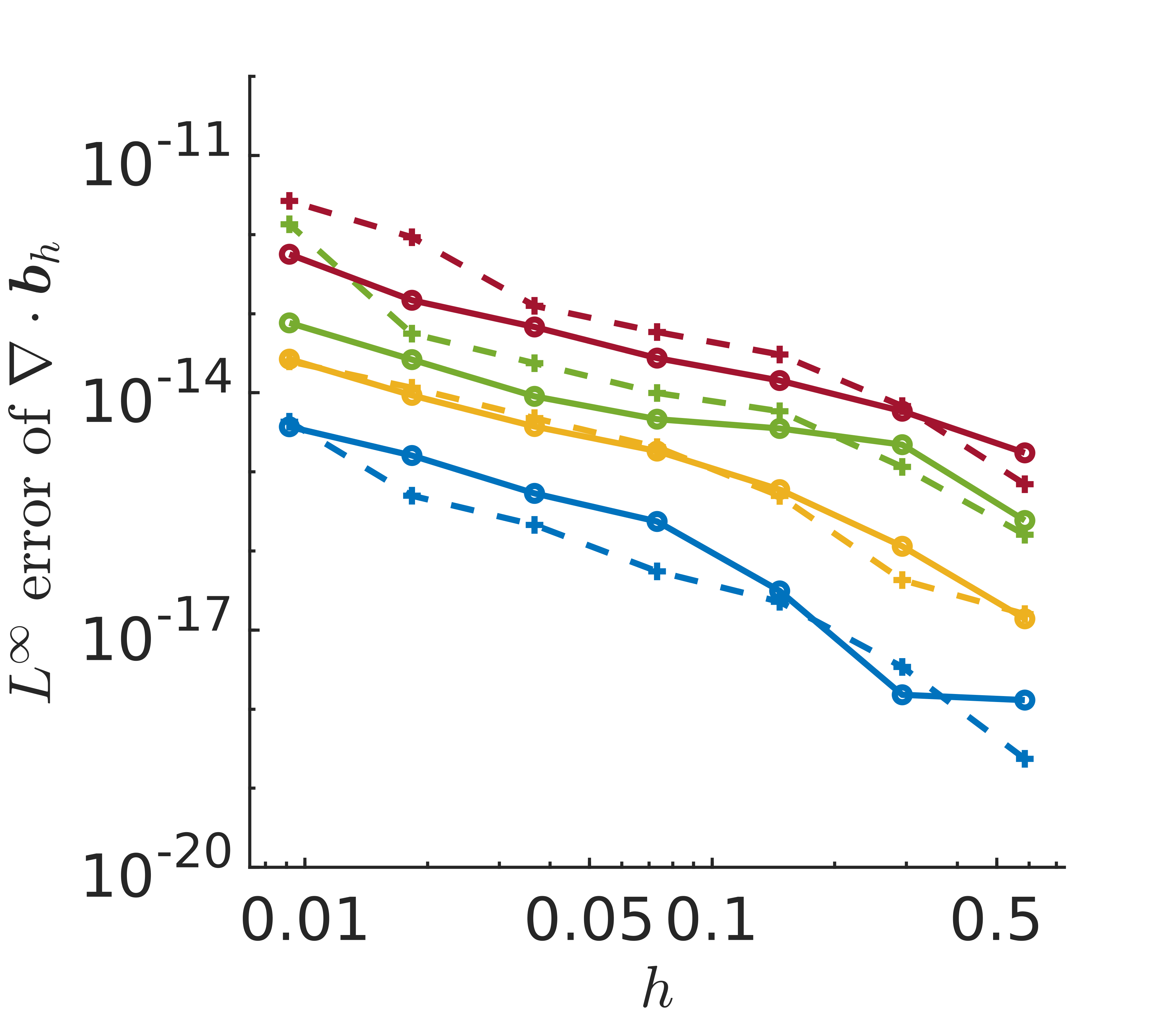}} &
     \centered{\includegraphics{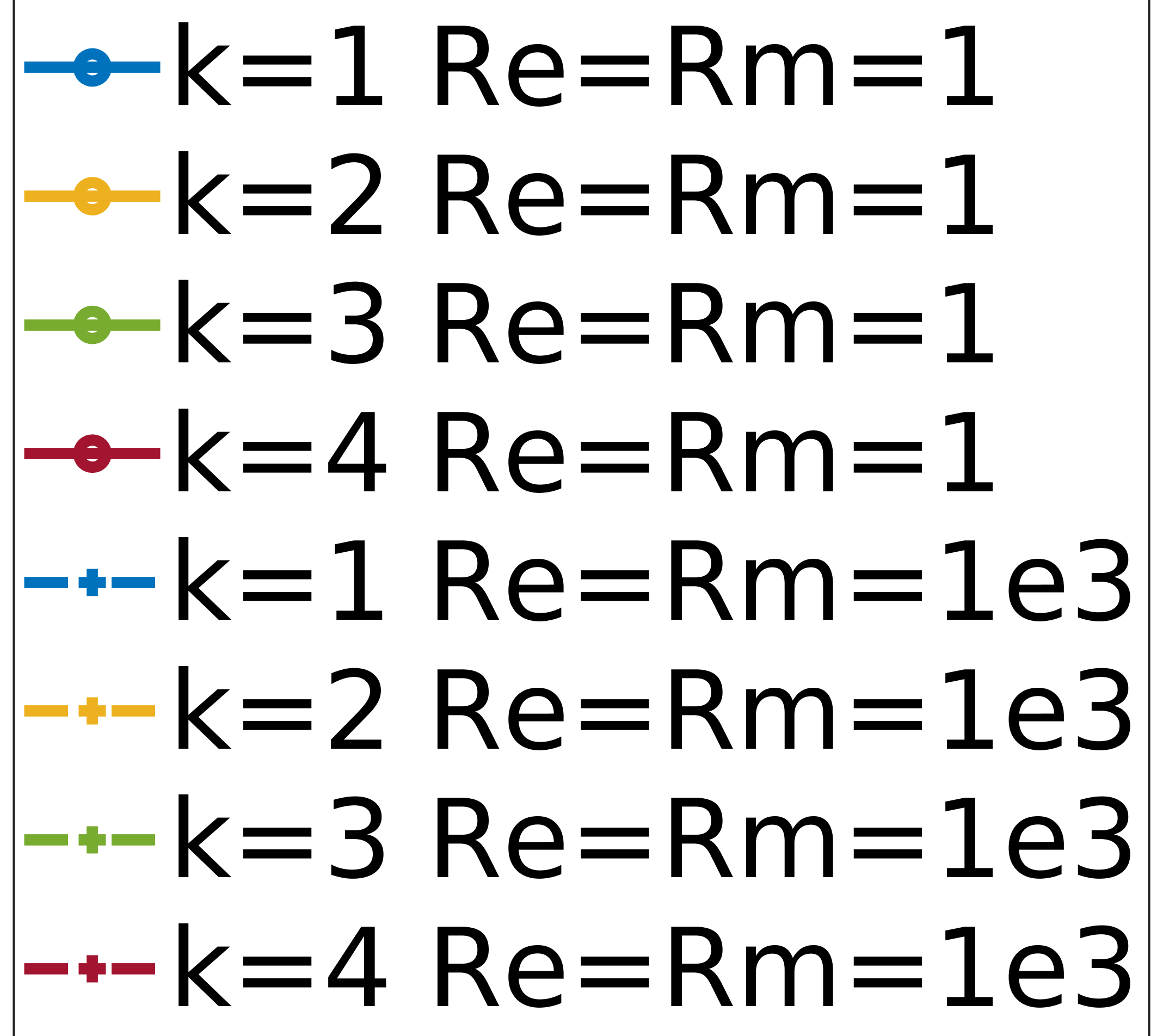}}
    \end{tabular}
    }
    \caption{Convergence histories of all local variables and divergence errors for the E-HDG method applied to solve the two-dimensional problem with a smooth manufactured solution given in \eqnref{smooth_2d} where we set $\p_0=1$. Only the convergence rates for $\Rey=\Rm=1$ are presented here.}\figlab{lin_smooth_2d}
\end{figure*}

\subsubsection{Two-dimensional singular manufactured solution}\seclab{lin_singular_2d}
To assess the robustness of our E-HDG scheme, we apply it to a problem where a strong singularity exists on the boundary. This example illustrates the convergence of the E-HDG scheme using a manufactured solution with a singularity (similar to the example in Section 5.2 of \cite{houston_mixed_2009} and Section 5.3 of \cite{lee_analysis_2019}). In particular, we consider a  nonconvex domain given by $\Omega=(-1,1)\times(-1,1)\backslash [0,1)\times(-1,0]$. We take $\Rey=\Rm=\kappa=1$, $\wb=\bs{0}$, and $\db=(-1,1)$. We pick $\gb$ and $\fb$ such that the analytical solution of \eqnref{mhdlin-first}-\eqref{compatibility} has the form 
\begin{subequations}\eqnlab{singular_2d}
\begin{align}
    \begin{split}
        \bu = 
        \begin{pmatrix}
        \rho^{\lambda}\LRs{(1+\lambda)\sin{(\phi)}\psi(\phi)+\cos{(\phi)}\psi'(\phi)},\\
        \rho^{\lambda}\LRs{-(1+\lambda)\cos{(\phi)}\psi(\phi)+\sin{(\phi)}\psi'(\phi)}
        \end{pmatrix}, 
    \end{split}\\
    \begin{split}
        \bb= \Grad{\LRp{\rho^{2/3}\sin{\LRp{\frac{2\phi}{3}}}}},
    \end{split}\\
    \begin{split}
        \p = -\rho^{\lambda-1}\frac{\LRp{1+\lambda}^2\psi'(\phi)+\psi'''(\phi)}{1-\lambda},
    \end{split}\\
    \begin{split}
        \r = 0,
    \end{split}
\end{align}
\end{subequations}
where
\begin{align*}
    &\psi(\phi) = \cos{(\lambda\omega)}\LRs{\frac{\sin{\LRp{(1+\lambda)\phi}}}{1+\lambda}-\frac{\sin{\LRp{(1-\lambda)\phi}}}{1-\lambda}}-\cos{\LRp{(1+\lambda)\phi}}+\cos{\LRp{(1-\lambda)\phi}},\\
    &\omega=\frac{3\pi}{2},\quad\lambda\approx0.54448373678246,\quad\phi\in\LRs{0,\frac{3\pi}{2}}.
\end{align*}
For this problem, it is known that $\ub\in\LRs{\Hsp^{1+\lambda}(\Omega)}^2$, $\p\in\Hsp^{\lambda}(\Omega)$, and $\bb\in\LRs{\Hsp^{2/3}(\Omega)}^2$, and the solution contains magnetic and hydrodynamic singularities that are among the strongest singularities \cite{houston_mixed_2009}.

Convergence results for this problem are summarized in Table \tabref{lin_singular_2d} and illustrated in Figure \figref{lin_singular_2d}. For the fluid variables $\LbH$, $\ubH$, and $\pH$, we observe convergence rates of approximately 2/3. In the case of magnetic variables, namely $\JbH$, $\bbH$, and $\rH$, the observed convergence rates are approximately 1/5, 2/3, and 1/3 respectively. Compared to the result presented in \cite{lee_analysis_2019}, the convergence rates of the fluid variables are lower, while the ones of the magnetic variables are similar. Remarkably, despite the accuracy challenges, divergence errors in both velocity and magnetic fields remain close to machine zero in this singular test case.


\begin{table}[!htb]
\centering
\begin{tabular}{|c|cccccc|cc|}
\multicolumn{9}{c}{$\Rey=\Rm=\kappa=1$} \\
\hline
 &
\begin{tabular}{@{}c@{}} $\Rey\LbH$ \end{tabular} & 
\begin{tabular}{@{}c@{}} $\ubH$\end{tabular} & 
\begin{tabular}{@{}c@{}} $\pH$ \end{tabular} & 
\begin{tabular}{@{}c@{}} $\frac{\Rm}{\kappa}\JbH$\end{tabular} & 
\begin{tabular}{@{}c@{}} $\bbH$\end{tabular} & 
\begin{tabular}{@{}c@{}} $\rH$\end{tabular} & 
$\norm{\Div{\bu_h}}_{\infty}$ & $\norm{\Div{\bb_h}}_{\infty}$\\
\hline
$k=1$ & 0.64 &	0.64 &	0.65 &	0.09 &	0.68 &	0.46 &	1.70E-12 & 3.82E-11 \\ 
$k=2$ & 0.63 &	0.63 &	0.65 &	0.12 &	0.65 &	0.31 &	7.57E-12 & 3.78E-10 \\ 
$k=3$ & 0.65 &	0.67 &	0.69 &	0.20 &	0.64 &	0.32 &	1.68E-11 & 8.57E-10 \\ 
$k=4$ & 0.66 &	0.68 &	0.71 &	0.26 &	0.62 &	0.36 &	3.68E-11 & 4.26E-09 \\    
\hline
\end{tabular}
\caption{Convergence rates of all local variables and divergence errors of velocity and magnetic fields for the E-HDG method applied to solve the two-dimensional problem with a singular manufactured solution given in \eqnref{singular_2d}. The corresponding results are also presented in Figure \figref{lin_singular_2d}. In this table, the convergence rates are evaluated at the last two data sets and the divergence errors are evaluated at the last data set.}\tablab{lin_singular_2d}
\end{table}

\begin{figure*}
    \centering
    \resizebox{\textwidth}{!}{
    \begin{tabular}{ccc}
     \centered{\includegraphics{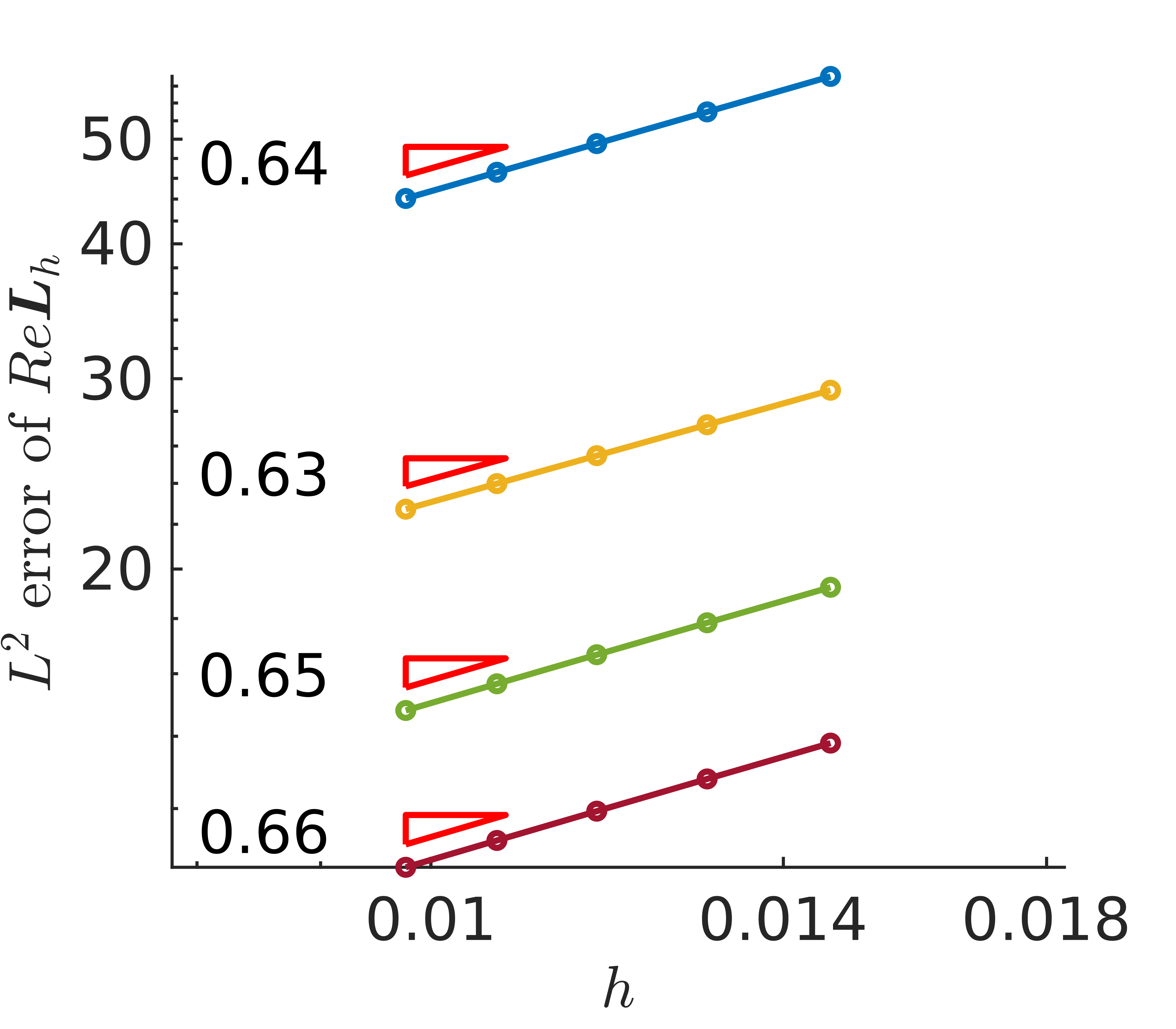} }& 
     \centered{\includegraphics{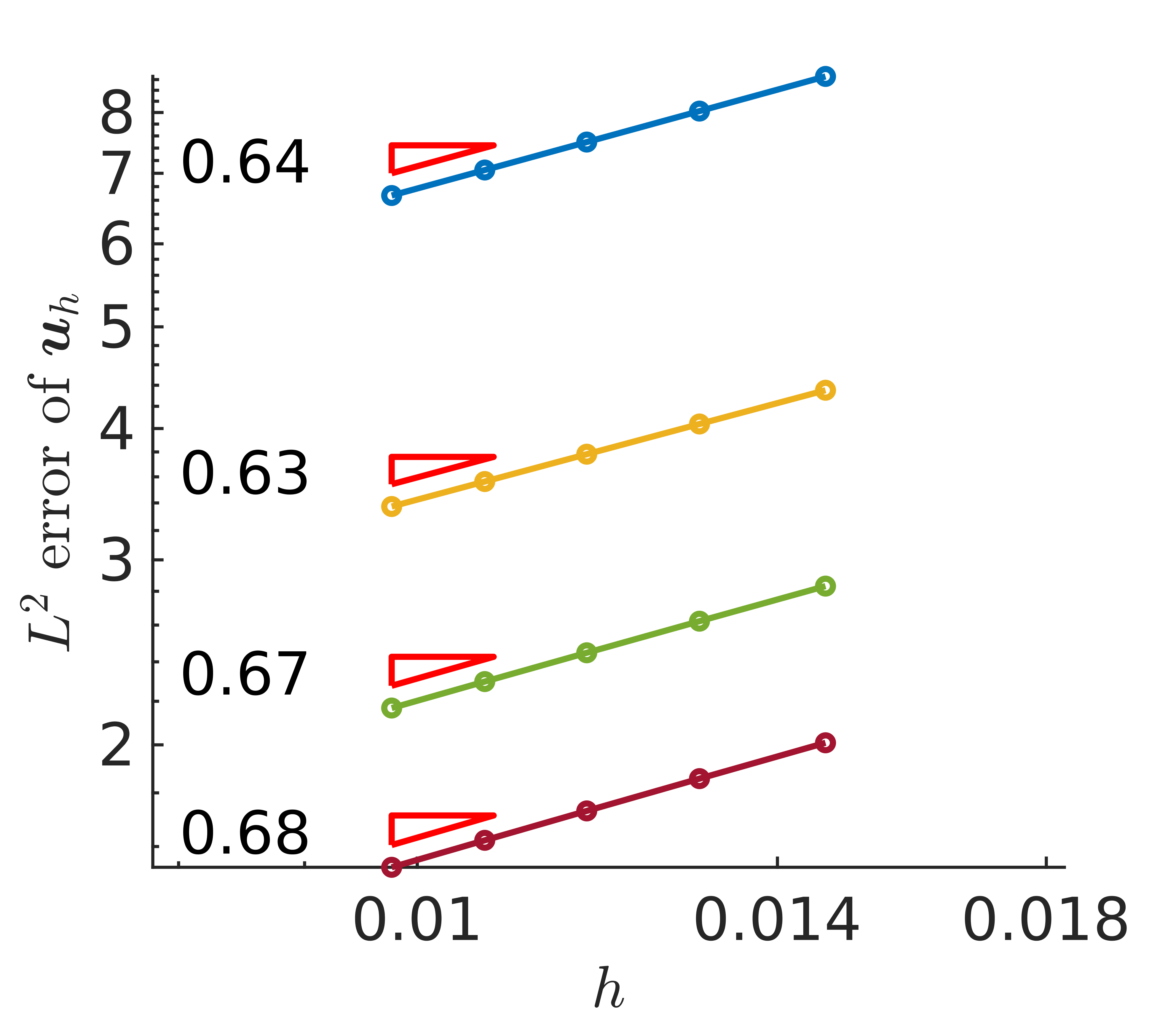}} &
     \centered{\includegraphics{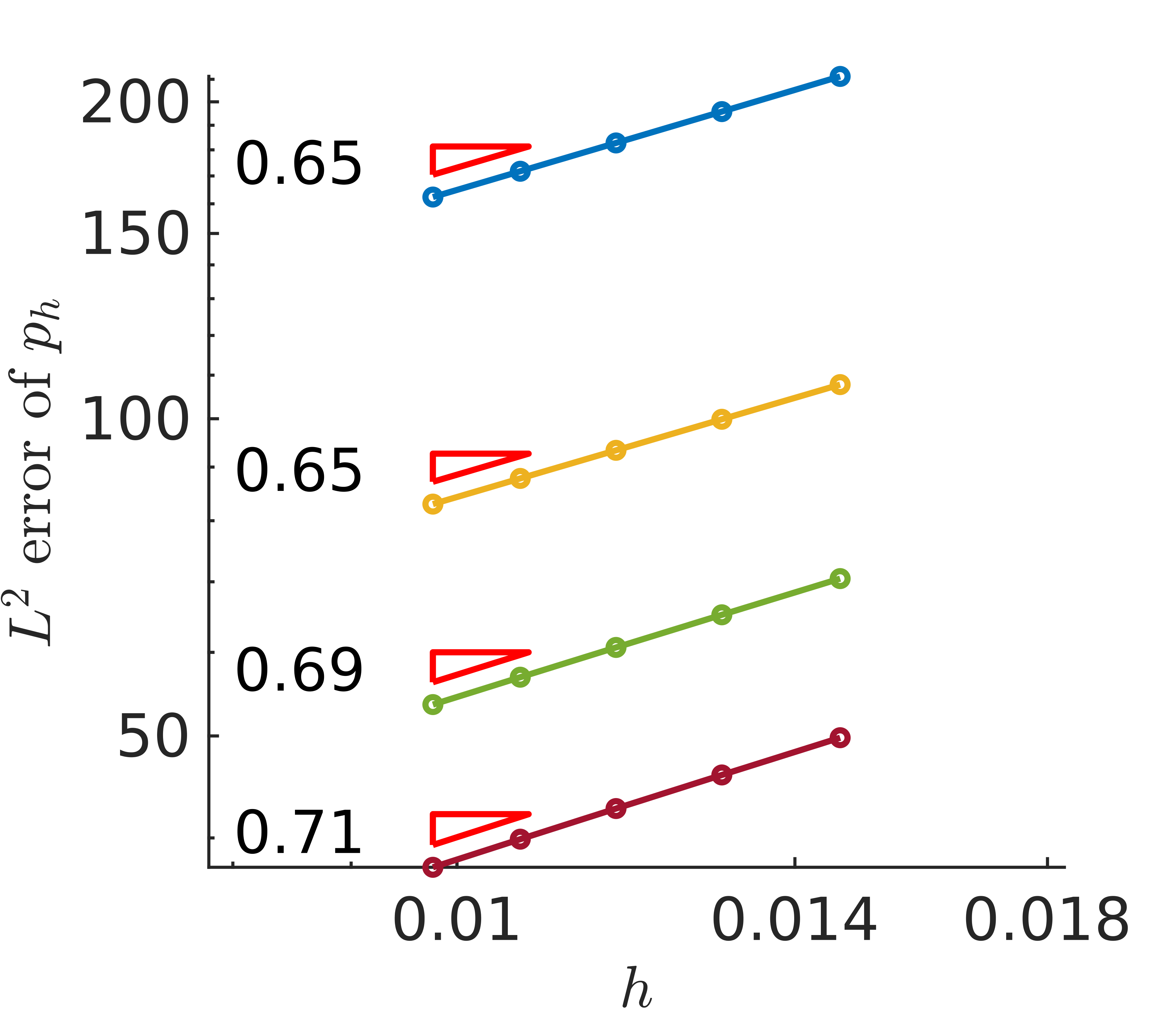}}\\
     \centered{\includegraphics{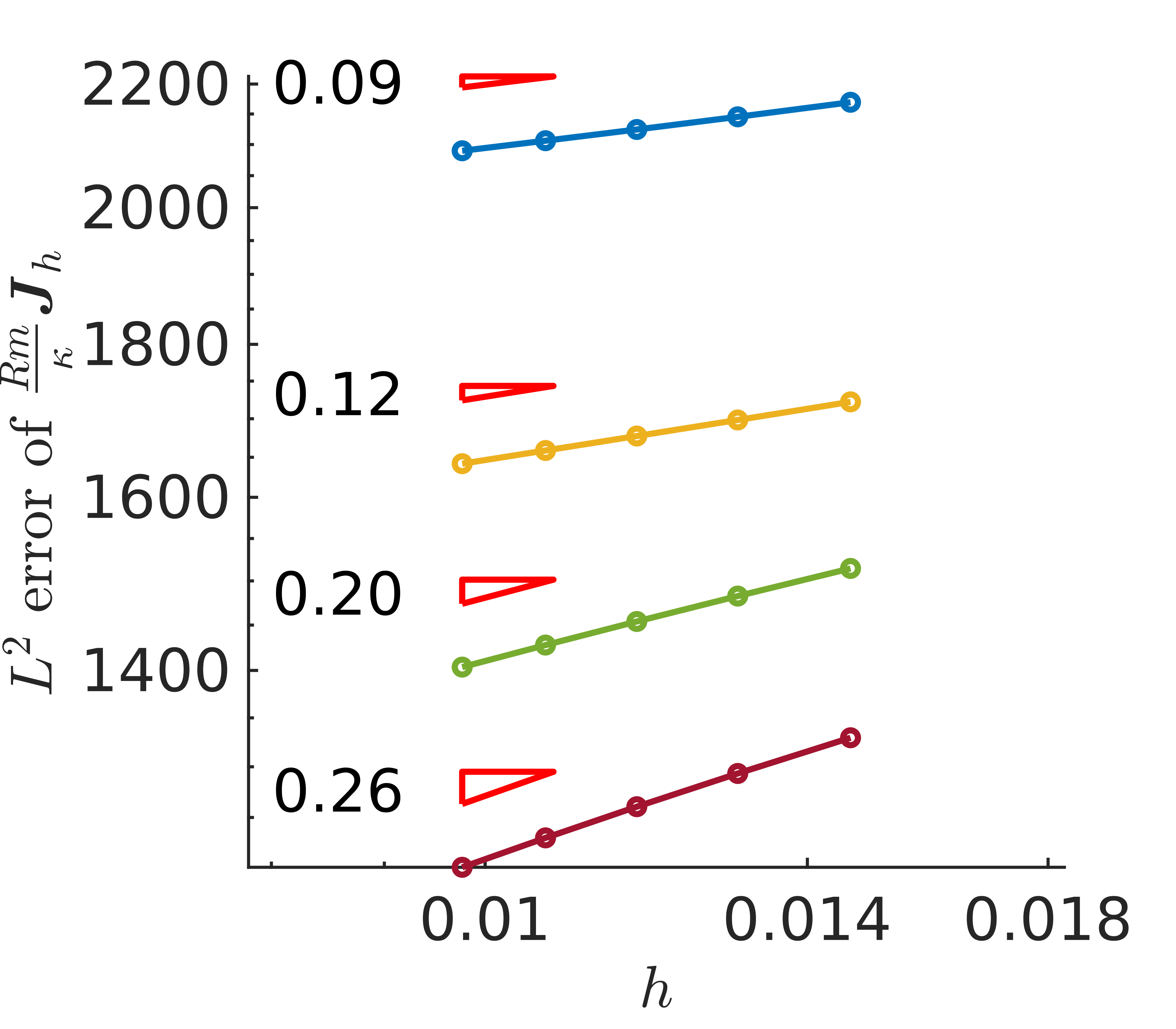}} &
     \centered{\includegraphics{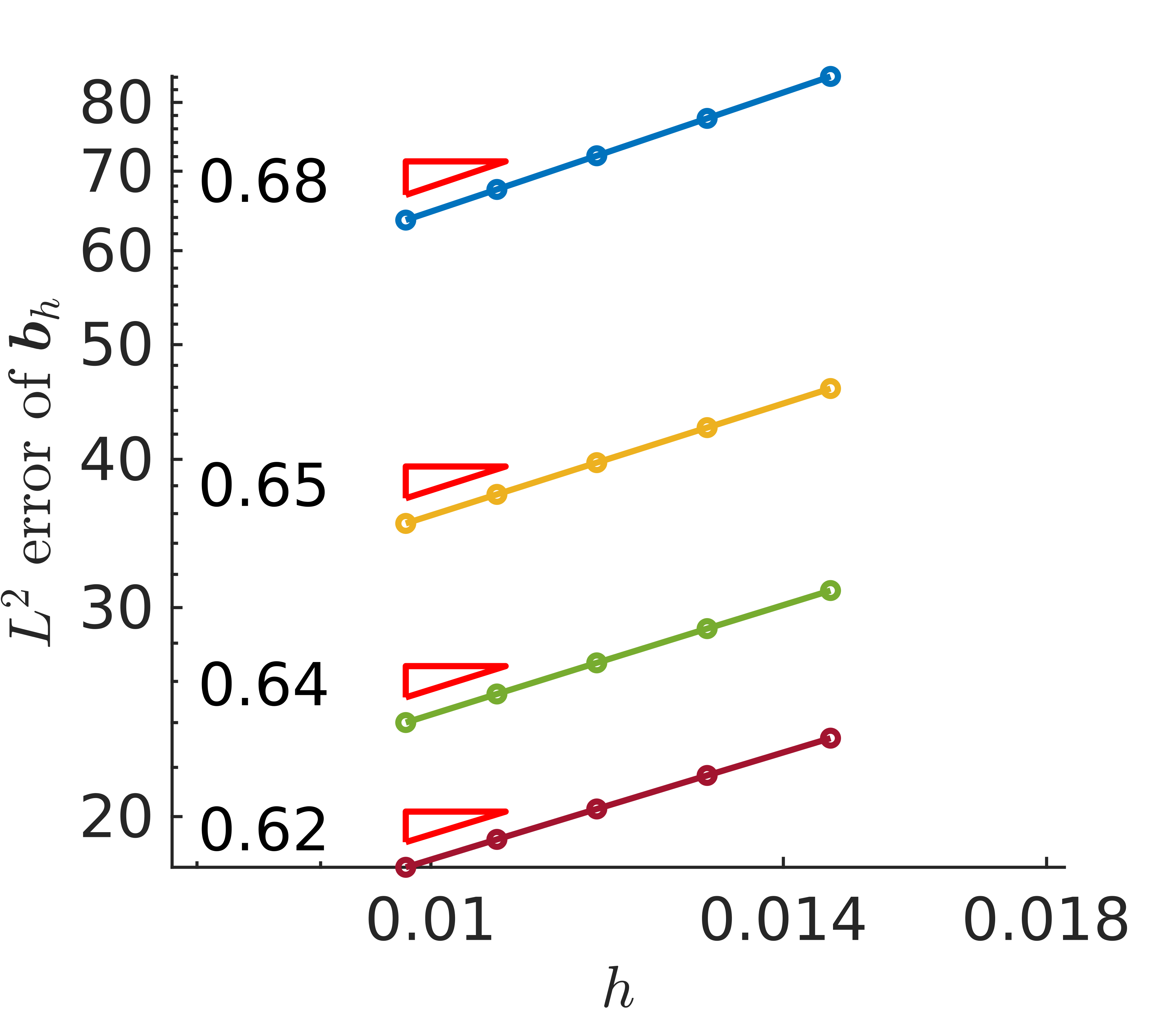}} &
     \centered{\includegraphics{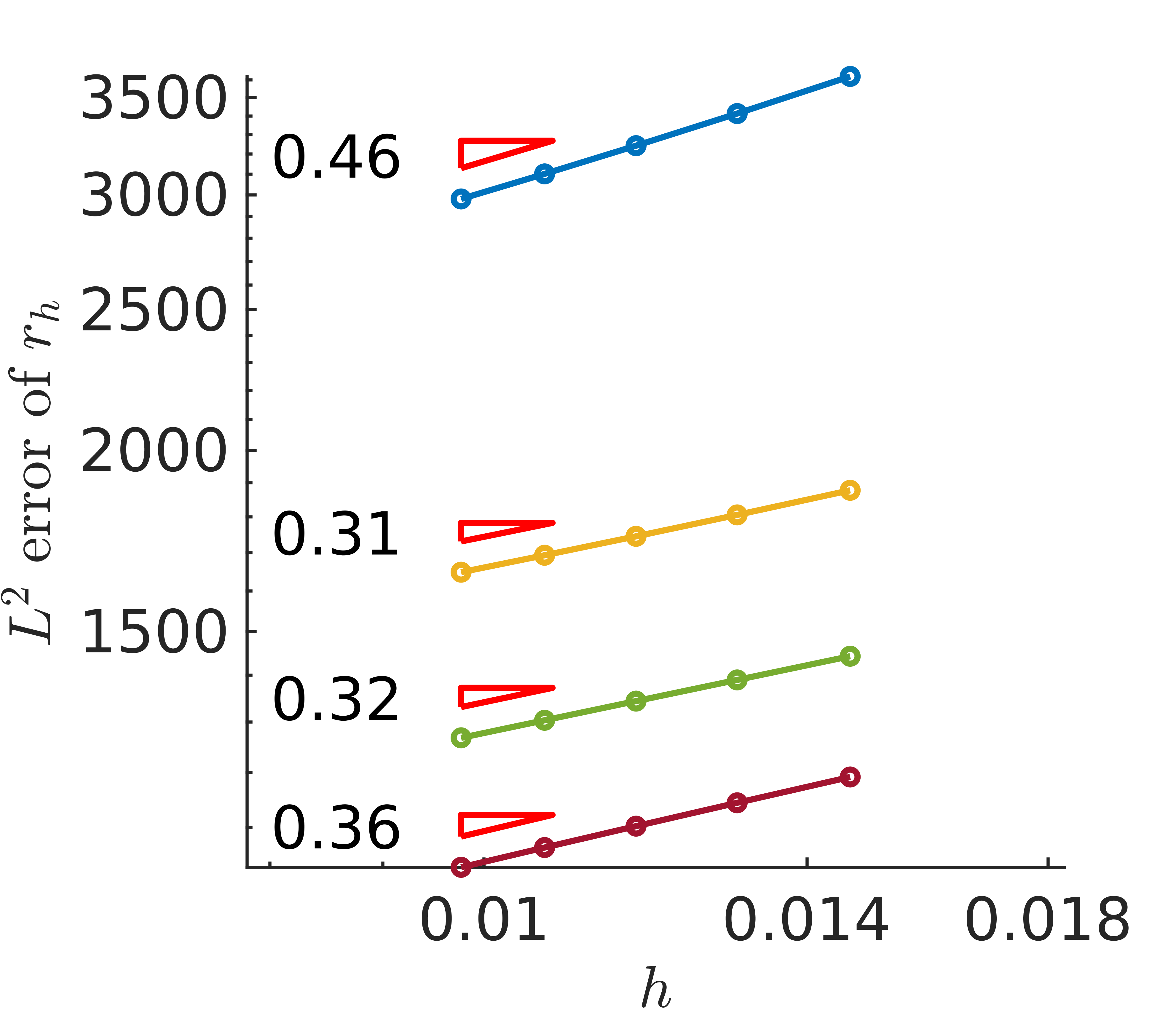}}\\
     \centered{\includegraphics{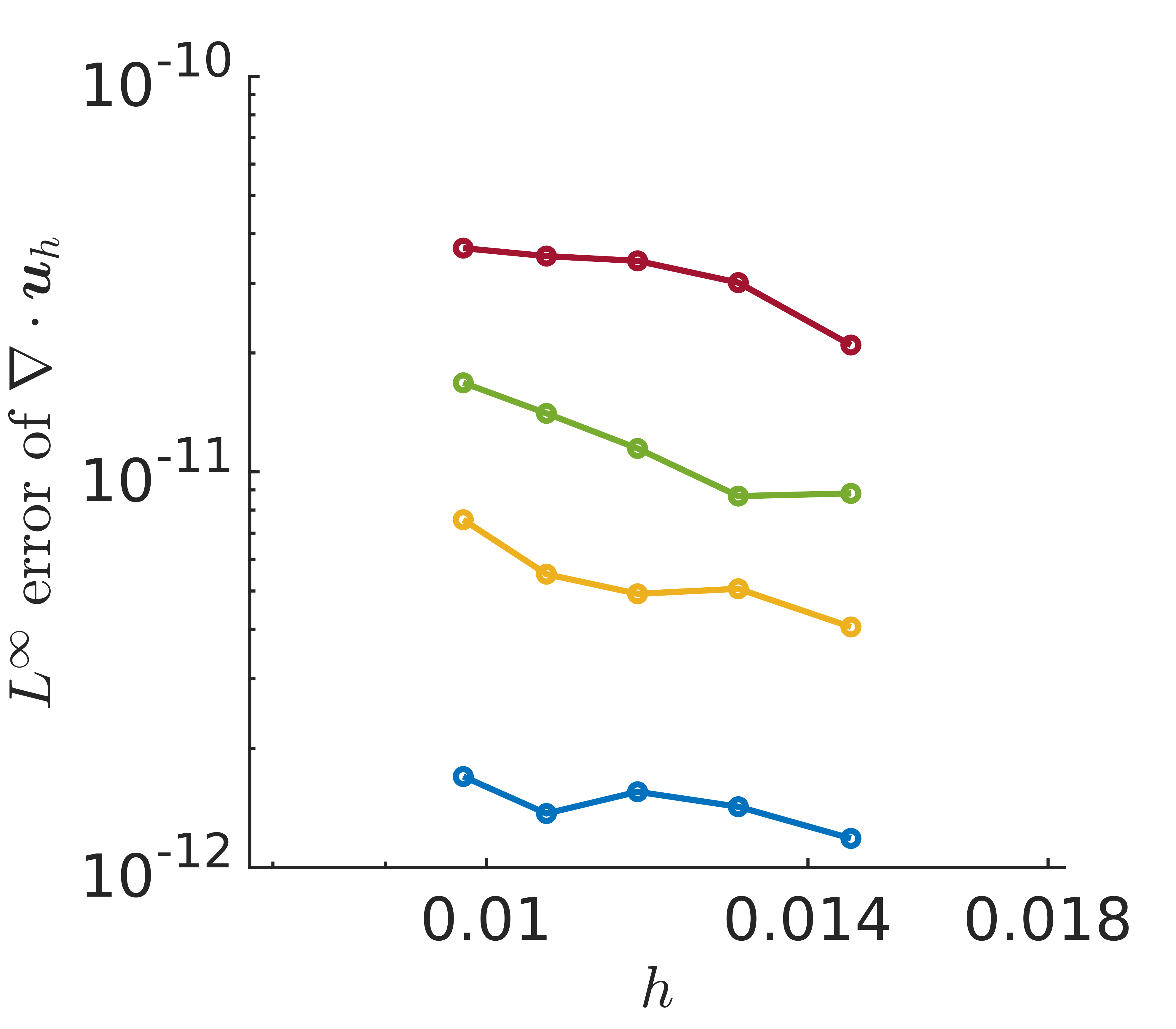}} &
     \centered{\includegraphics{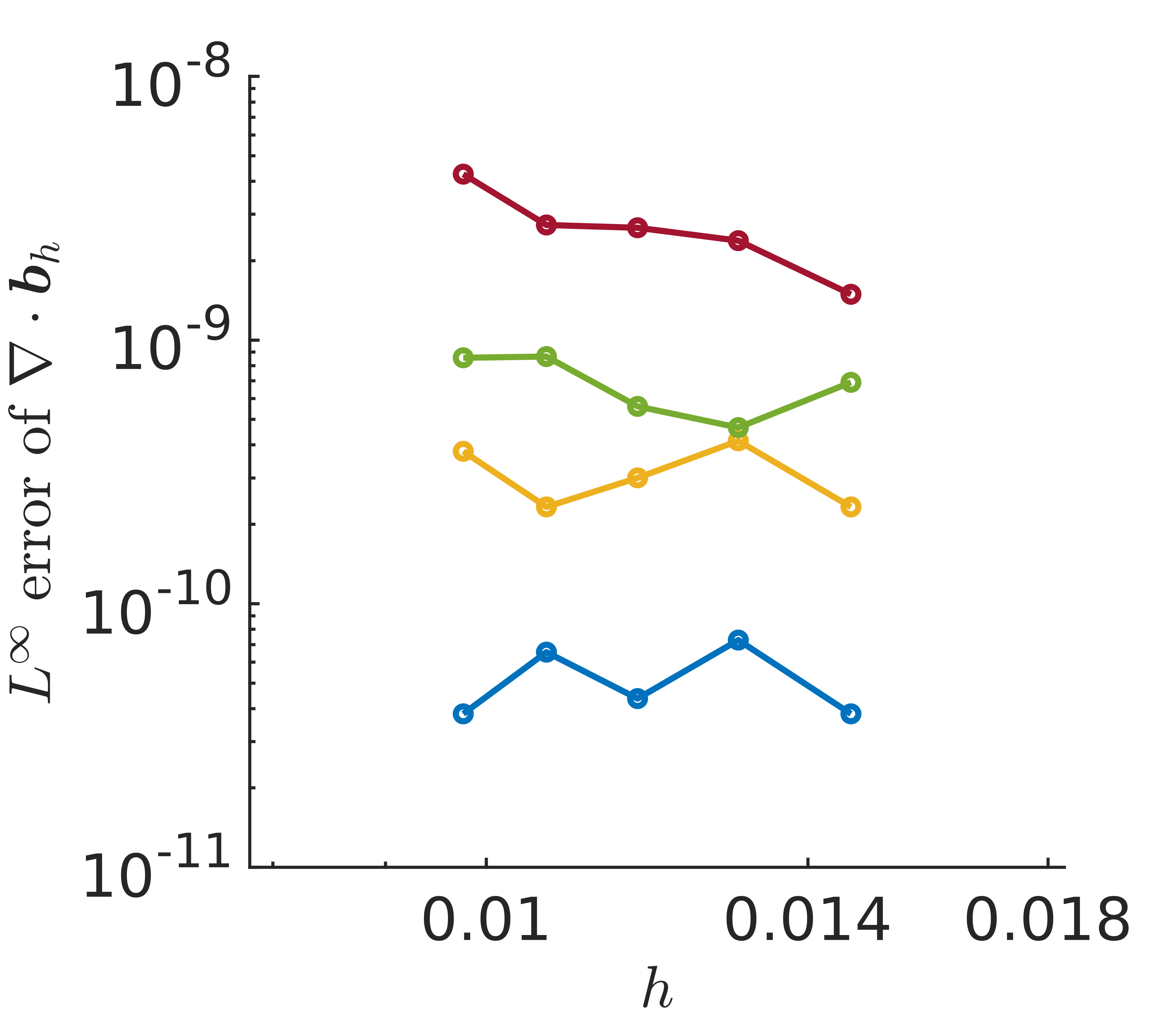}} &
     \centered{\includegraphics{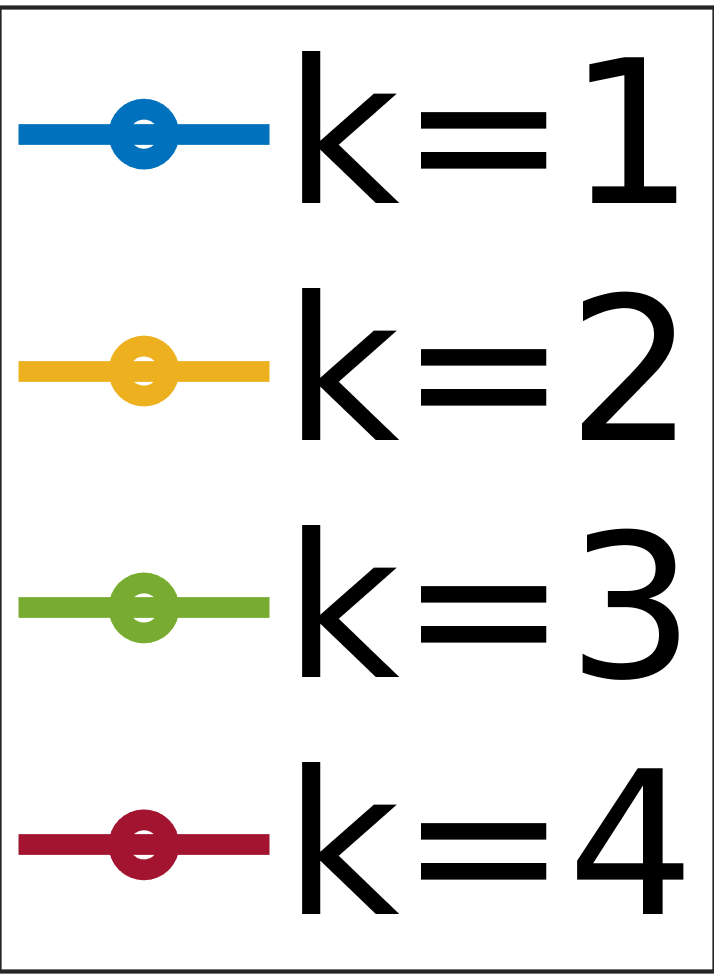}}
    \end{tabular}
    }
    \caption{Convergence histories of all local variables and divergence errors for the E-HDG method applied to solve the two-dimensional problem with a manufactured solution given in \eqnref{singular_2d} where a strong singularity exists in the magnetic field.}\figlab{lin_singular_2d}
\end{figure*}

\subsubsection{Three-dimensional smooth manufactured solution}\seclab{lin_smooth_3d} 
We now apply our E-HDG method to a three-dimensional problem on structured tetrahedron meshes. Note that our well-posedness analysis is still valid for this case. We set $\Omega=(0,1)\times(0,1)\times(0,1)$ and take $\Rey=\Rm\in\LRc{1,1000}$ and $\kappa=1$. For this test case, we choose the forcing function such that the exact solution is given by
\begin{subequations}\eqnlab{smooth_3d}
\begin{align}
    \begin{split}
        \bu =
        \begin{pmatrix}
            -\LRp{y\cos{(y)}+\sin{y}}e^{x},\\
            y\sin{(y)}e^x-\LRp{z\cos{(z)}+\sin{(z)}}e^{y},\\
            z\sin{(z)}e^{y}
        \end{pmatrix},
    \end{split}\\
    \begin{split}
        \bb = \bu
    \end{split}\\
    \begin{split}
        \p = \p_0\LRp{2e^{x}\sin{(y)}z^2-\LRp{\frac{-2}{3}(e\cos{(1)}-\cos{(1)}-e+1}}
    \end{split}\\
    \begin{split}
        \r = 0
    \end{split}
\end{align}
\end{subequations}
with the prescribed fields $\wb = \bu$ and $\db = \bb$, and a constant $\p_0$. Table \tabref{lin_smooth_3d} summarizes the convergence rates of all local variables and shows the $\Lsp^{\infty}$-norm of the divergence errors. The corresponding convergence histories are shown in Figure \figref{lin_smooth_3d}. Similar to the two-dimensional smooth testing case presented in Section \secref{lin_smooth_2d}, we observed that the convergence rates are affected by $\Rey$ and $\Rm$ here as well, but in an adverse manner. The effect is evident for $\LbH$, $\ubH$ and $\bbH$. 
We present the convergence rates for the case $\Rey=\Rm=1$. As can be seen, $\pH$ and $\rH$ exhibit superconvergence with a rate of $\kbr+3/2$, and the convergence rates of $\ubH$ and $\bbH$ are optimal with $\k+1$. For $\LbH$ and $\JbH$, the convergence rate is, however, suboptimal with $\k$. The conclusion is consistent with the one made in Section \secref{lin_smooth_2d} where the two-dimensional smooth manufactured solution is applied. 

The numerical assessment of the pressure robustness of our method is also carried out for this manufactured solution. The examination is conducted by perturbing the solution in pressure on two meshes, one consisting of 48 elements and the other with 24576 elements, with $k=2$ and various values of $\p_0$. Table \tabref{pressure_robust_3d} details the results. Similar to the two-dimensional case presented in Table \tabref{pressure_robust_2d}, the $\Lsp^2$-errors in velocity and magnetic field are independent of pressure on different meshes.   

\begin{table}[!htb]
\centering
\begin{tabular}{|c|cccccc|cc|}
\multicolumn{9}{c}{$\Rey=\Rm=1,\kappa=1$} \\
\hline
 &
\begin{tabular}{@{}c@{}} $\Rey\LbH$ \end{tabular} & 
\begin{tabular}{@{}c@{}} $\ubH$\end{tabular} & 
\begin{tabular}{@{}c@{}} $\pH$ \end{tabular} & 
\begin{tabular}{@{}c@{}} $\frac{\Rm}{\kappa}\JbH$\end{tabular} & 
\begin{tabular}{@{}c@{}} $\bbH$\end{tabular} & 
\begin{tabular}{@{}c@{}} $\rH$\end{tabular} & 
$\norm{\Div{\bu_h}}_{\infty}$ & $\norm{\Div{\bb_h}}_{\infty}$\\
\hline
$k=1$ & 0.72 &	1.78 &	1.81 &	1.02 &	2.04 &	1.95 &	6.19E-13 & 2.11E-13 \\ 
$k=2$ & 2.21 &	3.50 &	2.85 &	2.21 &	3.23 &	2.79 &	1.41E-12 & 1.29E-12 \\ 
$k=3$ & 3.08 &	3.99 &	3.66 &	3.22 &	4.19 &	3.75 &	9.07E-10 & 1.96E-11 \\ 
$k=4$ & 4.22 &	5.23 &	4.73 &	4.24 &	5.23 &	4.70 &	3.66E-09 & 8.06E-11 \\  
\hline
\multicolumn{9}{c}{$\Rey=\Rm=1000,\kappa=1$} \\
\hline
$k=1$ & 0.47 &	1.38 &	1.89 &	0.56 &	0.77 &	1.96 &	5.73E-13 & 2.13E-13 \\ 
$k=2$ & 1.28 &	2.25 &	3.01 &	1.31 &	2.19 &	2.93 &	1.64E-12 & 1.45E-12 \\ 
$k=3$ & 3.64 &	4.69 &	3.92 &	3.77 &	4.72 &	3.93 &	1.46E-09 & 1.93E-11 \\ 
$k=4$ & 3.94 &	4.22 &	4.91 &	4.03 &	4.29 &	5.49 &	5.55E-09 & 5.94E-09 \\ 
\hline
\end{tabular}
\caption{Convergence rates of all local variables and divergence errors of velocity and magnetic fields for the E-HDG method applied to solve the three-dimensional problem with a smooth manufactured solution given in \eqnref{smooth_3d} where we set $\p_0=1$. The corresponding results are also presented in Figure \figref{lin_smooth_3d}. In this table, the convergence rates are evaluated at the last two data sets and the divergence errors are evaluated at the last data set.}\tablab{lin_smooth_3d}
\end{table}

\begin{table}[!htb]
\centering
\resizebox{\textwidth}{!}{
\begin{tabular}{|c|c|c|c|c|c|c|c|c|}
\multicolumn{9}{c}{48 elements in total, $h\approx1.06E-1$} \\
\hline
$p_0$ &
$\Rey\norm{\Lb-\LbH}_0$ &
$\norm{\ub-\ubH}_0$ &
$\norm{\p-\pH}_0$ &
$\frac{\Rm}{\kappa}\norm{\Jb-\JbH}_0$ &
$\norm{\bb-\bbH}_0$ &
$\norm{\r-\rH}_0$ &
$\norm{\Div{\ub_h}}_{\infty}$ &
$\norm{\Div{\bb_h}}_{\infty}$ \\
\hline
1   & 7.52E-2  & 2.69E-3 & 1.59 & 6.42E-2 & 2.42E-3 & 1.29 & 2.42E-13 & 2.19E-13	\\
10  & 7.52E-2  & 2.69E-3 & 5.99 & 6.42E-2 & 2.42E-3 & 1.29 & 2.99E-13 & 2.48E-13	\\
25  & 7.52E-2  & 2.69E-3 & 15.57 & 6.42E-2 & 2.42E-3 & 1.29 & 3.09E-13 & 2.53E-13	\\ 
100 & 7.52E-2  & 2.69E-3 & 64.09 & 6.42E-2 & 2.42E-3 & 1.29 & 2.89E-13 & 2.60E-13	\\  
\hline
\multicolumn{9}{c}{24576 elements in total, $h\approx2.64E-2$} \\
\hline
1   & 3.19E-3  & 2.10E-5 & 3.16E-2 & 2.70E-3 & 2.53E-5 & 2.90E-2 & 1.40E-12 & 1.43E-12	\\
10  & 3.19E-3  & 2.10E-5 & 5.83 & 2.70E-3 & 2.53E-5 & 2.90E-2 & 1.40E-12 & 1.31E-12	\\ 
25  & 3.19E-3  & 2.10E-5 & 15.55 & 2.70E-3 & 2.53E-5 & 2.90E-2 & 1.41E-12 & 1.22E-12	\\  
100 & 3.19E-3  & 2.10E-5 & 64.13 & 2.70E-3 & 2.53E-5 & 2.90E-2 & 1.32E-12 & 1.44E-12	\\
\hline
\end{tabular}
}
\caption{The errors in the local variables for the smooth manufactured solution given in \eqnref{smooth_3d} for meshes of 48 and 24576 elements, a polynomial degree of $k=2$, and a range of $p_0$ values. We set $\Rey=\Rm=1$ and $\kappa=1$.}
\tablab{pressure_robust_3d}
\end{table}

\begin{figure*}
    \centering
    \resizebox{\textwidth}{!}{
    \begin{tabular}{ccc}
     \centered{\includegraphics{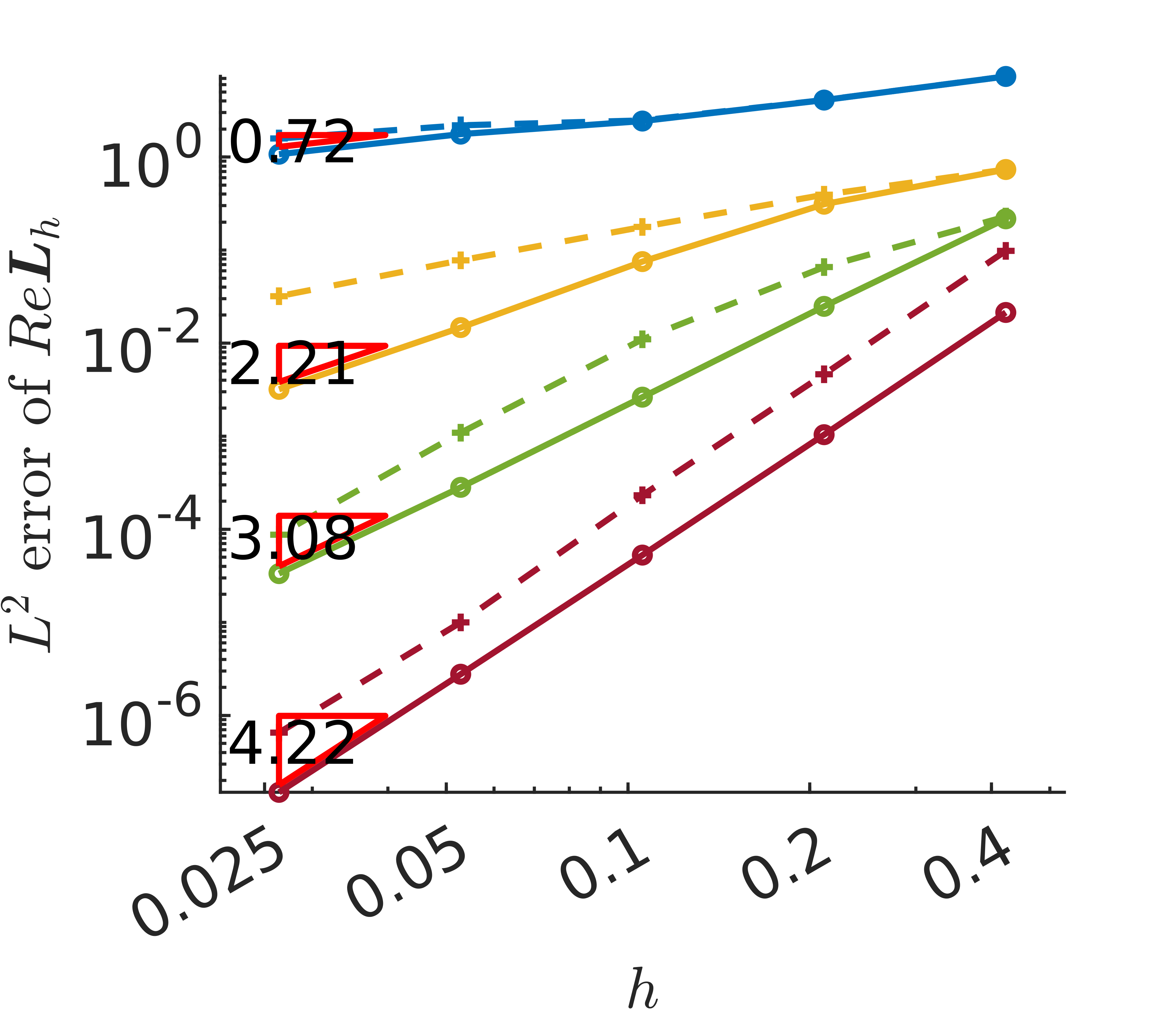} }& 
     \centered{\includegraphics{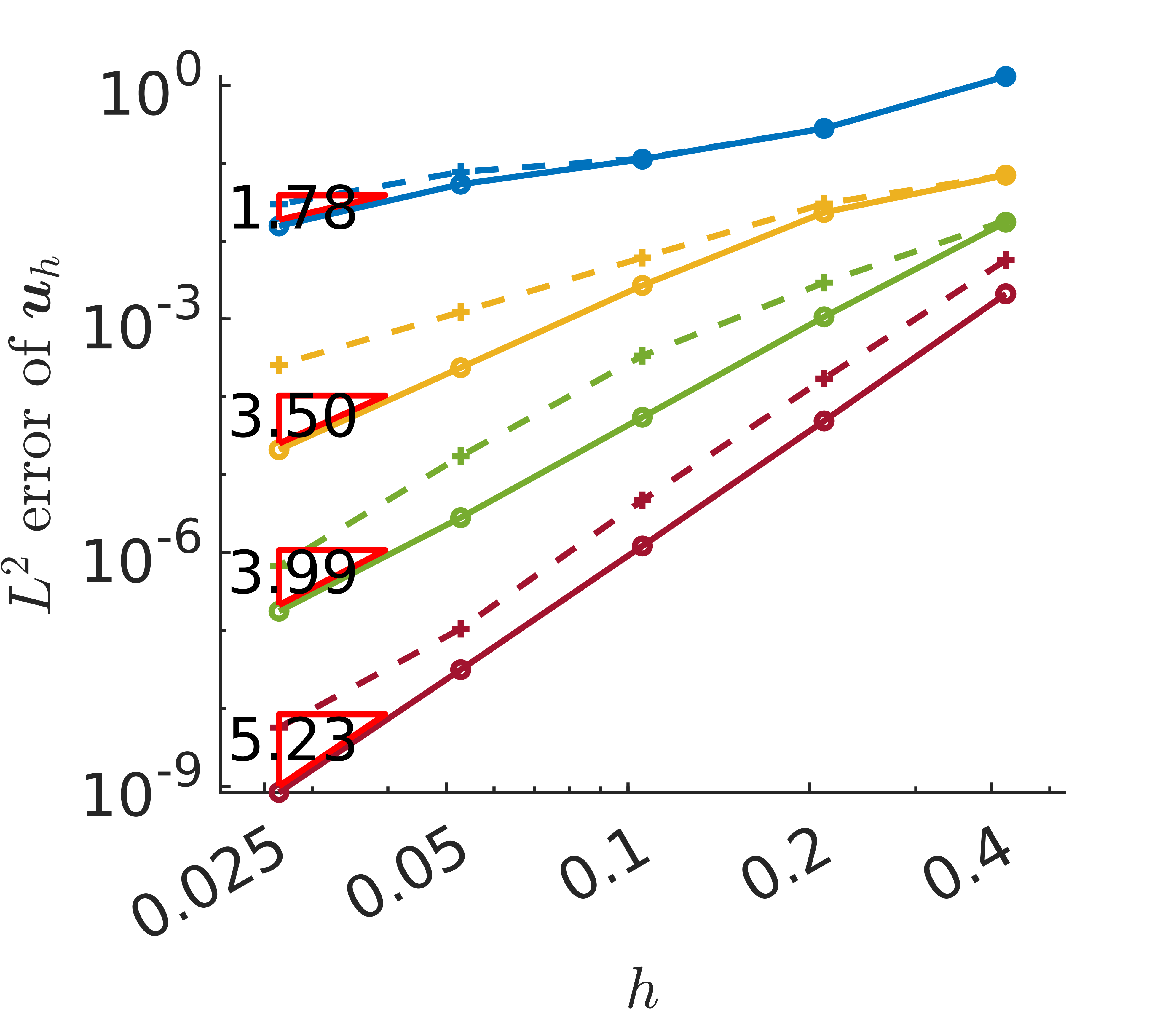}} &
     \centered{\includegraphics{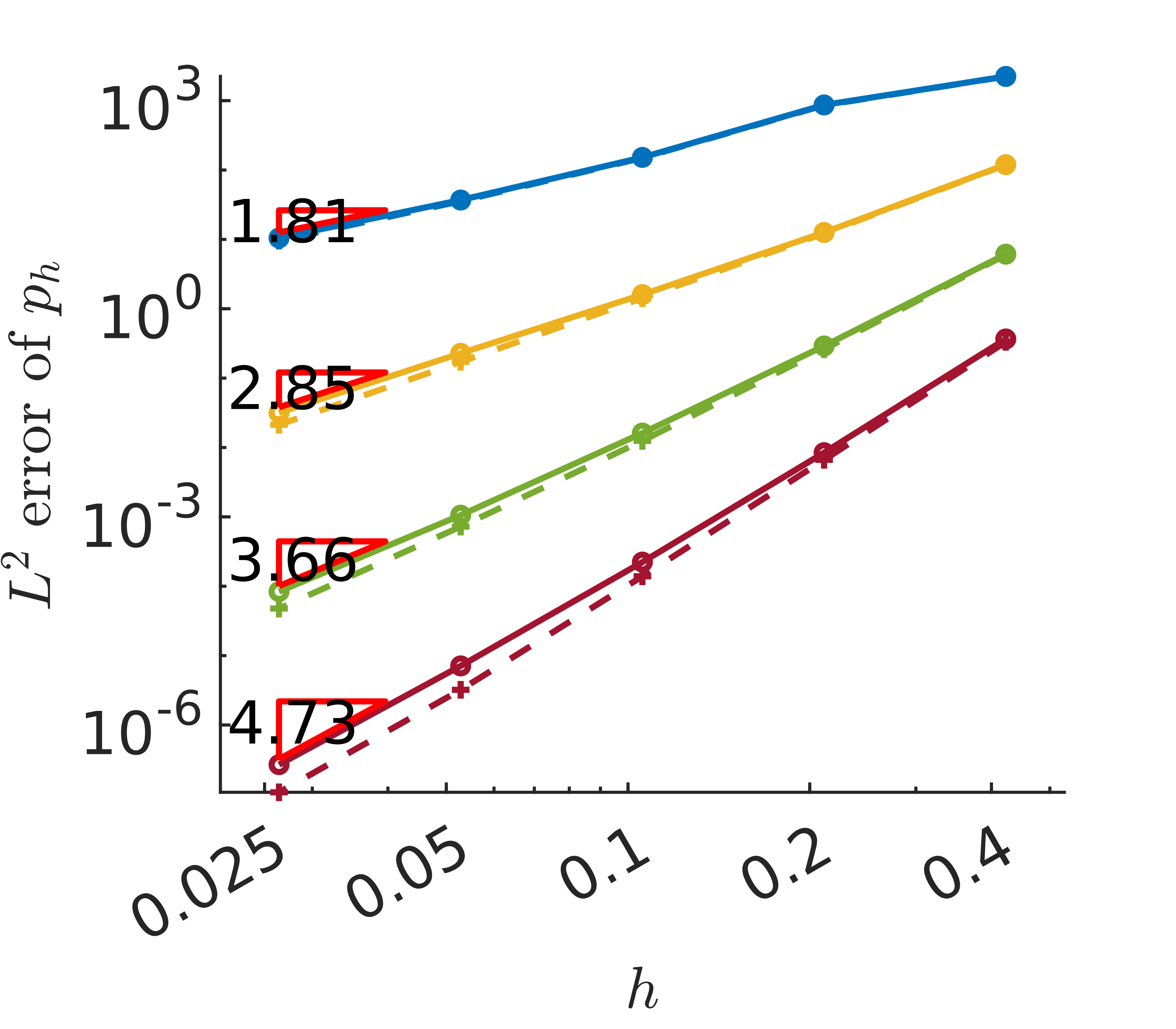}}\\
     \centered{\includegraphics{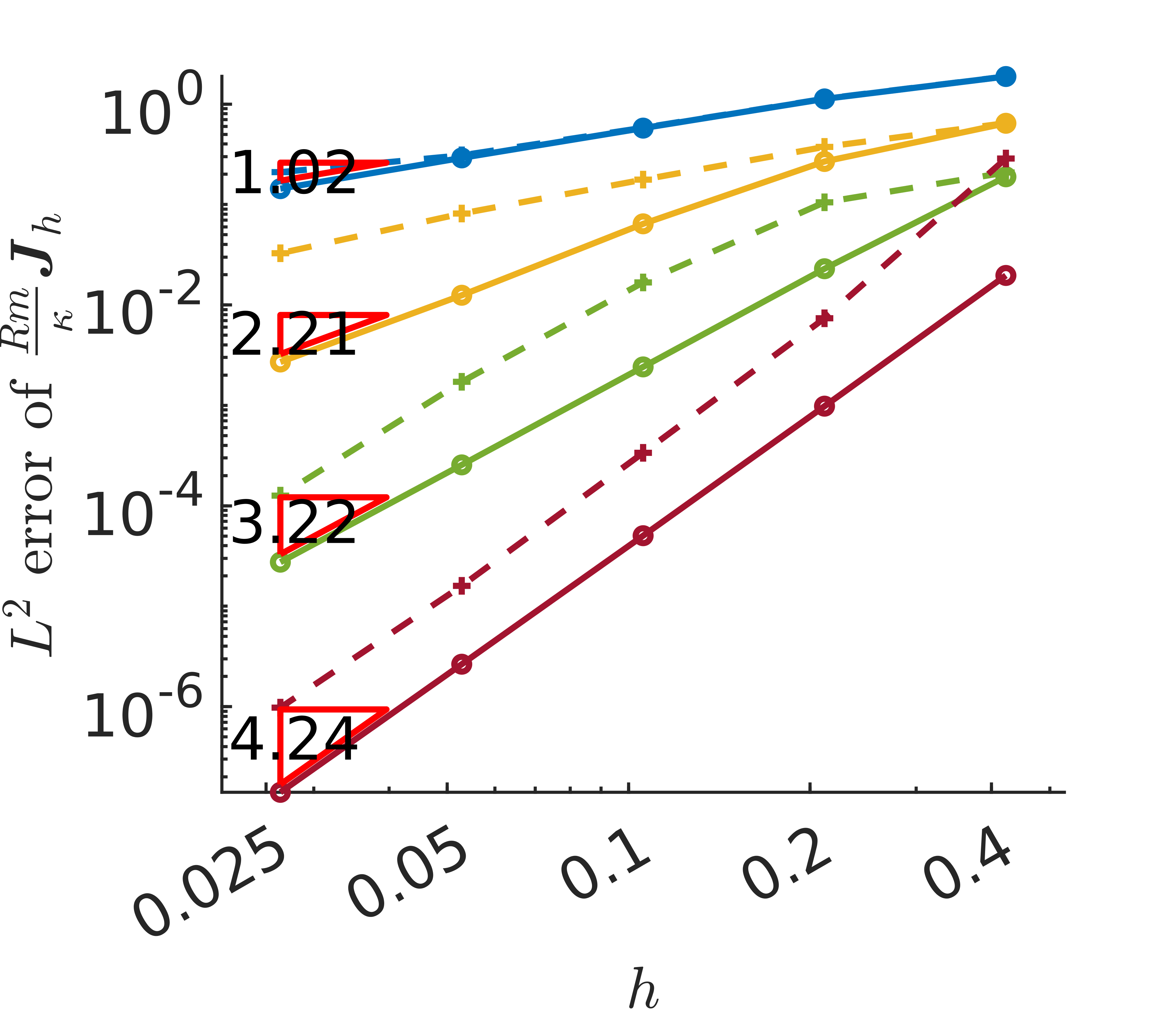}} &
     \centered{\includegraphics{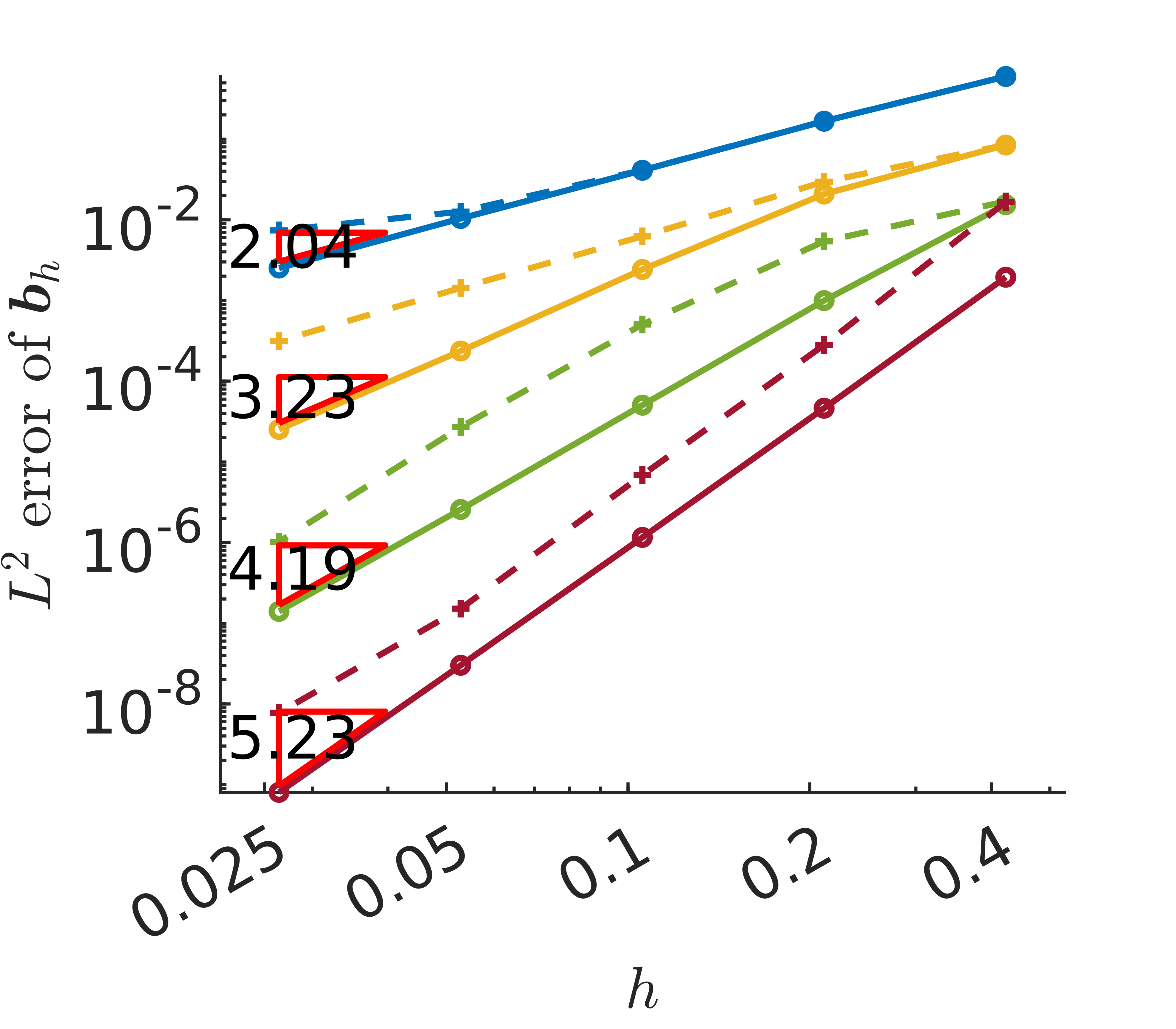}} &
     \centered{\includegraphics{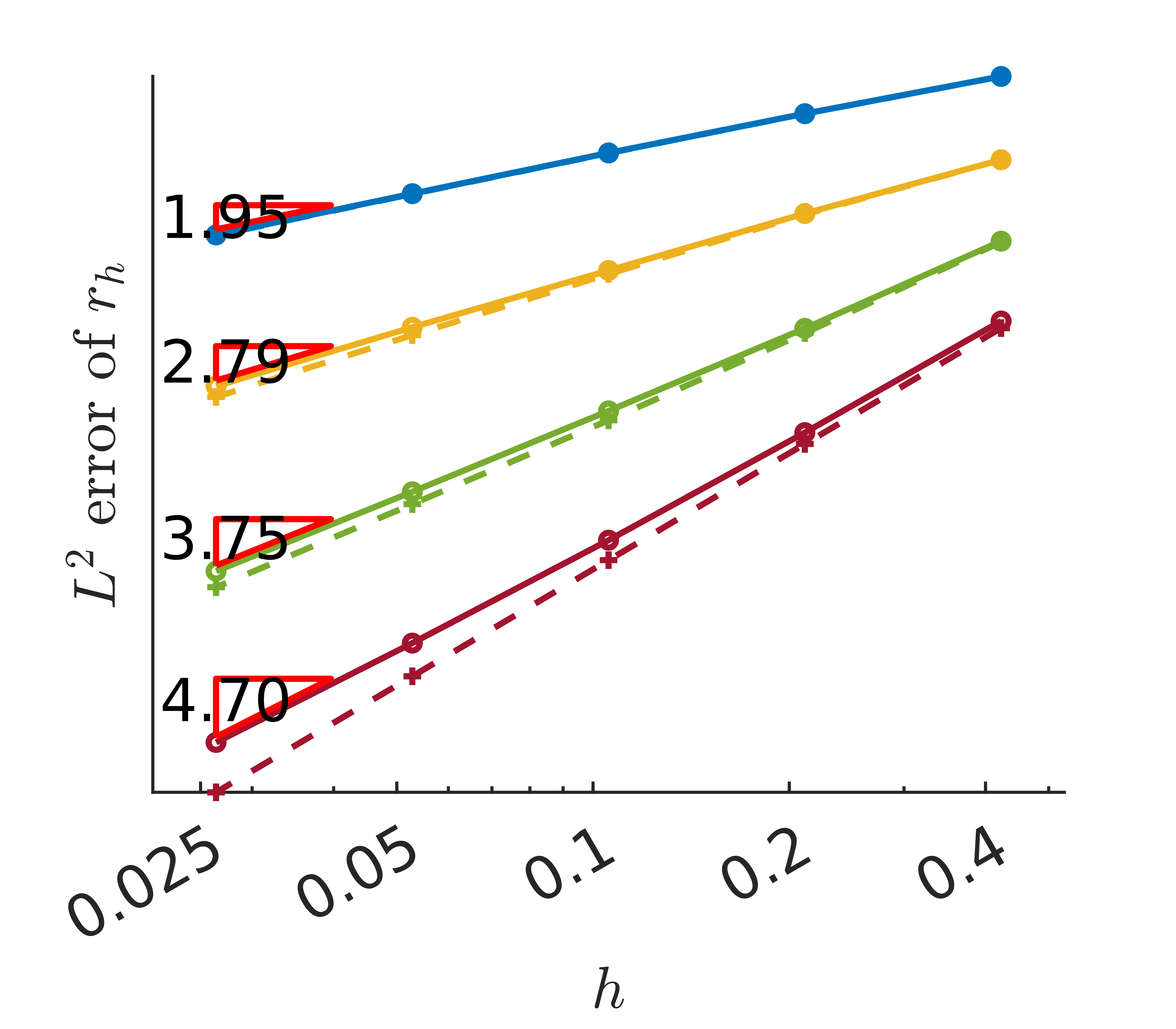}}\\
     \centered{\includegraphics{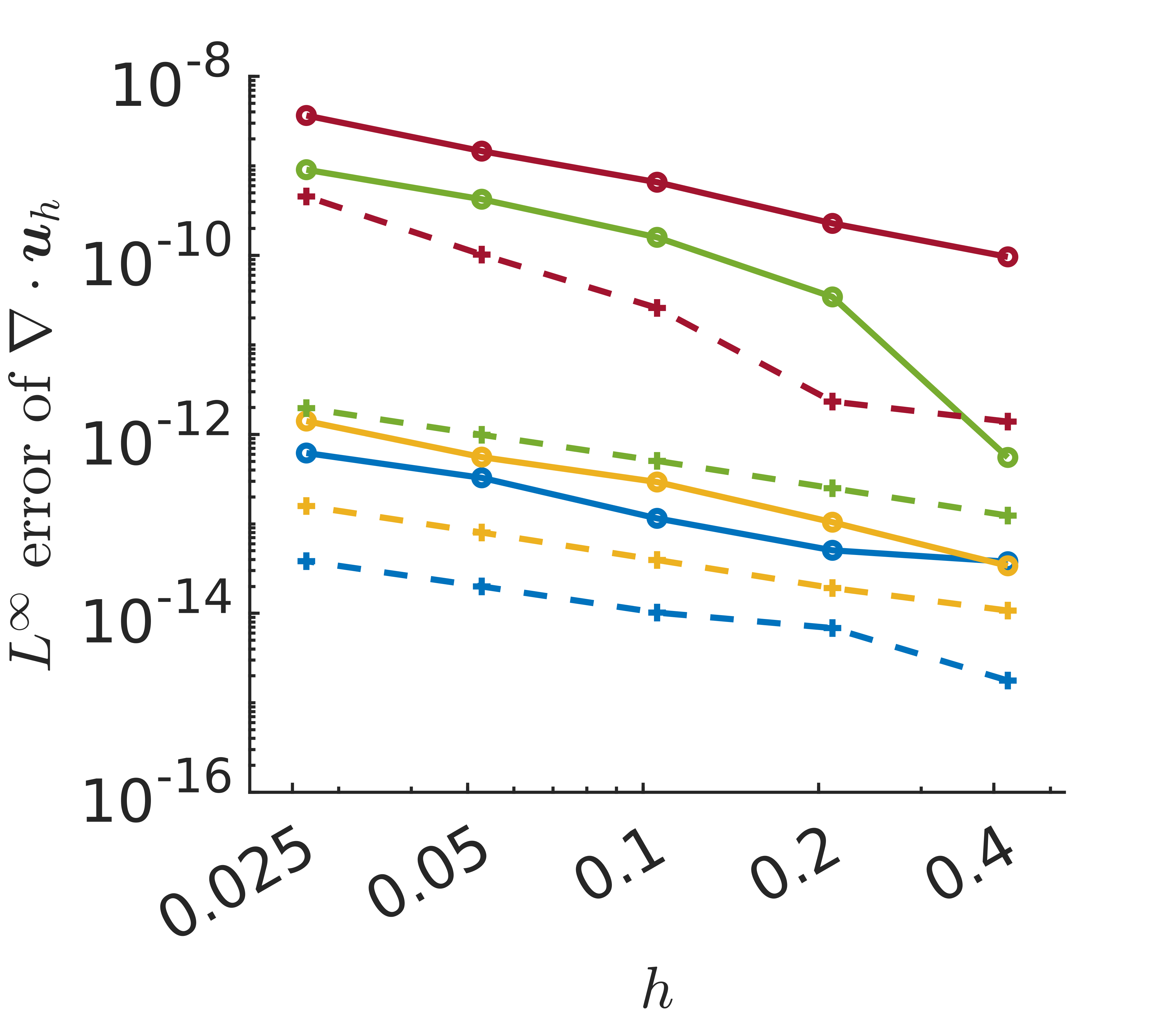}} &
     \centered{\includegraphics{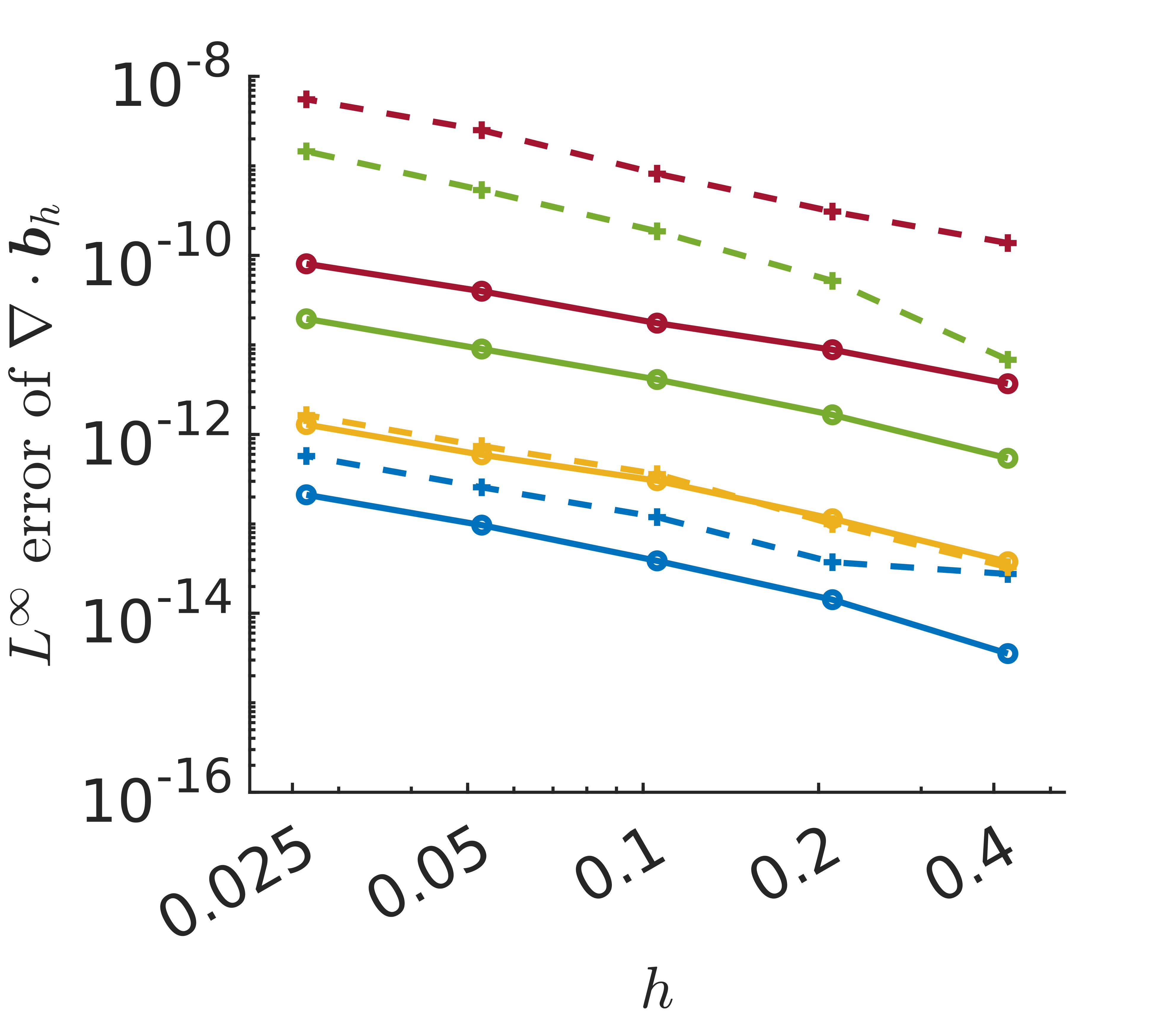}} &
     \centered{\includegraphics{conv_legend_multi.png}}
    \end{tabular}
    }
    \caption{Convergence histories of all local variables and divergence errors for the E-HDG method applied to solve the three-dimensional problem with a smooth manufactured solution given in \eqnref{smooth_3d} where we set $\p_0=1$. Only the convergence rates for $\Rey=\Rm=1$ are presented here.}\figlab{lin_smooth_3d}
\end{figure*}

\subsection{Nonlinear examples}\seclab{nonlinear}
To verify our nonlinear solver, we conducted several numerical experiments and studied the accuracy and convergence. The first example is the two-dimensional smooth manufactured solution, the second one is the so-called Hartmann flow problem, and the last one is the three-dimensional smooth manufactured solution. 

\subsubsection{Two-dimensional smooth manufactured solution}\seclab{nonlin_smooth_2d}
Our first numerical experiment for the nonlinear solver is a steady manufactured solution. In particular, we use the same solution presented in Section \secref{lin_smooth_2d} to investigate the convergence. The results are presented in Table \tabref{nonlin_smooth_2d} and are illustrated in Figure \figref{nonlin_smooth_2d}. The observed convergence rates are almost the same as the rates observed in the linear problem presented in Section \secref{lin_smooth_2d}. Moreover, the divergence errors also exhibit the same order of magnitude.  

\begin{table}[!htb]
\centering
\begin{tabular}{|c|cccccc|cc|}
\multicolumn{9}{c}{$\Rey=\Rm=1,\kappa=1$} \\
\hline
 &
\begin{tabular}{@{}c@{}} $\Rey\LbH$ \end{tabular} & 
\begin{tabular}{@{}c@{}} $\ubH$\end{tabular} & 
\begin{tabular}{@{}c@{}} $\pH$ \end{tabular} & 
\begin{tabular}{@{}c@{}} $\frac{\Rm}{\kappa}\JbH$\end{tabular} & 
\begin{tabular}{@{}c@{}} $\bbH$\end{tabular} & 
\begin{tabular}{@{}c@{}} $\rH$\end{tabular} & 
$\norm{\Div{\bu_h}}_{\infty}$ & $\norm{\Div{\bb_h}}_{\infty}$\\
\hline
$k=1$ & 1.02 &	2.29 &	1.12 &	1.20 &	2.40 &	1.96 &	4.02E-15 & 3.72E-15 \\ 
$k=2$ & 2.02 &	3.08 &	2.21 &	2.28 &	3.06 &	2.73 &	2.39E-14 & 2.66E-14 \\ 
$k=3$ & 3.04 &	4.06 &	3.17 &	3.35 &	4.03 &	3.75 &	7.75E-14 & 7.65E-14 \\ 
$k=4$ & 4.03 &	5.08 &	4.28 &	4.51 &	4.91 &	4.79 &	2.81E-13 & 5.65E-13 \\     
\hline
\multicolumn{9}{c}{$\Rey=\Rm=1000,\kappa=1$} \\
\hline
$k=1$ & 1.27 &	1.35 &	0.99 &	1.38 &	1.47 &	0.65 &	2.69E-15 & 2.83E-15 \\ 
$k=2$ & 2.93 &	4.04 &	2.02 &	3.01 &	4.14 &	2.40 &	2.61E-14 & 2.35E-14 \\ 
$k=3$ & 3.98 &	5.28 &	3.02 &	3.85 &	5.18 &	3.81 &	1.13E-12 & 1.10E-13 \\ 
$k=4$ & 4.16 &	5.20 &	4.00 &	4.17 &	5.20 &	4.39 &	2.78E-12 & 1.73E-12 \\  
\hline
\end{tabular}
\caption{Convergence rates of all local variables and divergence errors of velocity and magnetic fields for the nonlinear solver applied to solve the two-dimensional problem with a smooth manufactured solution given in \eqnref{smooth_2d} where we set $\p_0=1$. The corresponding results are also presented in Figure \figref{nonlin_smooth_2d}. In this table, the convergence rates are evaluated at the last two data sets and the divergence errors are evaluated at the last data set.}\tablab{nonlin_smooth_2d}
\end{table}

\begin{figure*}
    \centering
    \resizebox{\textwidth}{!}{
    \begin{tabular}{ccc}
     \centered{\includegraphics{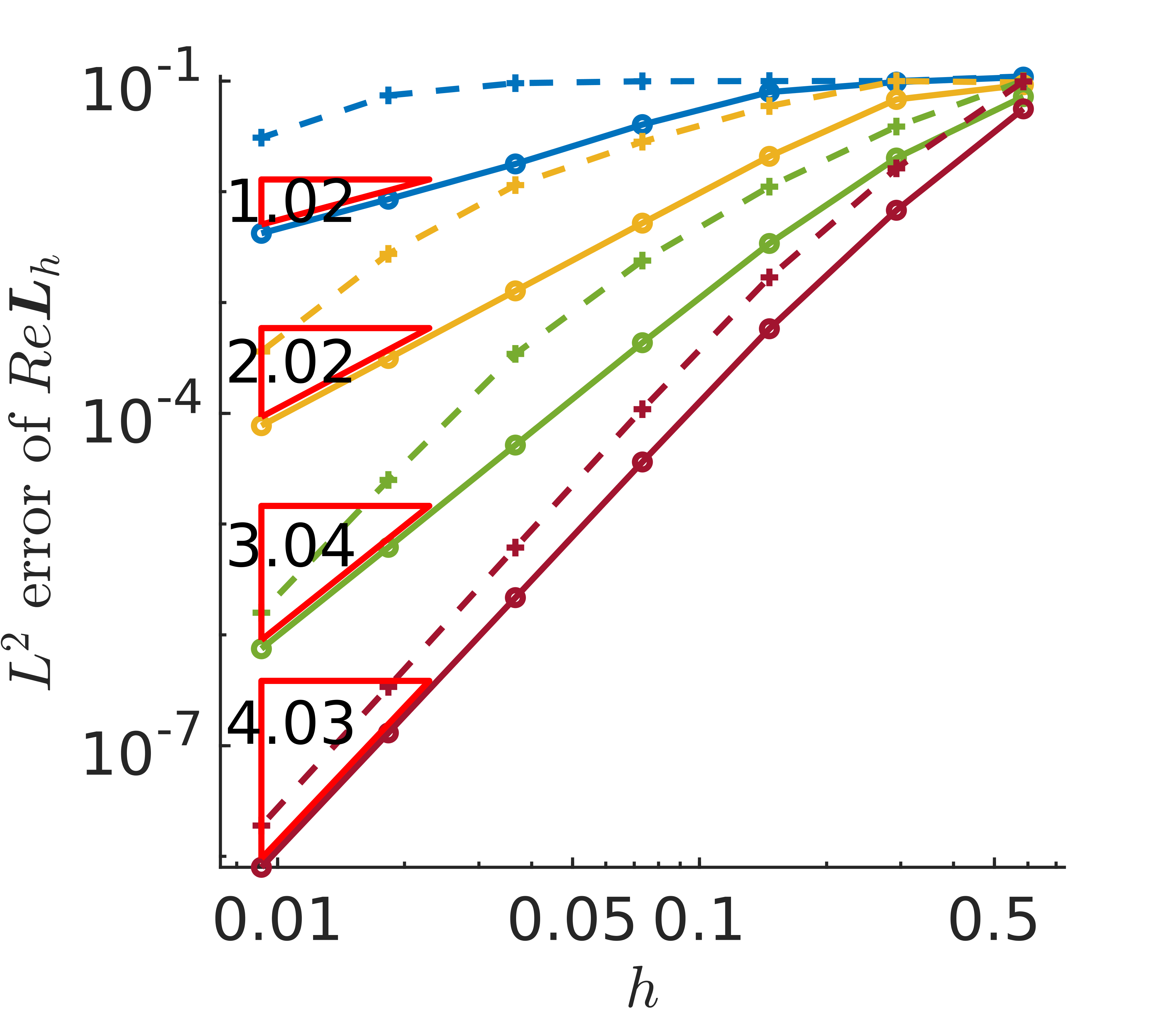} }& 
     \centered{\includegraphics{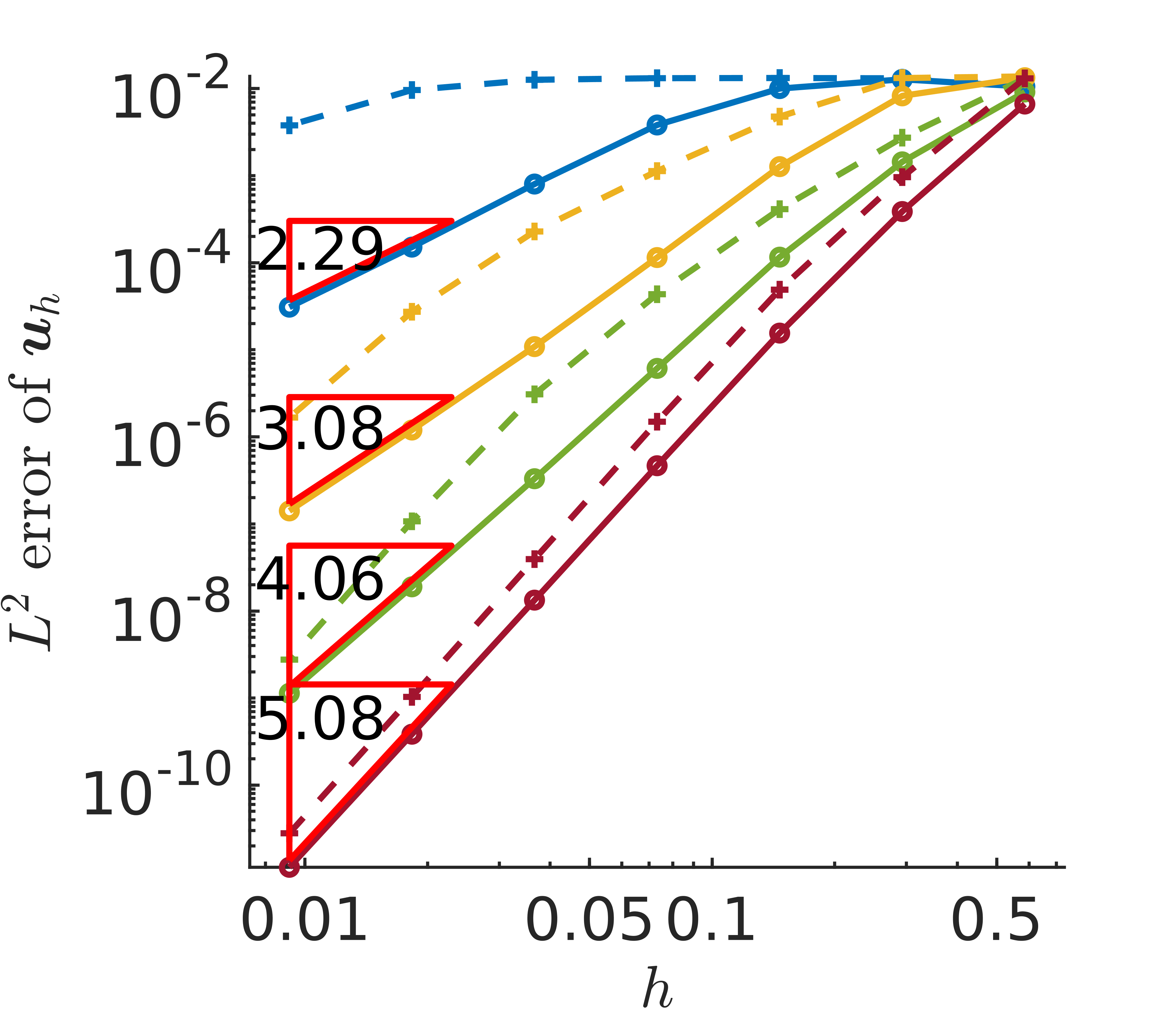}} &
     \centered{\includegraphics{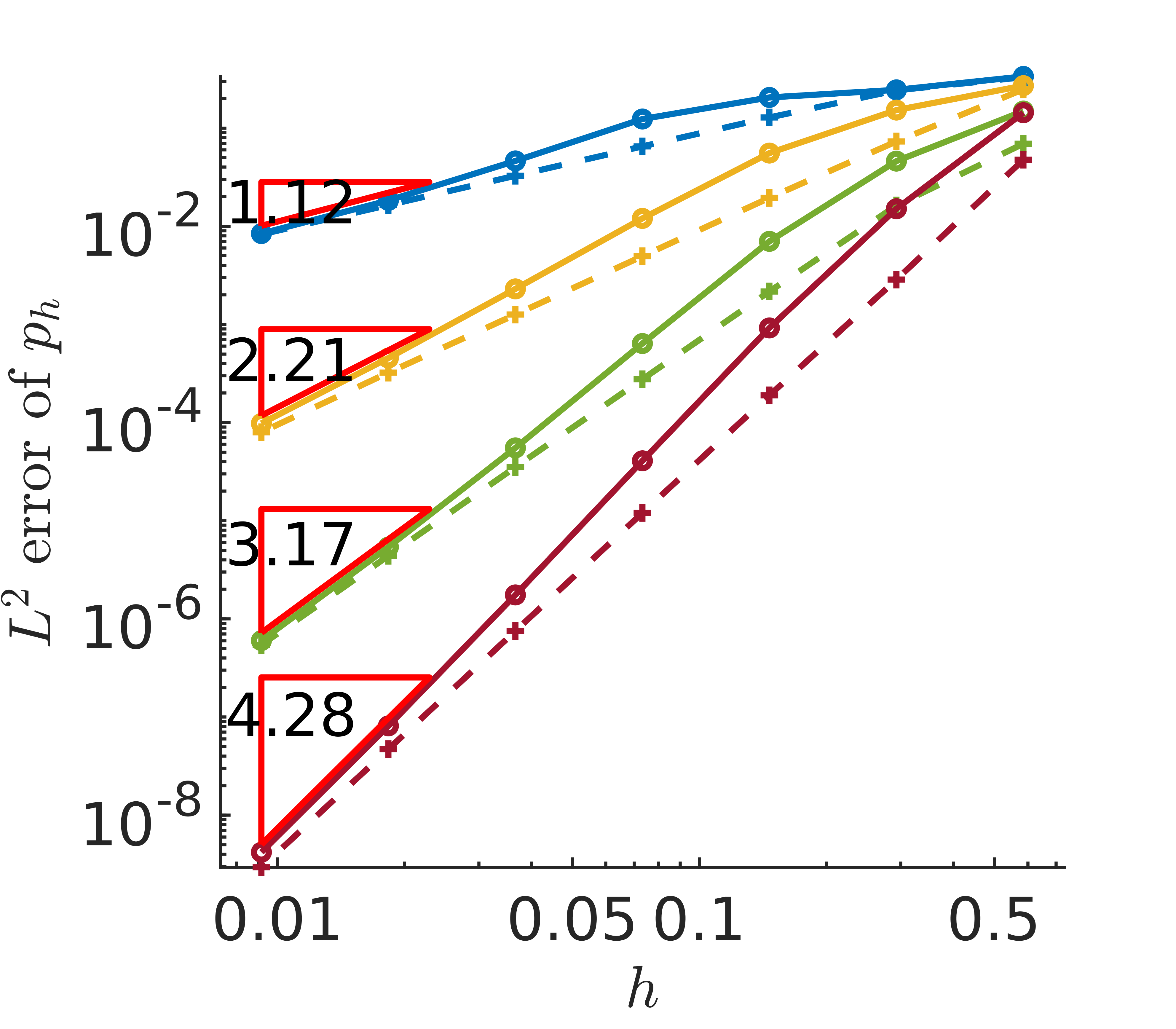}}\\
     \centered{\includegraphics{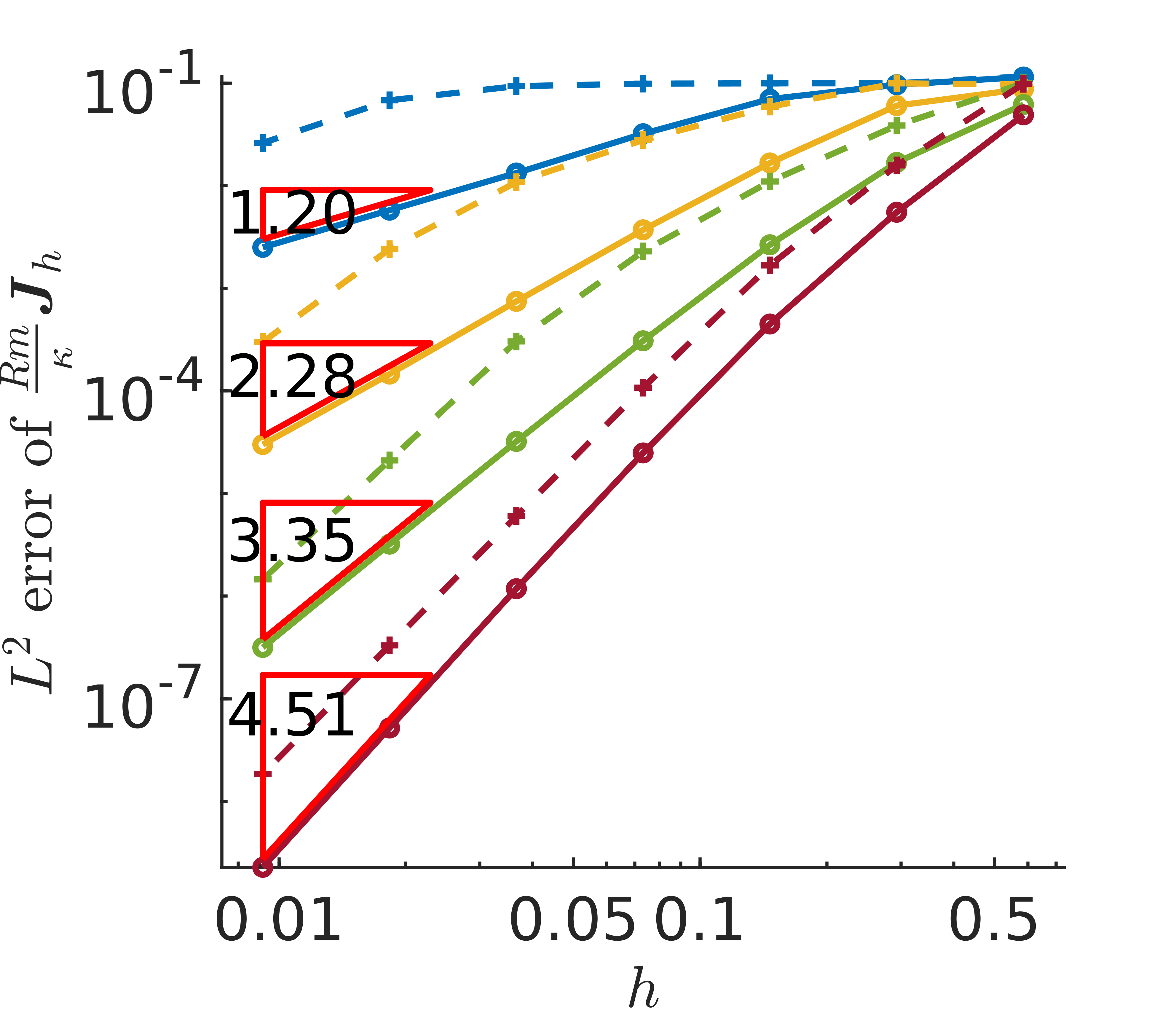}} &
     \centered{\includegraphics{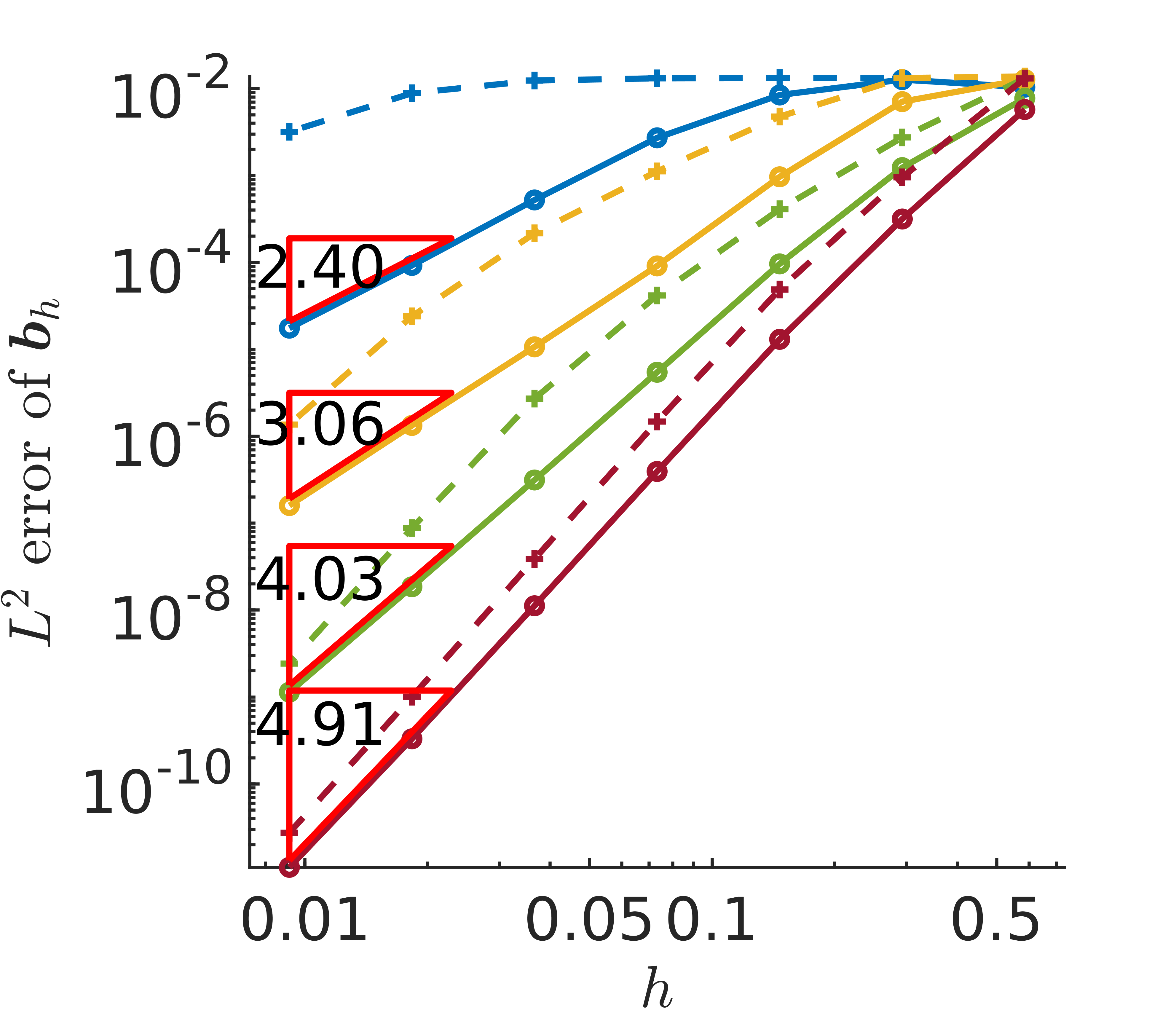}} &
     \centered{\includegraphics{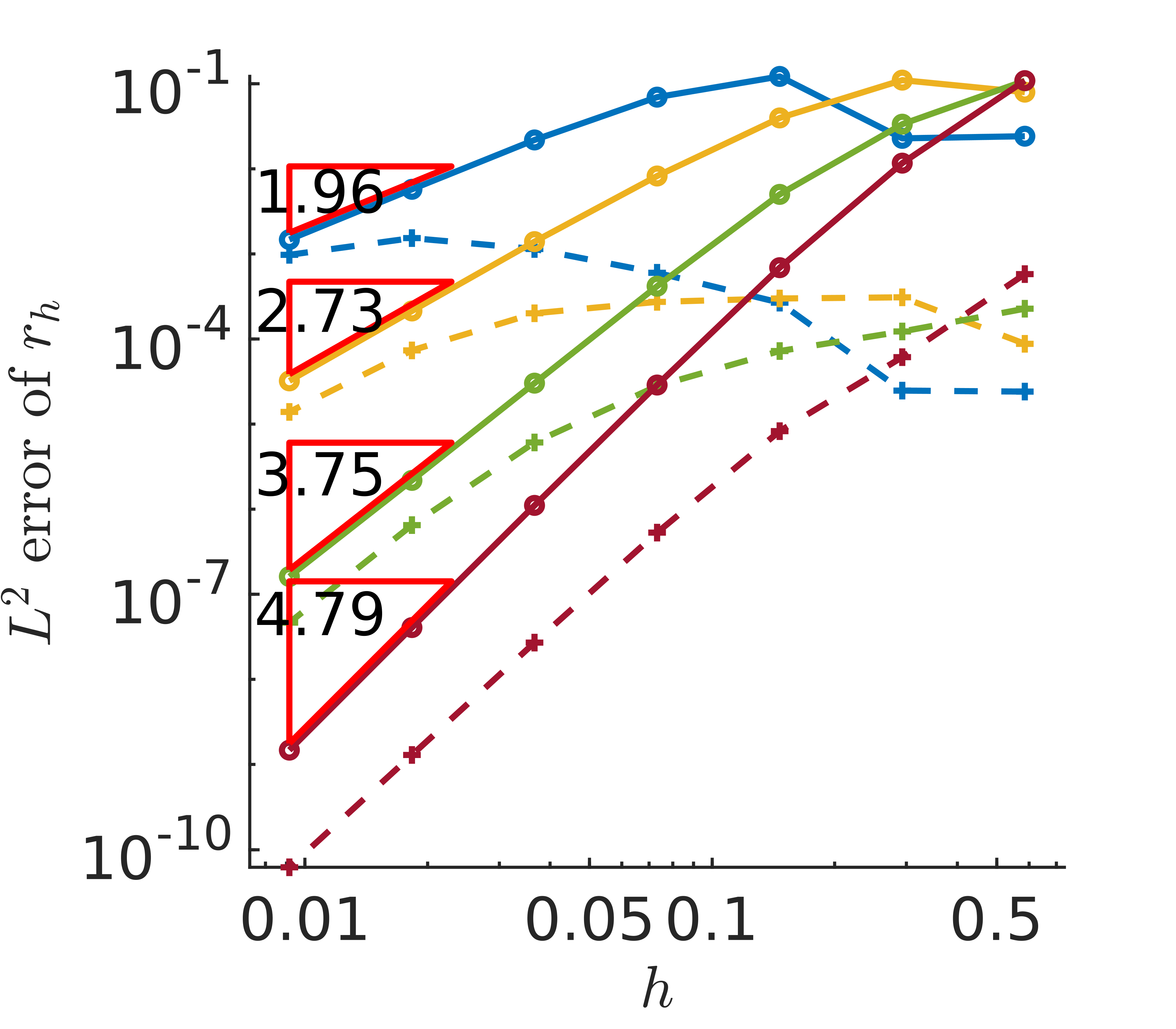}}\\
     \centered{\includegraphics{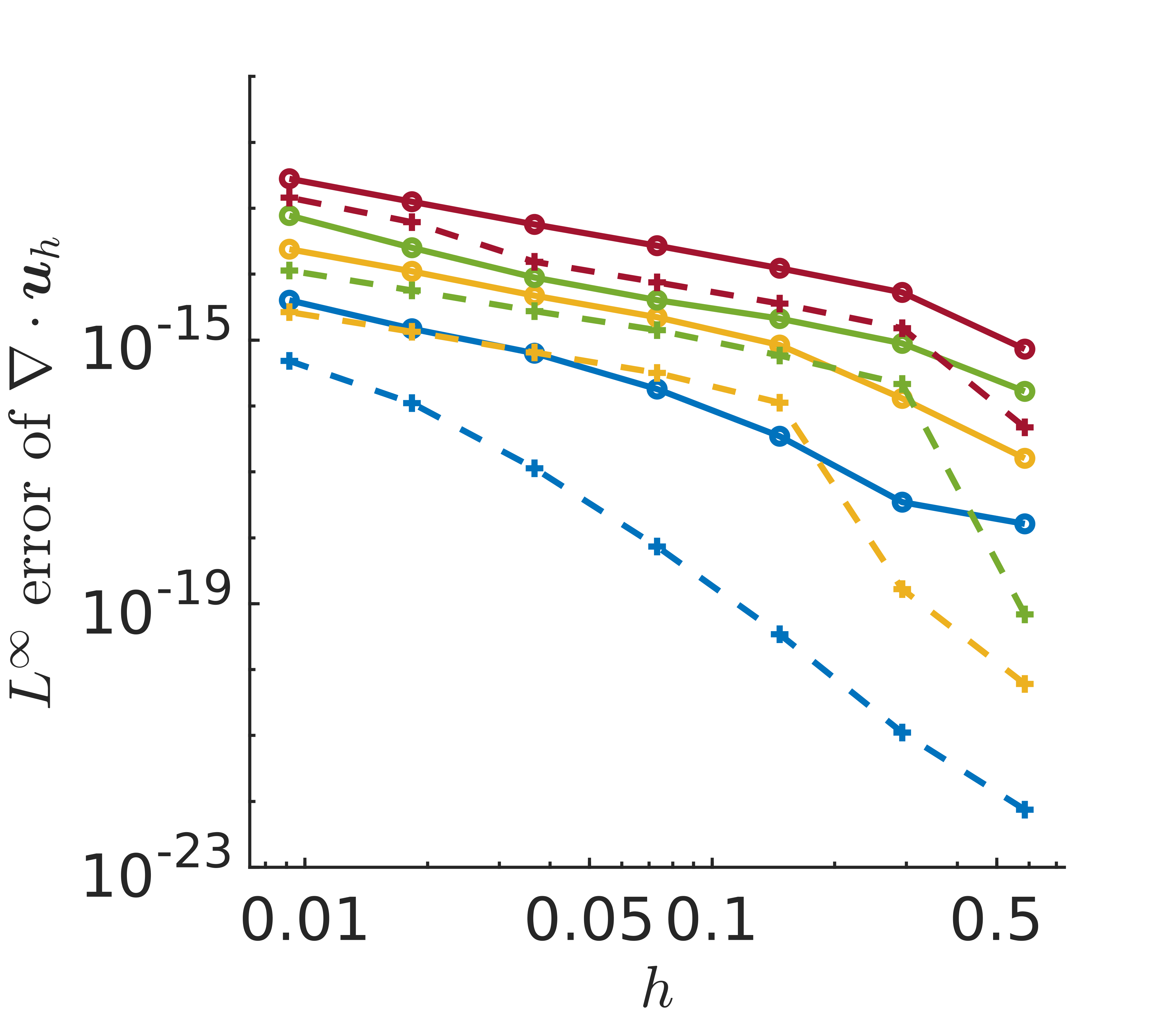}} &
     \centered{\includegraphics{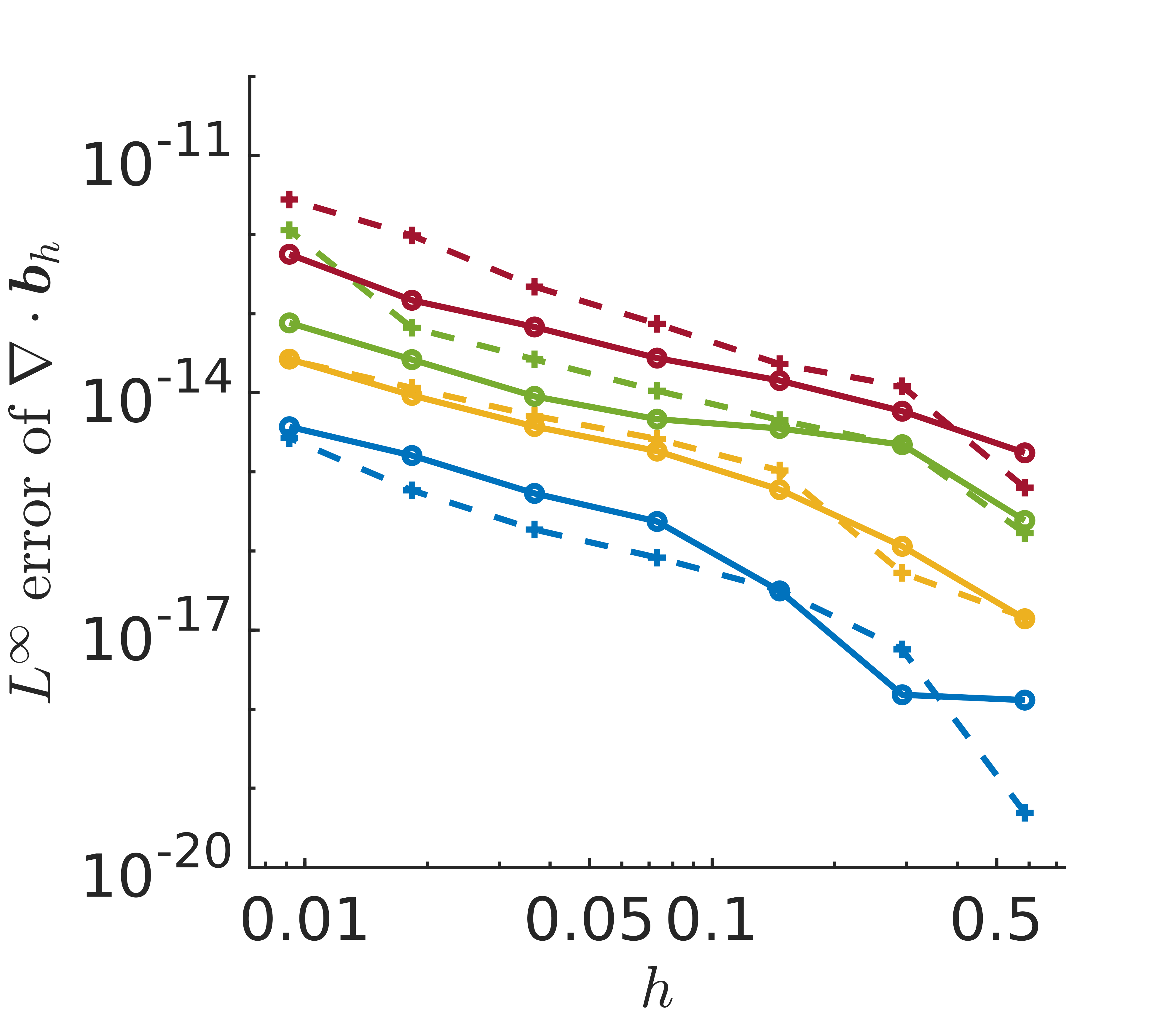}} &
     \centered{\includegraphics{conv_legend_multi.png}}
    \end{tabular}
    }
    \caption{Convergence histories of all local variables and divergence errors for the nonlinear solver applied to solve the two-dimensional problem with a smooth manufactured solution given in \eqnref{smooth_2d} where we set $\p_0=1$. Only the convergence rates for $\Rey=\Rm=1$ are presented here.}\figlab{nonlin_smooth_2d}
\end{figure*}

\subsubsection{Two-dimensional Hartmann flow}\seclab{nonlin_hartmann_2d}
We next consider the Hartmann channel flow, a generalization of the classic plane Poiseuille problem to the setting of the incompressible visco-resistive MHD. In this problem, a conducting incompressible fluid (liquid metal, for example) in a domain $(-\infty,\infty)\times(-l_0,l_0)\times(-\infty,\infty)$ (bounded by infinite parallel plates in the $x_2$ direction) is driven by a uniform pressure gradient $G:=-\pp{\p}{x_1}$ in the $x_1$ direction, and is subject to a uniform external magnetic field $b_0$ in the $x_2$ direction. In addition, we enforce no-slip boundary conditions on the $x_2$ boundaries and assume the infinite parallel plates are perfectly insulating. The resulting flow pattern admits an analytical solution that is one dimension in nature. In this numerical study, we consider the simulation of Hartmann flow in a two-dimensional domain $\Omega=(0,0.025)\times(-1,1)$. If we define the characteristic velocity as $u_0:=\sqrt{Gl_0/\rho}$ and consider the driving pressure gradient $G$ as a forcing term (incorporated in $\gb$), the nondimensionalized solution with $\gb=(1,0)$, $\fb=\bs{0}$ takes the form (see, i.e., \cite{shadid_towards_2010,shadid_scalable_2016})
\begin{subequations}\eqnlab{hartmann_2d}
\begin{align}
    \begin{split}
        \bu = \LRp{\frac{\Rey}{Ha\tanh{(Ha)}}\LRs{1-\frac{\cosh{(Ha\cdot y)}}{\cosh{(Ha)}}},0},
    \end{split}\\
    \begin{split}
        \bb = \LRp{\frac{1}{\kappa}\LRs{\frac{\sinh{(Ha\cdot y)}}{\sinh{(Ha)}}-y},1},
    \end{split}\\
    \begin{split}
        \p = -\frac{1}{2\kappa}\LRs{\frac{\sinh{(Ha\cdot y)}}{\sinh{(Ha)}}-y}^2-\p_0,
    \end{split}\\
    \begin{split}
        \r = 0
    \end{split}
\end{align}
\end{subequations}
where $Ha:=\sqrt{\kappa\Rey\Rm}$, and $\p_0$ is a constant that enables $\p$ to satisfy the zero average pressure condition. 

At refinement level $l$, the domain is divided into $l\times80l$ squares, each of which is divided into two triangles from top right to bottom left. Figure \figref{nonlin_hartmann_2d} shows the convergence plots with $\Rey=\Rm=7.07$ and $\kappa=200$ and the corresponding convergence rates are summarized in Table \tabref{nonlin_hartmann_2d}. The convergence rates for $\LbH$, $\ubH$, $\pH$, $\JbH$, $\bbH$, and $\rH$ are observed to be approximately $\k$, $\k-1/2$, $\kbr+1$, $\k$, $\k+1/2$, and $\kbr+1$. The observation is consistent with the rates observed in Section \secref{lin_smooth_2d} and \secref{lin_smooth_3d} except for the ones of the velocity and magnetic fields, which are sub-optimal here.  

\begin{table}[!htb]
\centering
\begin{tabular}{|c|cccccc|cc|}
\multicolumn{9}{c}{$\Rey=\Rm=7.07,\kappa=200$} \\
\hline
 &
\begin{tabular}{@{}c@{}} $\Rey\LbH$ \end{tabular} & 
\begin{tabular}{@{}c@{}} $\ubH$\end{tabular} & 
\begin{tabular}{@{}c@{}} $\pH$ \end{tabular} & 
\begin{tabular}{@{}c@{}} $\frac{\Rm}{\kappa}\JbH$\end{tabular} & 
\begin{tabular}{@{}c@{}} $\bbH$\end{tabular} & 
\begin{tabular}{@{}c@{}} $\rH$\end{tabular} & 
$\norm{\Div{\bu_h}}_{\infty}$ & $\norm{\Div{\bb_h}}_{\infty}$\\
\hline
$k=1$ & 1.01 &	3.68 &	1.01 &	1.03 &	1.87 &	1.26 &	3.52E-09 & 5.50E-12 \\ 
$k=2$ & 2.08 &	1.81 &	2.03 &	1.96 &	2.58 &	1.74 &	2.95E-08 & 1.28E-10 \\ 
$k=3$ & 3.20 &	2.59 &	3.16 &	3.55 &	3.64 &	3.18 &	1.30E-07 & 2.90E-10 \\ 
$k=4$ & 4.17 &	3.72 &	4.13 &	4.21 &	4.20 &	3.96 &	3.16E-07 & 7.02E-10 \\   
\hline
\end{tabular}
\caption{Convergence rates of all local variables and divergence errors of velocity and magnetic fields for the nonlinear solver applied to solve the two-dimensional Hartmann flow problem that admits the solution given in \eqnref{hartmann_2d}. The corresponding results are also presented in Figure \figref{nonlin_hartmann_2d}. In this table, the convergence rates are evaluated at the last two data sets and the divergence errors are evaluated at the last data set.}\tablab{nonlin_hartmann_2d}
\end{table}

\begin{figure*}
    \centering
    \resizebox{\textwidth}{!}{
    \begin{tabular}{ccc}
     \centered{\includegraphics{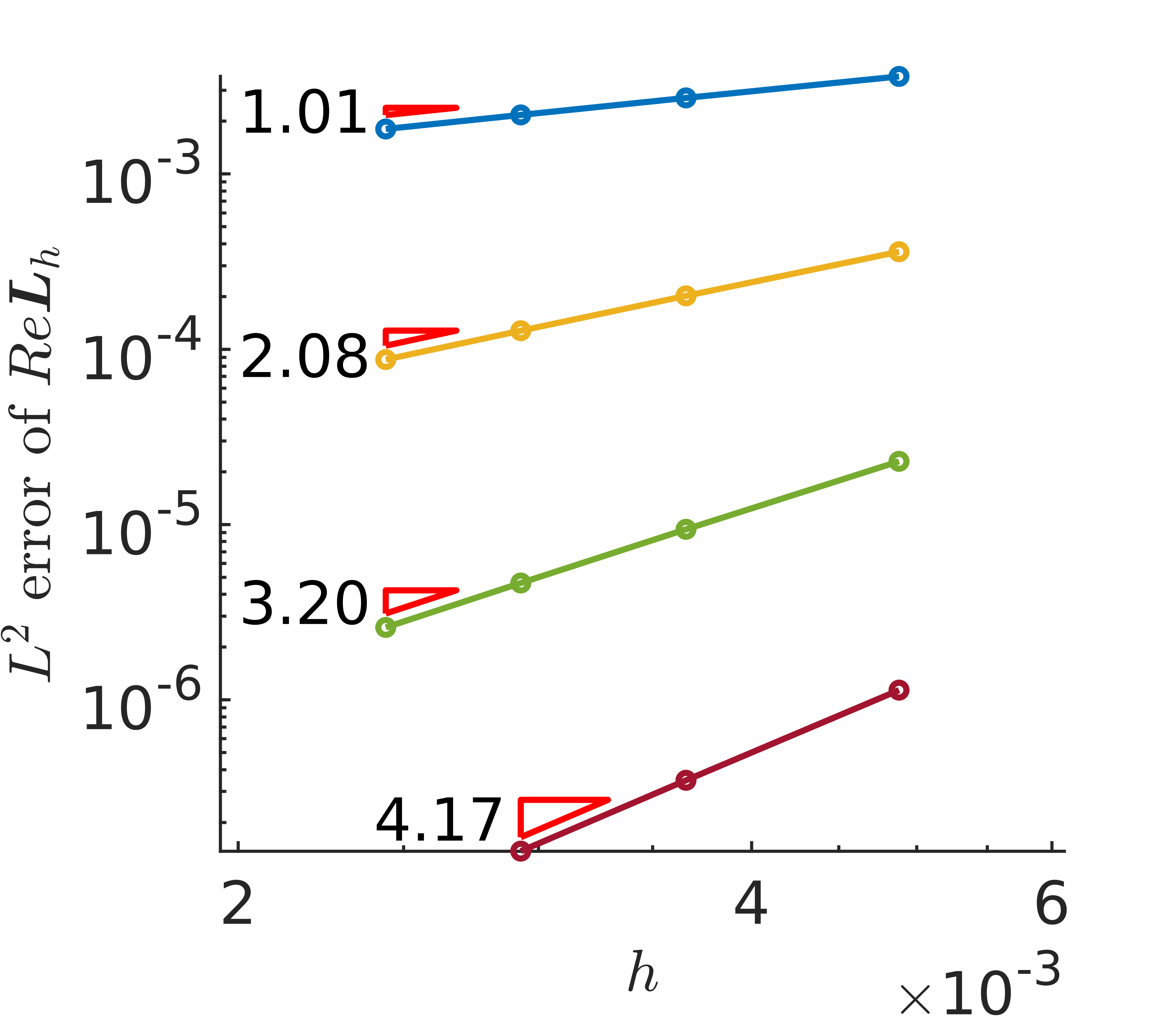} }& 
     \centered{\includegraphics{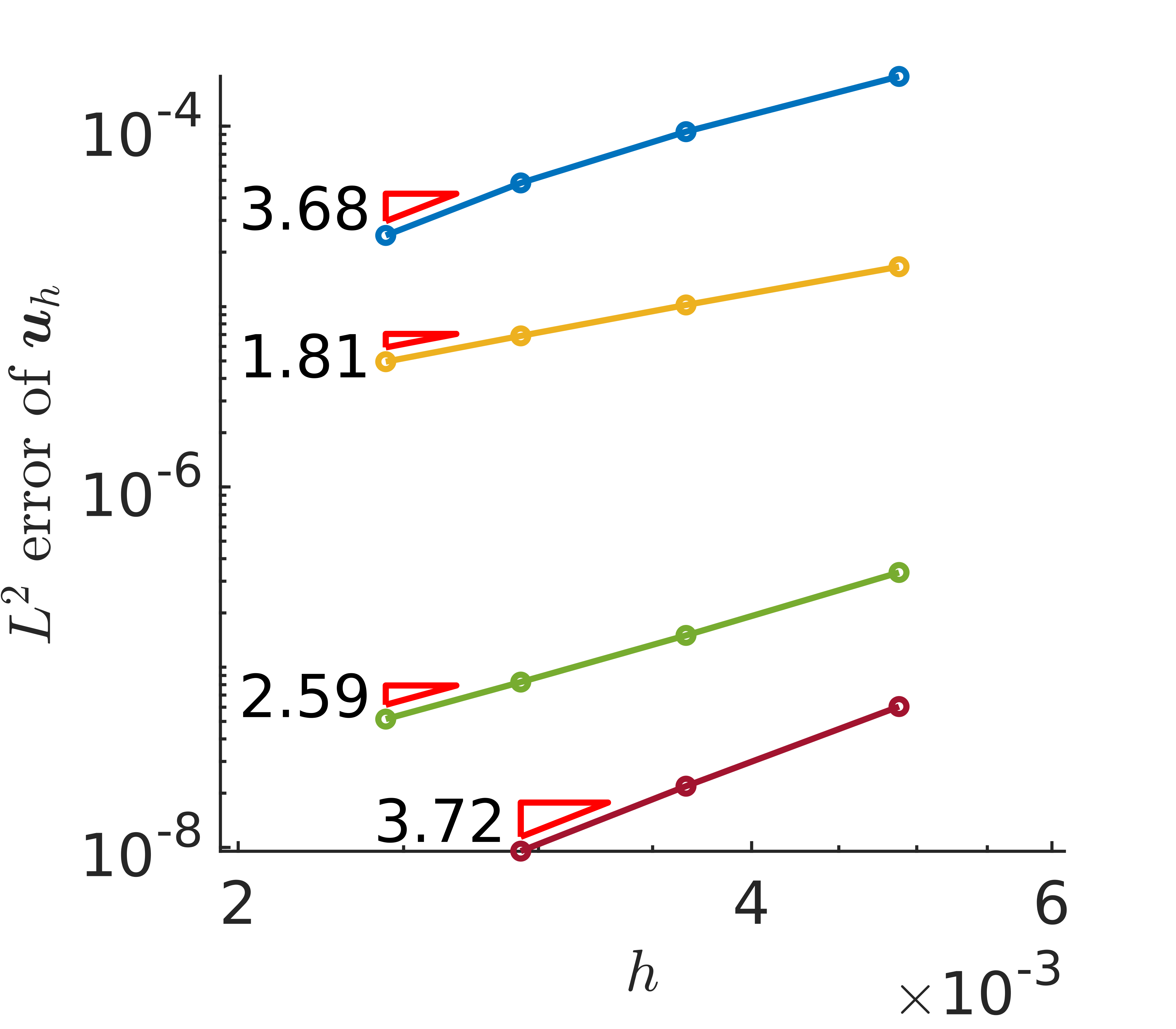}} &
     \centered{\includegraphics{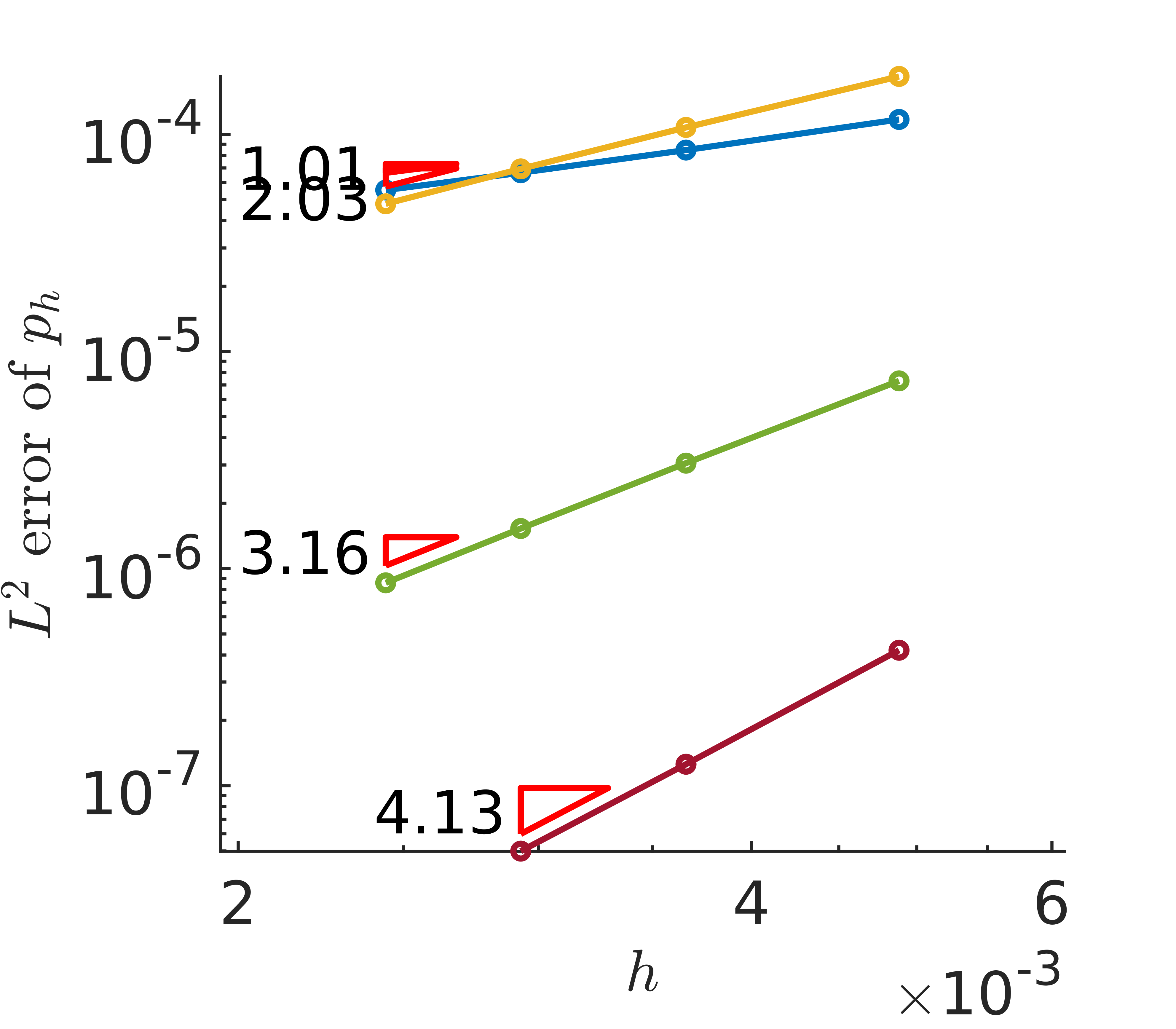}}\\
     \centered{\includegraphics{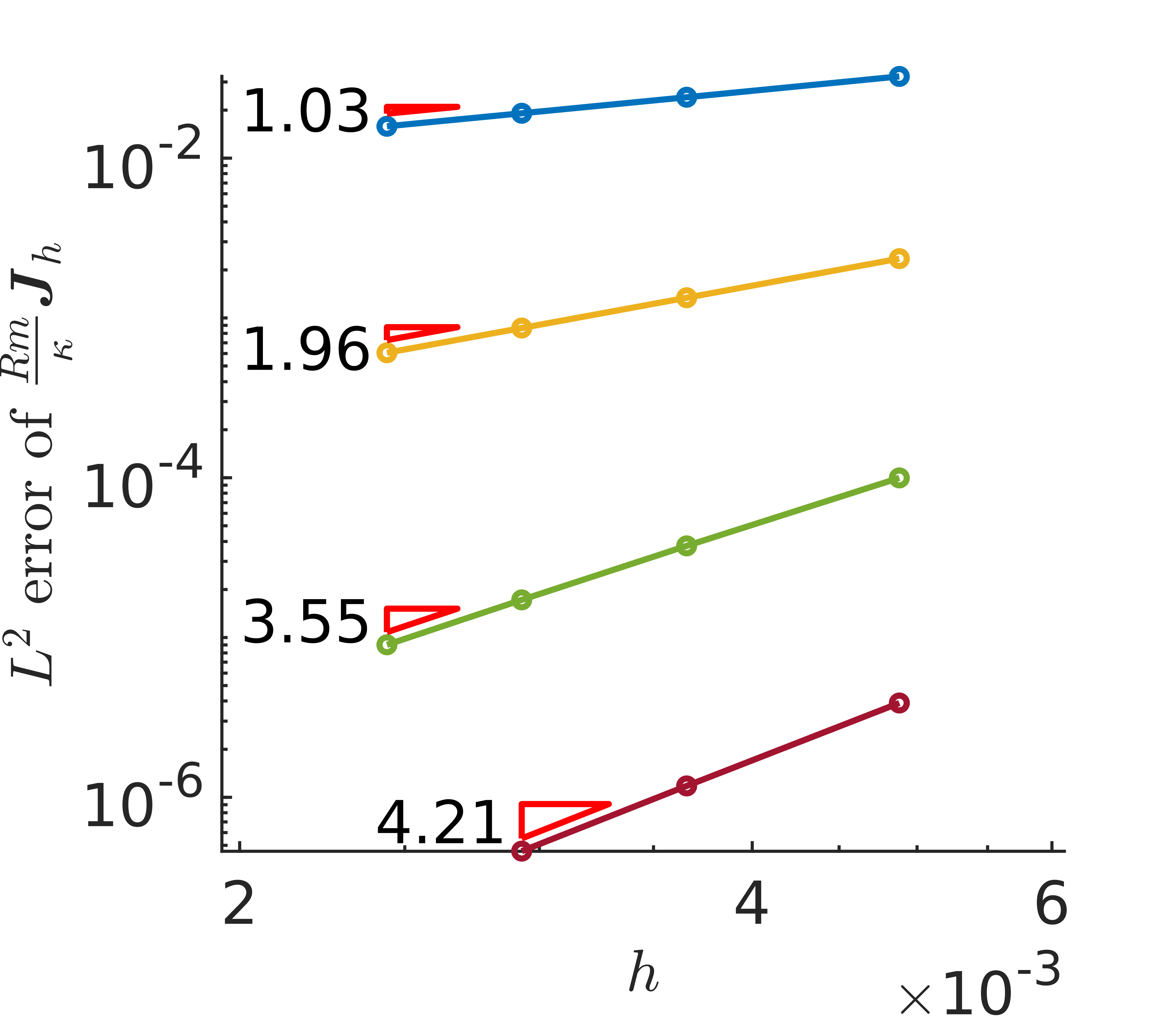}} &
     \centered{\includegraphics{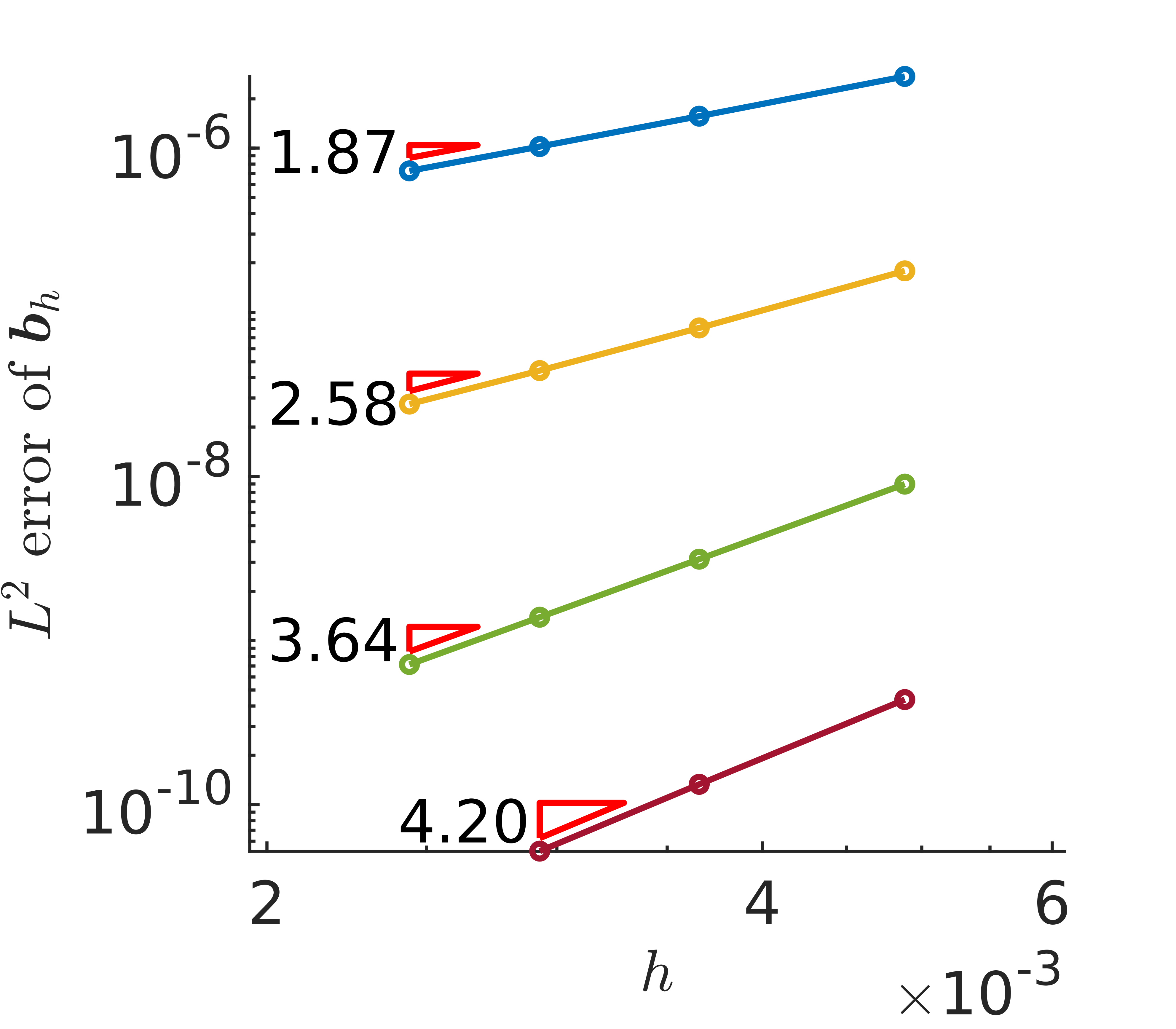}} &
     \centered{\includegraphics{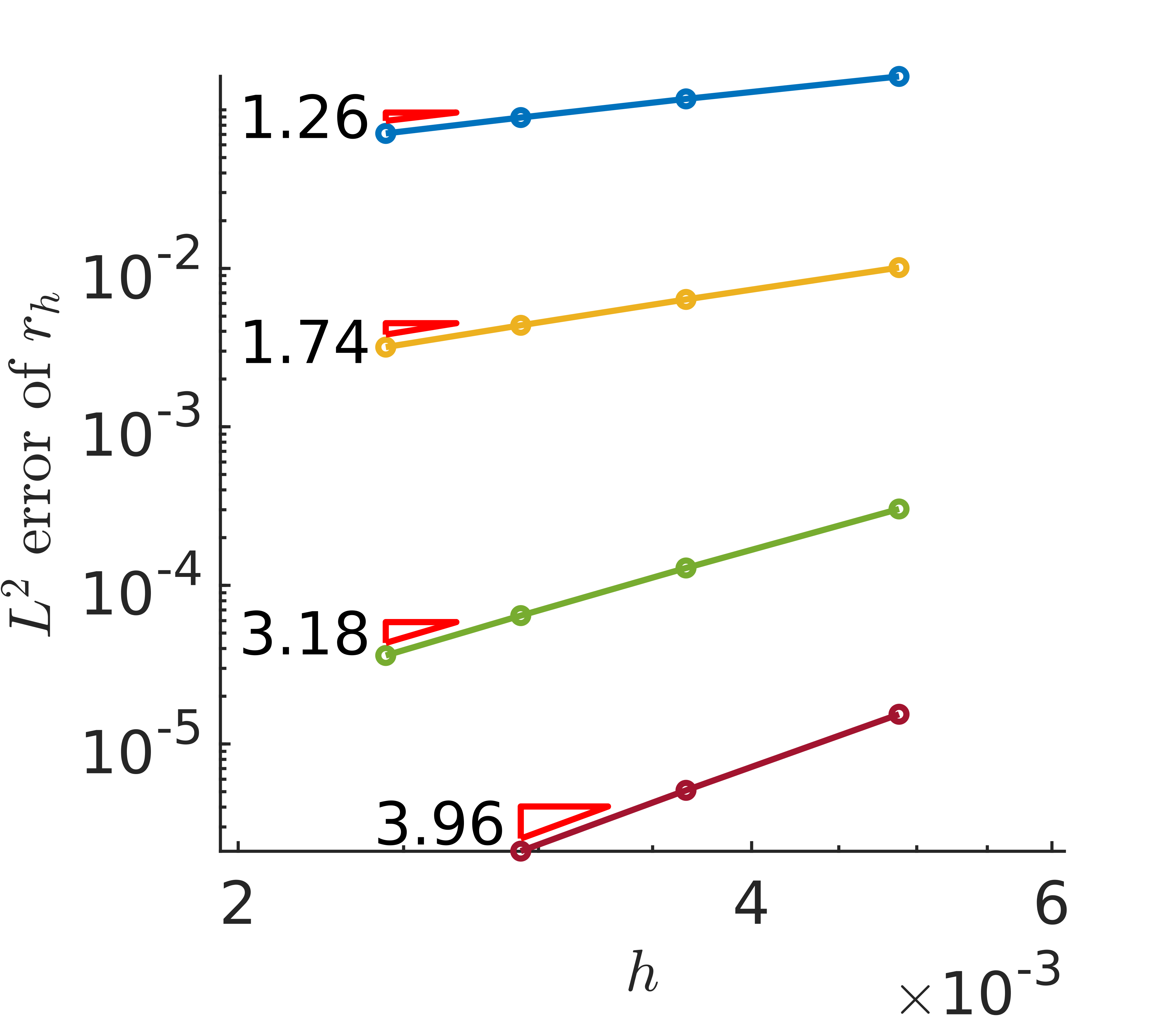}}\\
     \centered{\includegraphics{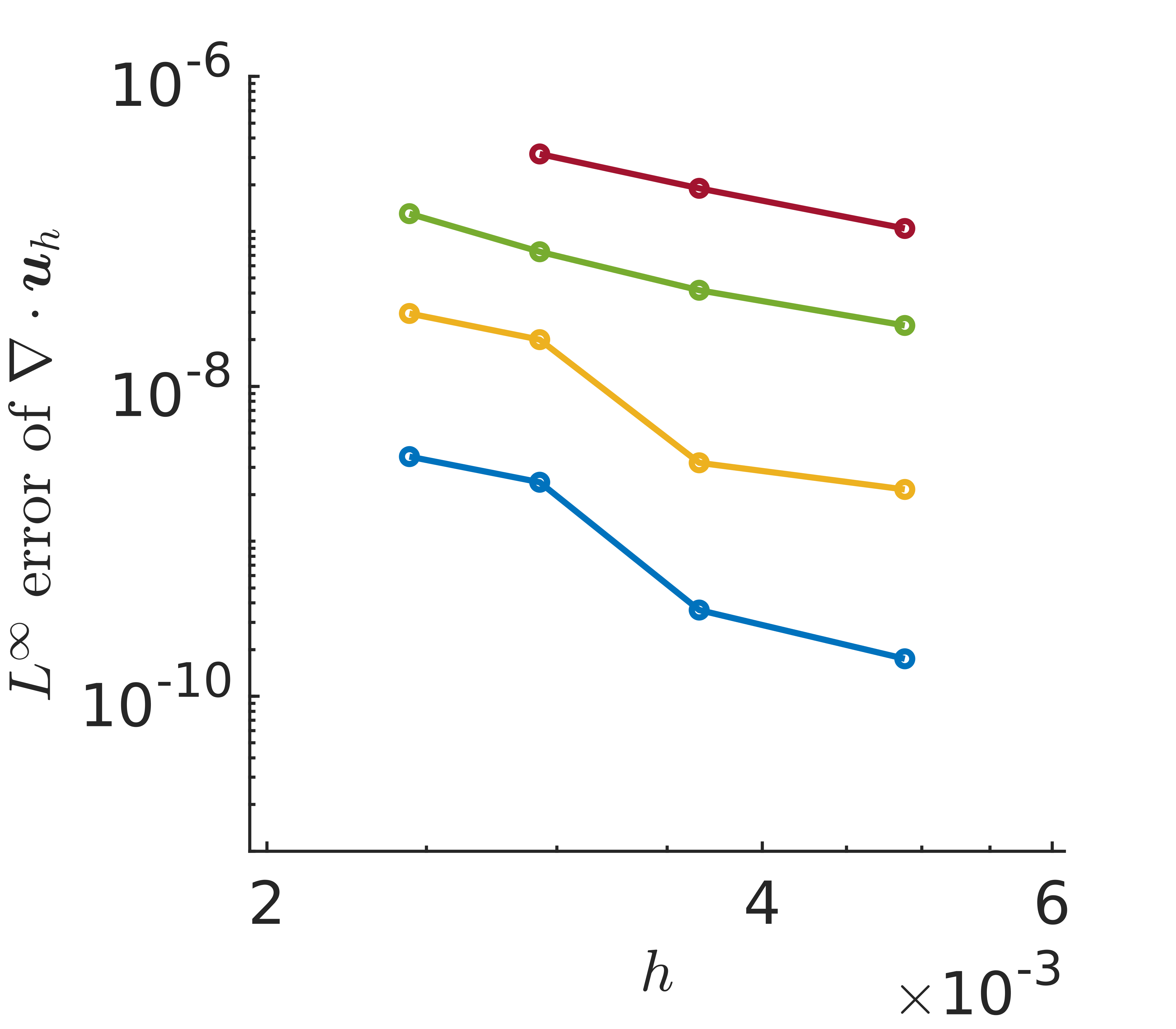}} &
     \centered{\includegraphics{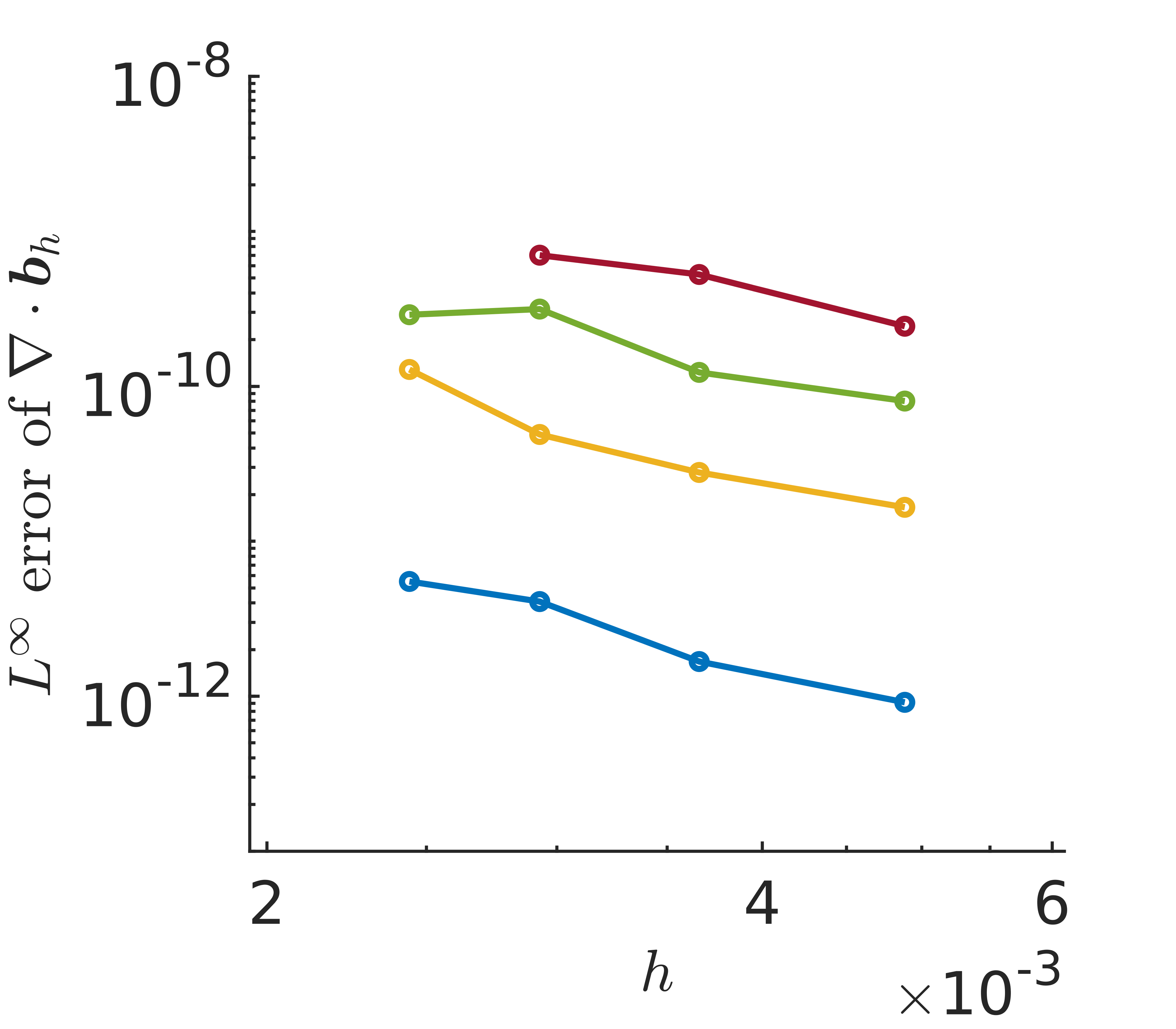}} &
     \centered{\includegraphics{conv_legend.png}}
    \end{tabular}
    }
    \caption{Convergence histories of all local variables and divergence errors for the nonlinear solver applied to solve the two-dimensional Hartmann flow problem that admits the solution given in \eqnref{hartmann_2d}.}\figlab{nonlin_hartmann_2d}
\end{figure*}

\subsubsection{Three-dimensional smooth manufactured solution}\seclab{nonlin_smooth_3d}
We now turn our attention to a three-dimensional nonlinear problem, demonstrating the convergence of the nonlinear solver utilizing a smooth manufactured solution as outlined in Section \secref{lin_smooth_3d}. The numerical results are presented in Table \tabref{nonlin_smooth_3d} and visually presented in Figure \tabref{nonlin_smooth_3d}. The observed convergence rates are consistent with the rates presented in Section \secref{nonlin_smooth_3d} where a linear problem with the same smooth manufactured solution is solved. Particularly, Table \tabref{nonlin_smooth_3d} closely mirrors the content of Table \tabref{lin_smooth_3d}. In addition, the same order of magnitude is observed for the divergence errors as well.

It is widely known that the Picard solver may not converge consistently, and the success of the iteration is contingent upon the initial guess and the contractive property. Our findings underscore that the convergence of the Picard solver is substantially influenced by the physical parameters $\Rey,\Rm$, the degree of approximation $k$, and the mesh refinement. This implies that the contractive property of the linear map $\LRp{\wb,\db}\mapsto\LRp{\ub,\bb}$ can be largely affected by these factors. This is not surprising as our analysis in \cite{muralikrishnan_multilevel_2023} showed the contraction factor is proportional to the initial guess, $\fb, \gb, \Rey, \Rm/\kappa$
and depends on $\wb$ and $\db$ in a nontrivial nonlinear manner.
Specifically, in the testing cases with $\Rey=\Rm=1000$ and $\k>1$, the Picard iteration does not converge when using the initial guess $\ubH^0=\bbH^0=\bs{0}$. Taking $\Rey=\Rm=1000$, the Picard iteration stalls when $k=2$ is used on the mesh with 364 elements, $k=3$ on the mesh with 48 elements, and $k=4$ on the mesh with 6 elements. Only the case with $k=1$ exhibits convergence across a sequence of meshes with 6, 48, 364, 3072, and 24576 elements, and the results of this case are presented in both Table \tabref{nonlin_smooth_3d} and Figure \tabref{nonlin_smooth_3d}. 

\begin{table}[!htb]
\centering
\begin{tabular}{|c|cccccc|cc|}
\multicolumn{9}{c}{$\Rey=\Rm=1,\kappa=1$} \\
\hline
 &
\begin{tabular}{@{}c@{}} $\Rey\LbH$ \end{tabular} & 
\begin{tabular}{@{}c@{}} $\ubH$\end{tabular} & 
\begin{tabular}{@{}c@{}} $\pH$ \end{tabular} & 
\begin{tabular}{@{}c@{}} $\frac{\Rm}{\kappa}\JbH$\end{tabular} & 
\begin{tabular}{@{}c@{}} $\bbH$\end{tabular} & 
\begin{tabular}{@{}c@{}} $\rH$\end{tabular} & 
$\norm{\Div{\bu_h}}_{\infty}$ & $\norm{\Div{\bb_h}}_{\infty}$\\
\hline
$k=1$ & 0.72 &	1.78 &	1.81 &	1.02 &	2.04 &	1.95 &	6.89E-13 & 2.97E-13 \\ 
$k=2$ & 2.21 &	3.50 &	2.85 &	2.21 &	3.23 &	2.78 &	1.37E-12 & 1.37E-12 \\ 
$k=3$ & 3.08 &	3.99 &	3.66 &	3.22 &	4.19 &	3.75 &	9.32E-10 & 2.07E-11 \\ 
$k=4$ & 4.22 &	5.23 &	4.73 &	4.24 &	5.23 &	4.70 &	3.47E-09 & 8.38E-11 \\ 
\hline
\multicolumn{9}{c}{$\Rey=\Rm=1000,\kappa=1$} \\
\hline
$k=1$ & 0.47 &	1.36 &	1.89 &	0.57 &	0.80 &	1.96 &	6.08E-13 & 2.47E-13 \\ 
\hline
\end{tabular}
\caption{Convergence rates of all local variables and divergence errors of velocity and magnetic fields for the Picard iterations applied to solve the three-dimensional problem with a smooth manufactured solution given in \eqnref{smooth_3d} where we set $\p_0=1$. The corresponding results are also presented in Figure \figref{nonlin_smooth_3d}. In this table, the convergence rates are evaluated at the last two data sets and the divergence errors are evaluated at the last data set.}\tablab{nonlin_smooth_3d}
\end{table}

\begin{figure*}
    \centering
    \resizebox{\textwidth}{!}{
    \begin{tabular}{@{}c@{}@{}c@{}@{}c@{}}
     \centered{\includegraphics{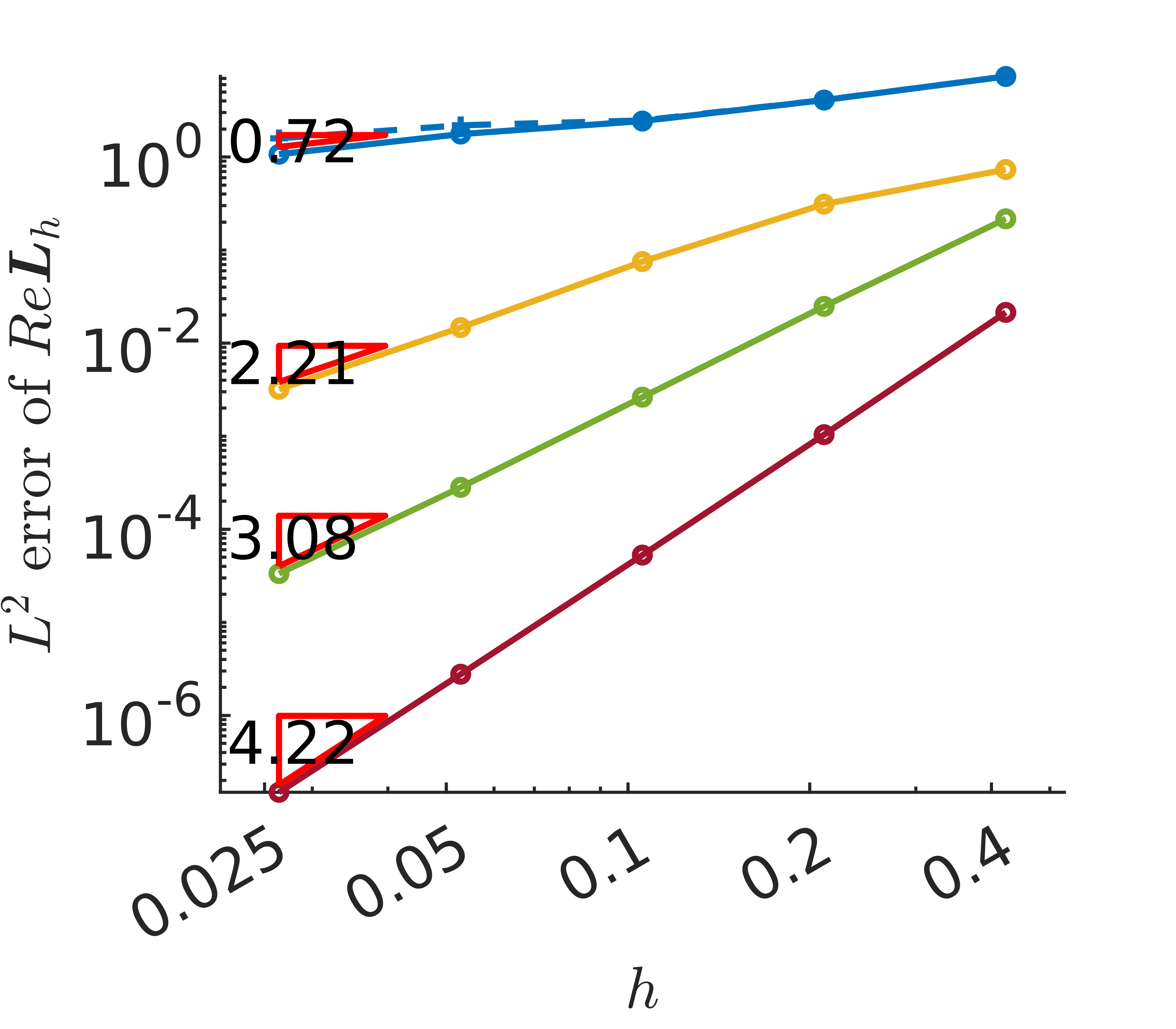} }& 
     \centered{\includegraphics{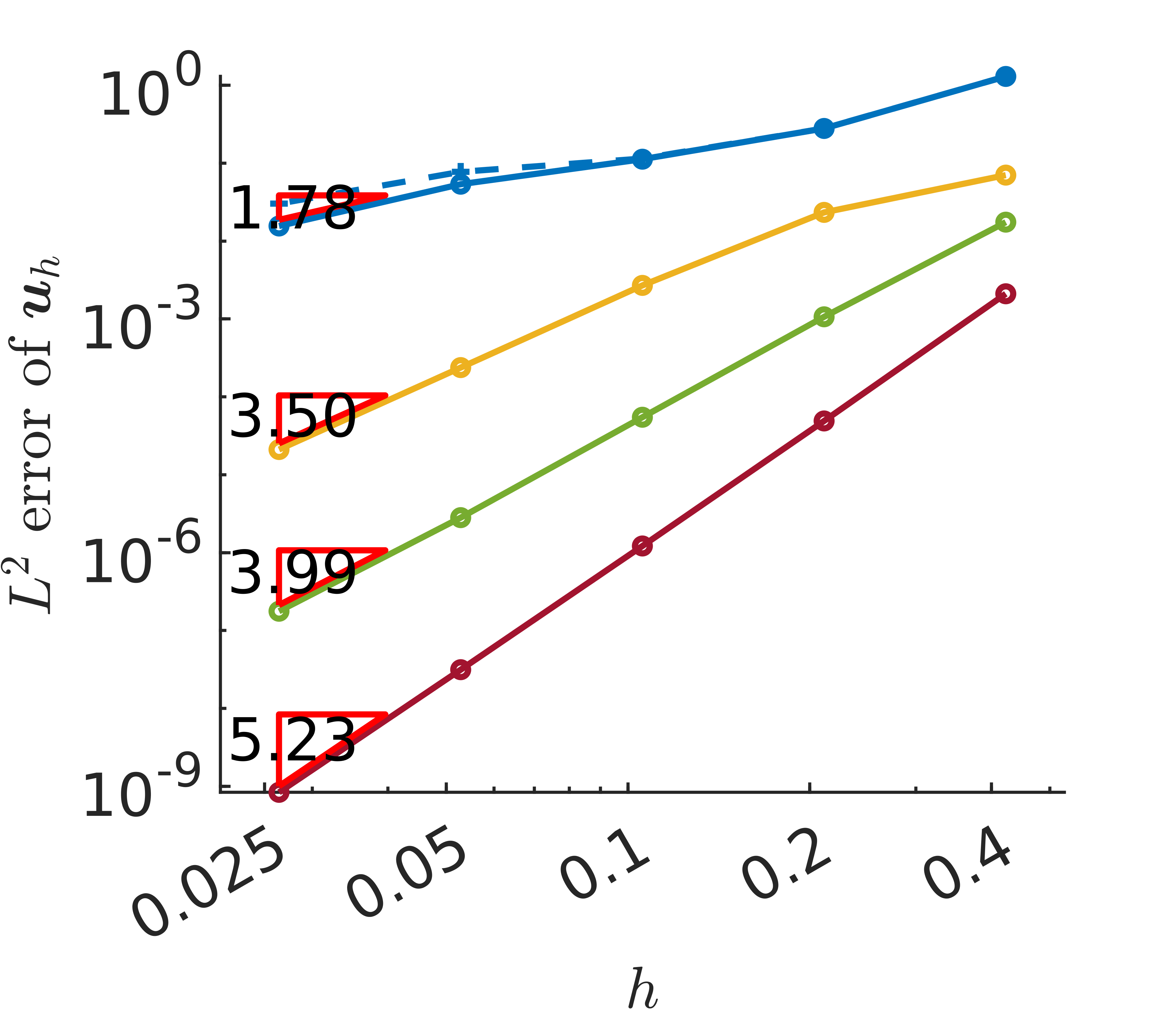}} &
     \centered{\includegraphics{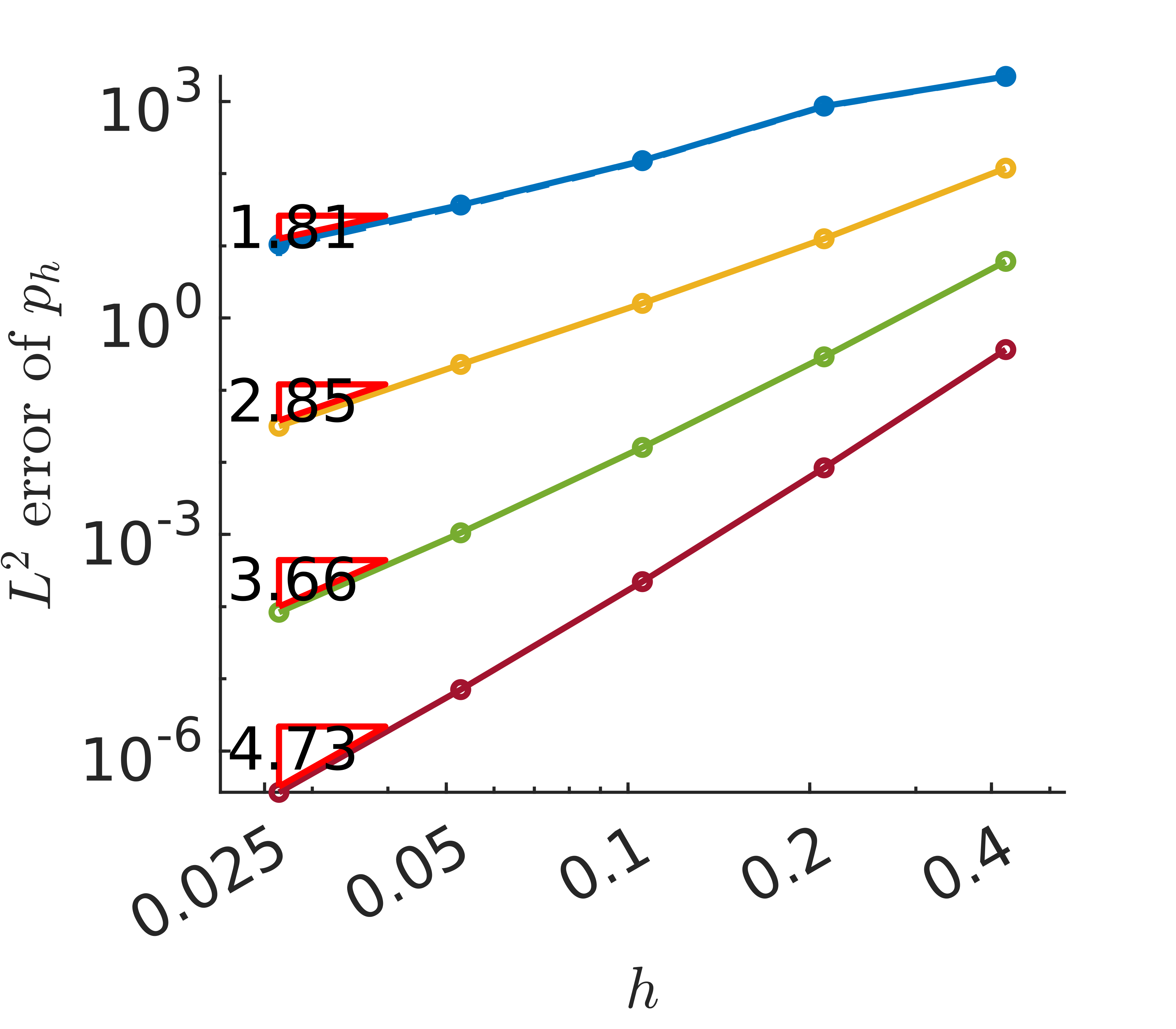}}\\
     \centered{\includegraphics{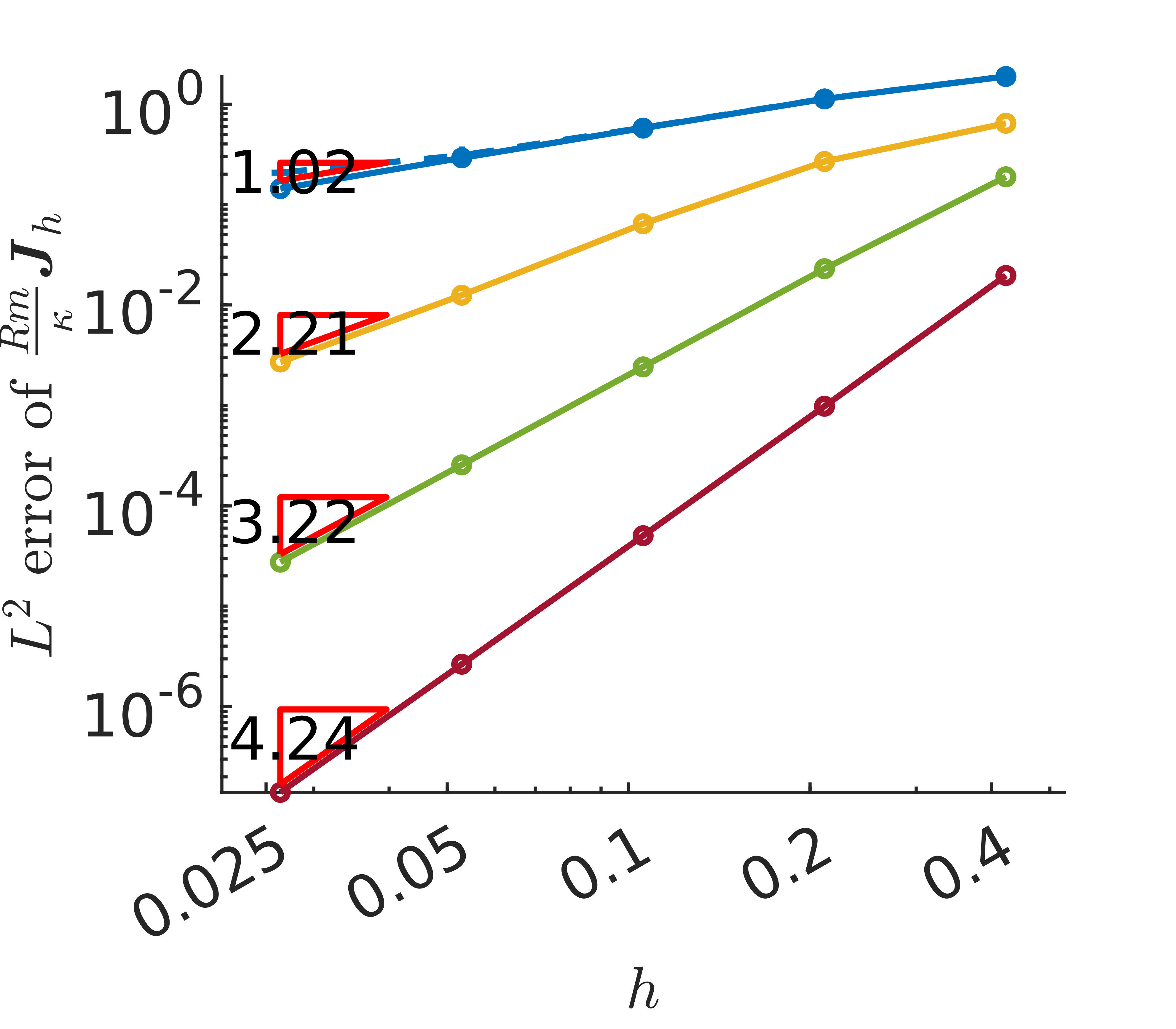}} &
     \centered{\includegraphics{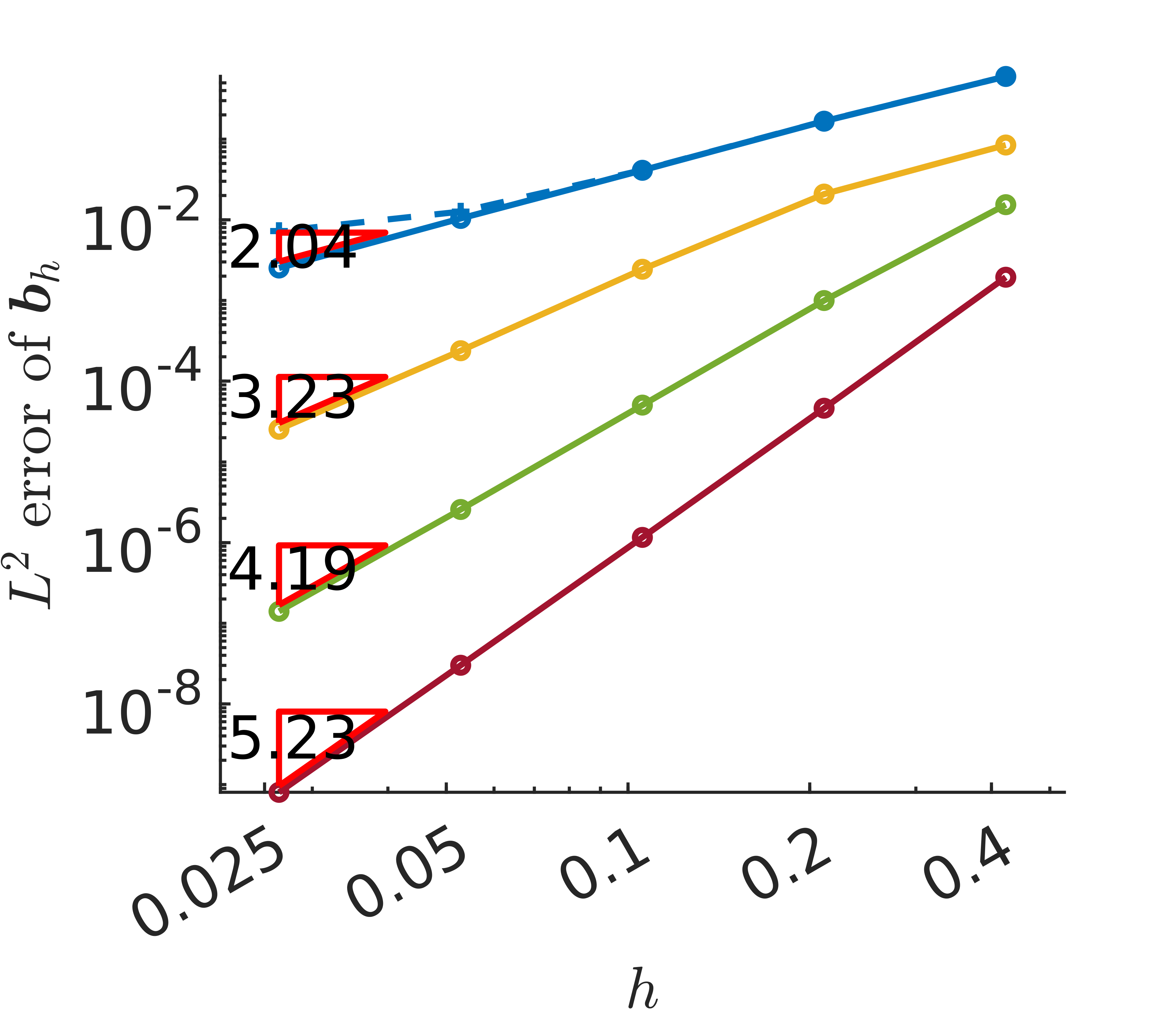}} &
     \centered{\includegraphics{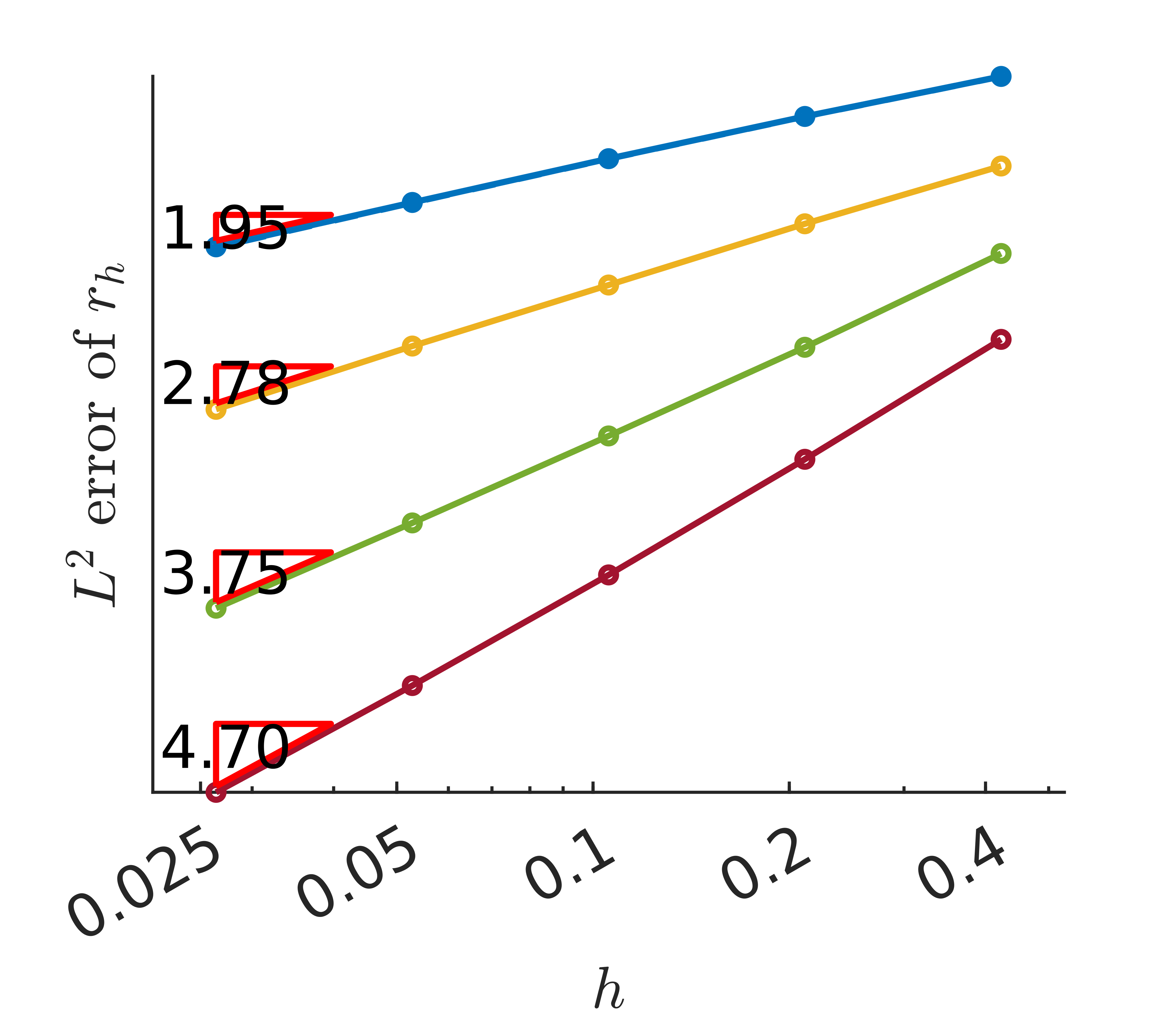}}\\
     \centered{\includegraphics{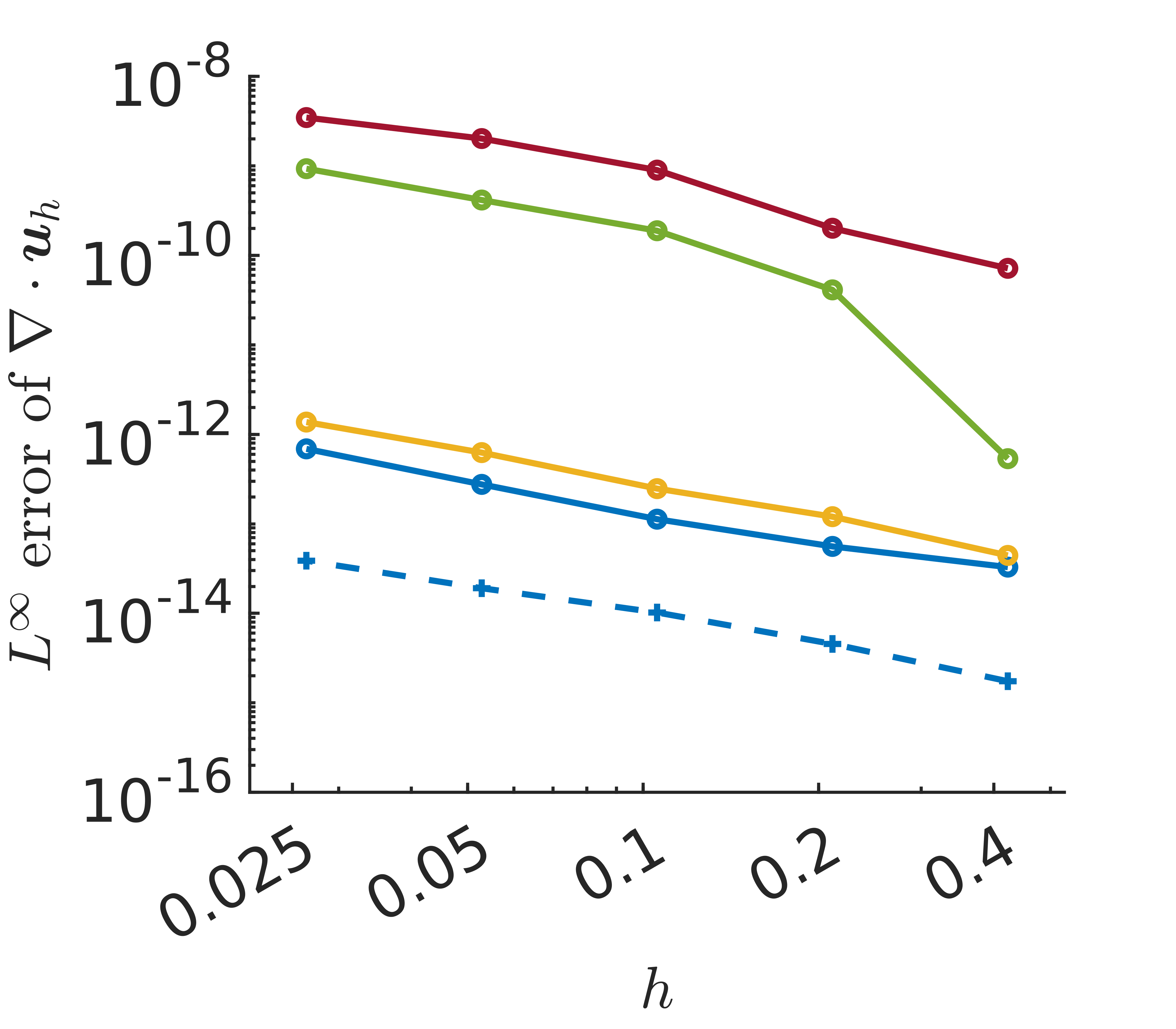}} &
     \centered{\includegraphics{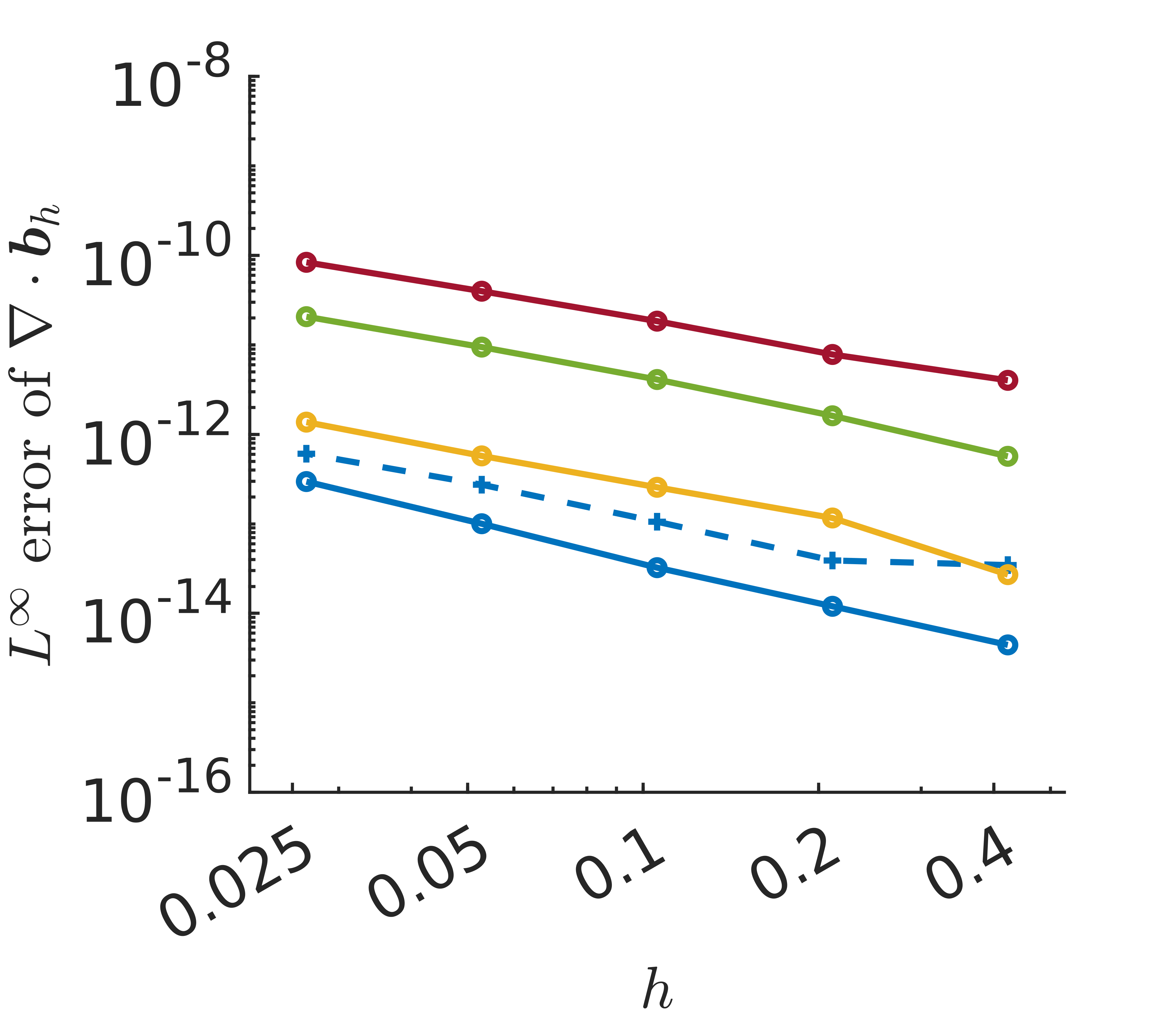}} &
     \centered{\includegraphics{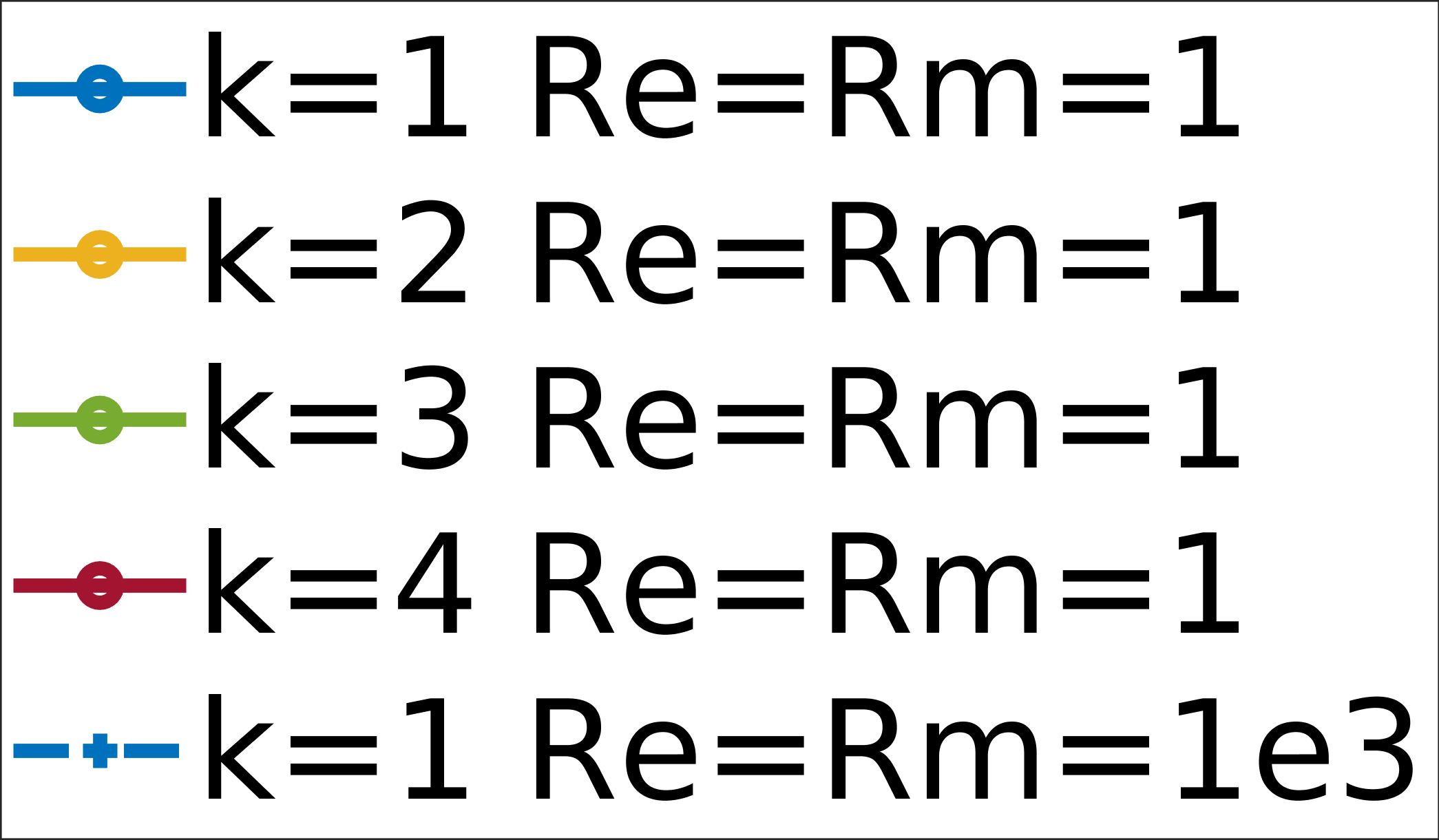}   }  
    \end{tabular}
    }
    \caption{Convergence histories of all local variables and divergence errors for the Picard iterations applied to solve the three-dimensional problem with a smooth manufactured solution given in \eqnref{smooth_3d} where we set $\p_0=1$. Only the convergence rates for $\Rey=\Rm=1$ are presented here.}\figlab{nonlin_smooth_3d}
\end{figure*}
\FloatBarrier
\section{Conclusion and future work}\seclab{conclusion}

This paper presents two new divergence-free and $\Hsp{}(div)$-conforming  HDG methods for the linearized incompressible viso-resistive MHD equations with well-posedness analysis. Particularly, we have showed that on simplicial meshes, the well-posedness of the proposed approaches can be established by the use of a one-order lower approximation in local variables for the pressure $\pH$ and the Lagrange multiplier $\rH$, and by appropriately chosen stabilization parameters. 
One of the motivations for adopting E-HDG in lieu of HDG methods lies in computational gain. Indeed, our experiments has revealed a significant acceleration in the runtime, manifested through the utilization of fewer DOFs in E-HDG, particularly in cases where the linear solver dominates the overall computational time, such as in three dimensions with high-order approximations on fine meshes. Linear problems with both smooth and singular solutions were presented to examine the convergence of the proposed E-HDG method. For problems with smooth solutions, both two- and three-dimensional settings were tested. The numerical convergence rates are shown to be optimal for both velocity and magnetic fields in the regime of low Reynolds number and magnetic Reynolds number. Moreover, the pressure robustness of our method was numerically verified. For the singular solution, the convergence rate is limited by the regularity of the solution. However, the divergence-free property is still guaranteed.

By incorporating the E-HDG discretization into the fixed point Picard iteration, we can solve the nonlinear incompressible viso-resistive MHD equations iteratively. {The globally divergence-free property still holds for both the velocity and the magnetic fields.} The convergence of the nonlinear solver is investigated through nonlinear problems with smooth solutions. The convergence rates in the tests are almost identical to the ones observed in the linear tests in both two- and three-dimensional settings. Further, divergence errors in both velocity and magnetic fields are indeed observed to be machine zero. 

While various aspects of our proposed E-HDG method have been discussed in this paper, there remain several noteworthy issues. Firstly, a rigorous convergence analysis is required, albeit consistent convergence rates for each local variable are observed in numerous numerical experiments in this paper. Secondly, the analysis presented in Section \secref{performance} may offer an incomplete depiction of the correlation between DOFs and computational time. This limitation arises from the potential inapplicability of the discussed insights to iterative solvers, which are heavily relied upon to address large-scale problems. {Therefore, the development of a scalable iterative approach that demonstrates efficacy across a wide spectrum of Reynolds and magnetic Reynolds numbers is necessary.} Finally, it is found that the Picard solver does not converge in some cases on three-dimensional meshes in the regime of high Reynolds number and magnetic Reynolds number. The observation implies that the linear map $\LRp{\wb,\db}\mapsto\LRp{\ub,\bb}$ can be largely affected by various factors. Investigating the contraction of this map could provide insights for devising a more robust algorithm. These topics are non-trivial and could each be expanded into individual papers. Thus, we aim to address them in our future research agenda.

\section*{Acknowledgments}
This research is partially funded by the National Science Foundation awards NSF-OAC-2212442, NSF-2108320, NSF-1808576 and NSF-CAREER-1845799; by the Department of Energy award DE-SC0018147 and DE-SC0022211; by the JTO Faculty Fellowship Research Program.

\appendix

\bibliographystyle{siamplain}
\bibliography{EHDG_for_MHD,references}
\end{document}